\documentclass[11pt]{extarticle}

\usepackage[english]{babel}
\usepackage[utf8]{inputenc}
\usepackage[OT1]{fontenc}
\usepackage[margin=1in]{geometry}
\usepackage{amsfonts,amsmath,amsthm,amssymb,comment,enumitem,hyperref,longtable,natbib,graphicx,footnote,xcolor,authblk,algorithm,algpseudocode,epstopdf,bbm,float,pdflscape,mathtools,setspace,rotating,tablefootnote,tabularx,xr}
\usepackage[title]{appendix}
\usepackage[singlelinecheck=off]{caption}
\usepackage[caption=false]{subfig}

\newtheorem{theorem}{Theorem}

\newtheorem{proposition}{Proposition}
\newtheorem{lemma}{Lemma}
\newtheorem{assumption}{Assumption}
\theoremstyle{definition}
\newtheorem{definition}{Definition}
\newtheorem{example}{Example}
\theoremstyle{remark}
\newtheorem{remark}{Remark}

\title{\vspace{-1cm}\textbf{Robust Optimization of Rank-Dependent Models\\ with Uncertain Probabilities}\thanks{We are very grateful to Ronny Ben-Tal, Henry Lam and conference and seminar participants at 
EURANDOM, 
the 2023 SIAM Conference on Optimization, 
the Tinbergen Institute,
the University of Amsterdam and 
the University of Vienna for their comments and suggestions.
{\tt Python} code to implement the optimization procedures developed in this paper is available from \url{https://github.com/GuanJinNL/ROptRDU.github.io.git}.
This research was funded in part by the Netherlands Organization for Scientific Research under grant NWO VICI 2020--2027 (Jin, Laeven).
\textit{Email addresses}: \href{mailto:G.Jin@uva.nl}{\tt G.Jin@uva.nl}, \href{mailto:R.J.A.Laeven@uva.nl}{\tt R.J.A.Laeven@uva.nl}, and \href{mailto:D.denHertog@uva.nl}{\tt D.denHertog@uva.nl}.}}
\author[a]{Guanyu Jin}
\author[a,c,d]{Roger J.~A.~Laeven}
\author[b]{Dick den Hertog}
\affil[a]{{\small Dept.~of Quantitative Economics, University of Amsterdam,\protect\\ 
1001 NJ Amsterdam, The Netherlands}}
\affil[b]{{\small Dept.~of Business Analytics, University of Amsterdam,\protect\\  1001 NJ Amsterdam, The Netherlands}}
\affil[c]{{\small EURANDOM, 5600 MB Eindhoven, The Netherlands}}
\affil[d]{{\small CentER, Tilburg University, 5000 LE Tilburg, The Netherlands}}

\begin{document}
\maketitle

\begin{abstract}
This paper studies distributionally robust optimization for a rich class of risk measures with ambiguity sets defined by $\phi$-divergences.
The risk measures are allowed to be non-linear in probabilities, are represented by Choquet integrals possibly induced by a probability weighting function, and encompass many well-known examples. 
Optimization for this class of risk measures 
is challenging due to their rank-dependent nature. 
We show that for various shapes of probability weighting functions, including concave, convex and inverse $S$-shaped, the robust optimization problem can be reformulated into a rank-independent problem. 
In the case of a concave probability weighting function, the problem can be reformulated further into a convex optimization problem that admits explicit conic representability for a collection of canonical examples. 
While the number of constraints in general scales exponentially with the dimension of the state space, we circumvent this dimensionality curse and develop two types of algorithms.
They yield tight upper and lower bounds on the exact optimal value and are formally shown to converge asymptotically. 
This is illustrated numerically in a robust newsvendor problem and a robust portfolio choice problem. 
\end{abstract}

\noindent \textit{OR/MS Classification}: Programming: Stochastic.
Decision analysis: Risk.\\[2mm] 


%


\onehalfspacing

\section{Introduction}
Many stochastic optimization problems in management science, operations research, engineering, economics, and finance arise from decisions involving risk (probabilities given) and ambiguity (probabilities unknown). 

A variety of models for decision under risk have been proposed.
Among the most popular and empirically viable models is the rank-dependent utility (RDU) model of \citet{Quiggin1982}.
In RDU, the utility loss associated to a random variable $X$ under the probabilistic model $\mathbb{P}$ is measured by a rank-dependent evaluation $\rho$ with respect to a non-additive, distorted measure $h\circ\mathbb{P}$:\footnote{See also \citet{Schmeidler1986, Schmeidler1989} for related pioneering developments.}
\begin{align*}
    \rho_{u,h,\mathbb{P}}(X)\triangleq\int -u(X)\mathrm{d}(h\circ \mathbb{P}),
\end{align*}
where $u:\mathbb{R}\rightarrow\mathbb{R}$ is a utility function, assumed to be non-decreasing, and $h:[0,1]\rightarrow[0,1]$, with $h(0)=1-h(1)=0$ and non-decreasing, is a distortion or probability weighting function. 
RDU serves as a pivotal building block of prospect theory \citep{TverskyKahneman1992}
and encompasses expected utility \citep{NeunmannMorgen} when $h$ is linear and the dual theory of \citet{Yaari1987} when $u$ is affine.
It accommodates \cite{Allais1953} type phenomena that are incompatible with expected utility.
The utility function captures attitude toward wealth and the shape of the distortion function, e.g., concave, convex or (inverse) $S$-shaped, dictates attitude toward risk.
Importantly, under RDU probabilities of outcomes are weighted according to $h(F_X(x))\triangleq h(\mathbb{P}(X\leq x))$, leading to an evaluation that is non-linear in probabilities, which depend on the ranking of outcomes.
These non-linear and rank-dependent features bring major computational challenges. 

In the aforementioned theories for decision under risk, $\mathbb{P}$ is assumed to be given.
When ambiguity is present, $\mathbb{P}$ is unknown.
Ambiguity is often treated via a worst-case approach that is robust against malevolent nature.
For example, \citet{MaxminEU} introduce maxmin expected utility, a.k.a.~multiple priors, under which risks are evaluated according to their worst-case expected utility taken over a set of probabilistic models.
\citet{HansenSargent01,HansenSargent07} propose multiplier preferences under which probabilistic models far away from a reference model are penalized according to the Kullback-Leibler divergence (a.k.a.~relative entropy).
The variational preferences of \citet{Maccheroni} admit a general penalty function, thus significantly generalizing the maxmin and multiplier preferences models.
\citet{dual_theory} develop a rank-dependent theory for decision under risk and ambiguity that encompasses both the dual and rank-dependent counterparts of \citet{MaxminEU} and \citet{Maccheroni}. 

Parallel to these developments, the field of (distributionally) robust optimization has been studying risk and ambiguity from a computational perspective. 
In robust optimization, ambiguity sets are often constructed exogenously from data, as a confidence region of the true underlying distribution, rather than endogenously based on preferences. 
A widely used family of statistical estimators for distributional uncertainty is given by $\phi$-divergences \citep{Csiszar1975,NCE,CE_phi_div,Pardo06}.
Many distributionally robust optimization problems with these types of ambiguity sets can be reformulated into tractable robust counterparts, which can then be solved efficiently using standard optimization algorithms (see e.g., \citealp{BenTal11}, \citealp{KuhnWieseman}, and \citealp{EsfahaniKuhn}).

Although a variety of distributionally robust optimization techniques have been developed for optimization under uncertainty, most of the literature is concerned with stochastic models in which the probabilities appear linearly in the optimization problems, such as (possibly generalized) variants of expected value or expected utility maximization.
Despite the growing interest in and applications of rank-dependent models across a wide variety of fields (see e.g., \citealp{Denneberg}, 
\citealp{Wakker}, 
\citealp{Follmer2011},
and the references therein), optimization of these models---whether ambiguity is present or not---is still relatively underdeveloped. 
The main difficulties lie in both the non-linearity in probabilities and the rank-dependence: for each value of a decision vector, the rank-dependent evaluation of an uncertain objective or constraint can be permuted, since the ranking of the outcomes depends upon the decision vector.

In this paper, we develop an efficient approach for optimizing rank-dependent models, with and without uncertain probabilities. 
More precisely, we study the following nominal and robust minimization problems, in a discrete probability space setting:
\begin{align}\label{P-Nom-intro}
     \inf_{\mathbf{a}\in \mathcal{A}}\rho_{u,h,\mathbf{p}}(f(\mathbf{a},\mathbf{X})), \\
\label{P-intro}
    \inf_{\mathbf{a}\in \mathcal{A}}\sup_{\mathbf{q}\in\mathcal{D}_{\phi}(\mathbf{p},r)}\rho_{u,h,\mathbf{q}}(f(\mathbf{a},\mathbf{X})) ,
\end{align}
where $\mathbf{a}\in \mathbb{R}^{n_a}$ is the decision vector, $\mathcal{A}$ is a set of constraints, $\mathbf{X}\in \mathbb{R}^l$ is a random vector, $f:\mathbb{R}^{n_a}\times \mathbb{R}^l\to \mathbb{R}$ is a deterministic function, $\mathcal{D}_{\phi}(\mathbf{p},r)$ is a $\phi$-divergence ambiguity set (formally defined in \eqref{defDivset}), and $\rho_{u,h,\mathbf{q}}(.)$ is a rank-dependent evaluation (formally defined in \eqref{ranked_sum}), for some probability vectors $\mathbf{p},\mathbf{q}\in \mathbb{R}^m$.

Our main contributions can be summarized as follows:
\begin{itemize}
    \item We show that 
    we can reformulate problems \eqref{P-Nom-intro} and \eqref{P-intro}, as well as their epigraph formulations, into rank-independent, tractable (robust) counterparts,
    for the various empirically relevant shapes of the probability weighting function $h$, including concave, convex, and inverse $S$-shaped functions. 
    \item For concave probability weighting functions, we show that the reformulated robust counterpart admits a conic representation, provided $h$ and $\phi$ are conic representable functions. 
    For a list of canonical examples of $h$ and $\phi$, we provide explicit epigraph representations of the reformulated robust counterpart that can directly be implemented into standard conic optimization programs such as CVXPY.
    \item While the reformulated robust counterparts are more tractable than \eqref{P-Nom-intro} and \eqref{P-intro}, their multiplicity of constraints increases exponentially in the dimension of the underlying discrete probability space.  
    We provide two types of algorithms to circumvent this curse of dimensionality. 
    Each of them, or a combination thereof in the case of convex and inverse $S$-shaped probability weighting functions, yields tight upper and lower bounds that, as we formally establish, converge to the optimal objective value of the exact problem. 
    Moreover, we show that one of our algorithms can also be applied to a more general type of rank-dependent model, namely Choquet expected utility (CEU,  \citealp{Schmeidler1989}).\footnote{Under CEU, the non-additive measure need not be obtained by distorting an additive probability measure.}
    \item We provide numerical examples of our approach in applications of robust optimization with rank-dependent models and uncertain probabilities, in the context of inventory planning and portfolio optimization. 
    We efficiently obtain optimal decisions that are robust against uncertainty, and yield accurate optimal objective values. 
    Concrete examples with codes are provided on \url{https://github.com/GuanJinNL/ROptRDU.github.io.git}.
\end{itemize}

To our best knowledge, our approach is the first that can handle nominal and distributionally robust, generic optimization problems involving rank-dependent evaluations with possibly inverse $S$-shaped probability weighting functions.

\subsection{Related Literature}\label{sec: related works}
Our work builds on the decision-theoretic literature on evaluating risk and ambiguity.
Specifically, we consider in \eqref{P-Nom-intro}--\eqref{P-intro} the rank-dependent evaluations of \citet{Quiggin1982} and adopt a robust approach as in \citet{MaxminEU},  \citet{Maccheroni}, \citet{HansenSargent01,HansenSargent07} and \citet{dual_theory}, which can be viewed as (generalized) decision-theoretic foundations of the classical decision rule of \citet{Wald1950}; see also \citet{Huber1981}.
We contribute to this literature by developing corresponding optimization techniques.

An initial connection between robust risk measures and robust optimization has been studied in an early paper by \citet{Bertsimas09}. 
They show that a decision-maker's preference, as represented by a coherent risk measure, can provide a device for constructing an uncertainty set for robust optimization purposes. 
The authors were able to obtain tractability results for distortion risk measures under a specific parameterization and the (strict) assumption that the underlying discrete probability space is uniform. 
Our paper relaxes these assumptions, while also providing a blueprint for constructing uncertainty sets tailored to general risk and ambiguity preferences. 
\citet{SiamPaper} studied distributionally robust optimization (DRO) problems for a broad collection of uncertainty sets and risk measures. 
However, for classes of risk measures that are non-linear in probabilities, such as distortion risk measures, the tractability remained unknown (see their Table~I); they are covered as special cases in this paper. 
Although different choices of uncertainty sets are possible, our paper focuses on uncertainty sets defined by $\phi$-divergences, which constitute a family of divergences including the Kullback-Leibler divergence, Burg entropy and Hellinger distance. 
The study of $\phi$-divergences in DRO is motivated by several earlier studies; see e.g., \citet{NCE,CE_phi_div} and \cite{BenTal11}. 

Recently, \citet{clmpaper} and \citet{ConcentrationPaper} studied DRO problems involving distortion risk measures with general non-concave distortion functions. 
Their interesting results show that the DRO problem associated with a non-concave distortion function is equivalent to that with its concave envelope approximation provided that the underlying uncertainty set obeys certain moment conditions. 
In this paper, we show that this equivalence is typically not satisfied in our general setting, thus requiring novel techniques. 
Optimization of non-expected utility models with possibly non-concave distortion functions has also been studied in innovative work of \citet{Quantile_Method}, in particular in the context of portfolio choice, introducing the so-called quantile method; see also \cite{carlier2006law}. 
The effectiveness of this method relies on the ability to parametrize the distributions of all random objective functions $\{f(\mathbf{a},\mathbf{X}), \mathbf{a}\in \mathcal{A}\}$ by a set of quantile functions 
that satisfy finitely many constraints. 
For a general convex function $f$ and feasibility set $\mathcal{A}$ that are considered in \eqref{P-Nom-intro}, there is, however, no systematic approach to obtain such a quantile formulation. 
Our work provides a method to optimize (both robust and nominal) rank-dependent models for a broad class of decision problems and can directly be implemented in standard optimization software. 
Finally, \citet{Delage2022} and \citet{WangXu2021} 
analyze `preference robust' optimization, which considers uncertainty in the distortion function rather than in the underlying probabilistic model.

The remainder of the paper is organized as follows.
Section~\ref{sec: Prelim} introduces the setting and notation. 
Sections~\ref{sec: RC} to \ref{sec:SolvingAlgorithms} study the optimization problems~\eqref{P-Nom-intro} and~\eqref{P-intro} for concave distortion functions.
The techniques we develop are extended in Section~\ref{sec:NonConcave} to non-concave, inverse $S$-shaped distortion functions.
Section~\ref{sec: NumericalExp} presents our numerical experiments. 
Conclusions are in Section~\ref{sec: Conclusion}. 
In Electronic Companions~\ref{sec:tab}--\ref{App:HitandRun}, we provide all the proofs, additional examples, and technical details.

\section{Setup and Notation}\label{sec: Prelim}
\subsection{Rank-Dependent Evaluation}
Let $(\Omega,\mathcal{F})$ be a measurable space.
We define the rank-dependent evaluation $\rho_{u,h,\mathbb{Q}}$ of the utility loss associated with a random variable $X:\Omega\rightarrow\mathbb{R}$ under a given probability measure $\mathbb{Q}$ on $(\Omega,\mathcal{F})$ as the integral with respect to the non-additive measure $h\circ \mathbb{Q}$ \citep{Denneberg}, or equivalently, as a Choquet integral \citep{Choquet1954}: 
\begin{align}
\rho_{u,h,\mathbb{Q}}(X)&\triangleq\int -u(X)\,\mathrm{d}(h\circ \mathbb{Q}) \label{def:distortion_non_additive}\\
&=\int_{0}^{\infty}h\left(\mathbb{Q}[-u(X)>t]\right)\mathrm{d}t + \int_{-\infty}^{0}\left(h\left(\mathbb{Q}[-u(X)>t]\right)-1\right)\mathrm{d}t.\label{def:distortionriskmeasure}
\end{align}
Here, $u$ is a non-decreasing 
utility function defined on a suitable domain containing the support of $X$. 
Furthermore, the function $h:[0,1]\to [0,1]$ is a distortion, or probability weighting, function that is non-decreasing and satisfies $h(0)=0$ and $h(1)=1$. 
We note that $\rho_{u,h,\mathbb{Q}}$ is also known as a \textit{distortion risk measure} when $u$ is the identity function. 
In this paper, we consider distortion functions that may be concave, convex or inverse $S$-shaped.

\begin{definition}\label{def:inverse_S_shape}
We say that $h$ is inverse $S$-shaped if, for some $p^0\in (0,1)$, we have that $h$ is concave for $p\leq p^0$ and $h$ is convex for $p\geq p^0$.
\end{definition} 
If $\Omega$ is discrete with $|\Omega|=m$, then the integral \eqref{def:distortionriskmeasure} reduces to a rank-dependent sum. 
Let $x_{(1)}\geq x_{(2)}\geq \ldots \geq x_{(m)}$ denote the ranked realizations of $X$, with $\{(i)\}^m_{i=1}$ denoting the indices of the ranked realizations. 
The monotonicity of $u$ preserves the ranking of $(x_i)^m_{i=1}$. Therefore, we have
\begin{align}\label{ranked_sum}
    \rho_{u,h,\mathbf{q}}(X)=\sum^m_{i=1}-\left(h\left(\sum^m_{j={i}} q_{(j)}\right)-h\left(\sum^m_{j=i+1} q_{(j)}\right)\right)u\left(x_{(i)}\right),
\end{align}
with $\mathbf{q}\in[0,1]^{m}$ the probability vector associated to $\Omega$ and $\sum^m_{j=m+1}q_{(j)}\triangleq 0$, by convention.

A well-known example of a rank-dependent evaluation is the Conditional-Value-at-Risk (a.k.a.~Expected Shortfall; \citealp{Follmer2011}, Section~4.6), $\mathrm{CVaR}_{1-\alpha}(X)$, which is defined as
\begin{align}\label{def:CVaR}
    \mathrm{CVaR}_{1-\alpha}(X)\triangleq\frac{1}{1-\alpha}\int^{1-\alpha}_0\mathrm{VaR}_{\gamma}(X)\,\mathrm{d}\gamma,\quad 0\leq\alpha<1,
\end{align}
where $\mathrm{VaR}_{\gamma}(X)\triangleq\inf\{x:\mathbb{Q}(-X\leq x)\geq 1-\gamma\}$, $0<\gamma<1$, is the Value-at-Risk. 
For a certain level $\alpha\in[0,1)$, $\mathrm{CVaR}_{1-\alpha}(X)$ can be interpreted as the (sign-changed) average of the left $(1-\alpha)\cdot 100\%$ tail of the risk $X$. 
It is easily verified that $\mathrm{CVaR}_{1-\alpha}$ is a rank-dependent evaluation with linear utility function and distortion function $h(p)=\min\left\{\frac{p}{1-\alpha},1\right\}$. 
Other canonical examples of distortion functions are given in Table~\ref{tab:h_conjugates} of Electronic Companion~\ref{sec:tab}. 
The distortion function captures attitude toward risk whereas the utility function describes attitude toward wealth (see e.g., \citealp{Quiggin1982}, \citealp{Yaari1987}, \citealp{HKS1987}, and \citealp{EL21}).

\subsection{\texorpdfstring{$\phi$}{TEXT}-Divergence Ambiguity Sets}
We construct ambiguity (or uncertainty) sets using $\phi$-divergences. 
For discrete outcome spaces $\Omega$ with $|\Omega|=m$, the $\phi$-divergence $I_\phi(\mathbf{q},\mathbf{p})$ between two probability vectors $\mathbf{q},\mathbf{p}\in[0,1]^{m}$ is defined as
\begin{align}\label{def:Iphi}
I_\phi(\mathbf{q},\mathbf{p})\triangleq\sum^m_{i=1}p_i\phi\left(\frac{q_i}{p_i}\right).
\end{align}
Here, $\phi:[0,\infty)\to \mathbb{R}$ is a convex function that satisfies the following conventions: $\phi(1)=0$, $0\phi(0/0)\triangleq 0$ and $0\phi(x/0)\triangleq x\lim_{t\to \infty}\phi(t)/t$ (see \citealp{Pardo06}, Definition~1.1).
For a nominal probability vector $\mathbf{p}$ and $r>0$, the $\phi$-divergence ambiguity set is defined as
\begin{align}
\mathcal{D}_\phi(\mathbf{p},r)\triangleq\left\{\mathbf{q}\in \mathbb{R}^m~\middle |~ \mathbf{q}\geq \mathbf{0},~ \mathbf{q}^T\mathbf{1}=1, I_{\phi}(\mathbf{q},\mathbf{p})\leq r\right\}.
\label{defDivset}
\end{align}
Here, $\mathbf{0},\mathbf{1}\in \mathbb{R}^m$ are the vectors with all entries equal to 0 and 1, respectively.
As outlined in \citet{BenTal11} (Section~3.2), one may construct $\mathcal{D}_\phi(\mathbf{p},r)$ as a statistical confidence set by replacing $\mathbf{p}$ with an empirical estimator $\hat{\mathbf{p}}_n$. 
Indeed, as shown in \citet{Pardo06} (Corollary~3.1), under the null-hypothesis $H_0:\mathbf{p}=\mathbf{p}_0$, and provided $\phi$ is twice continuously differentiable in a neighborhood of $1$ with $\phi''(1)>1$, the following object converges to a chi-squared distribution:\footnote{From a statistical perspective, it may be more natural to consider $I_{\phi}(\hat{\mathbf{p}}_n,\mathbf{p}_0)$, where the true model $\mathbf{p}_0$ appears as the reference model. 
However, in the robust optimization literature, $I_{\phi}(\mathbf{p}_0,\hat{\mathbf{p}}_n)$ appears more commonly. 
According to Theorem~3.1 and Corollary~3.1 of \citet{Pardo06}, both objects have the same limiting distribution under the null.} 
\begin{align}
    \frac{2n}{\phi''(1)}I_{\phi}(\mathbf{p}_0,\hat{\mathbf{p}}_n)\rightsquigarrow \chi^2_{m-1,1-\alpha}.
\label{eq:chisquare}\end{align}
Here, $n$ is the sample size used for constructing the empirical estimator, and $\chi^2_{m-1,1-\alpha}$ is the $(1-\alpha)$-quantile of the chi-square distribution with $m-1$ degrees of freedom.\footnote{That is, $\mathbb{P}(Z\leq \chi^2_{m-1,1-\alpha})=1-\alpha$, for $Z\sim \chi^2_{m-1}$.}  
Using \eqref{eq:chisquare}, one can construct an asymptotic confidence set $\mathcal{D}_{\phi}(\hat{\mathbf{p}}_n,r)$ by choosing
\begin{align}
    r=\frac{\phi''(1)}{2n}\chi^2_{m-1,1-\alpha}.
\label{r_choice}
\end{align}

\subsection{Problem Formulations, Terminology and Assumptions}\label{subsec: problem_formulation}
We study the following nominal and robust minimization problems: 
\begin{align}\label{P-Nom}
\tag{P-Nom}
     \inf_{\mathbf{a}\in \mathcal{A}}\rho_{u,h,\mathbf{p}}(f(\mathbf{a},\mathbf{X})),\\
\label{P}
\tag{P}
    \inf_{\mathbf{a}\in \mathcal{A}}\sup_{\mathbf{q}\in\mathcal{D}_{\phi}(\mathbf{p},r)}\rho_{u,h,\mathbf{q}}(f(\mathbf{a},\mathbf{X})),
\end{align}
where $\mathbf{a}\in \mathbb{R}^{n_a}$ is the decision vector contained in a compact set of constraints $\mathcal{A}$ consisting of convex inequalities, $\mathbf{X}\in \mathbb{R}^l$ is a random vector, $f:\mathbb{R}^{n_a}\times \mathbb{R}^l\to \mathbb{R}$ is a jointly convex function in $(\mathbf{a},\mathbf{x})$, $\mathcal{D}_{\phi}(\mathbf{p},r)$ is a $\phi$-divergence ambiguity set defined in \eqref{defDivset} with respect to a nominal probability vector $\mathbf{p}$, and $\rho_{u,h,\mathbf{q}}(.)$ is the rank-dependent evaluation defined in \eqref{ranked_sum}, with respect to some probability vector $\mathbf{q}\in \mathbb{R}^m$.

Henceforth, we often consider the nominal and robust rank-dependent evaluations in a constraint form induced by the following epigraph formulation:
\begin{align}\label{nom_rho_constr}\tag{P-Nom-EG}
    \rho_{u,h,\mathbf{p}}(f(\mathbf{a},\mathbf{X})) \leq c,\\
\label{rob_rho_constr}\tag{P-EG}
    \sup_{\mathbf{q}\in \mathcal{D}_\phi(\mathbf{p},r)}\rho_{u,h,\mathbf{q}}(f(\mathbf{a},\mathbf{X})) \leq c.
\end{align}
That is, we emphasize that we are able to deal with a robust rank-dependent evaluation both in the objective and in the constraint, where in the latter case we consider the following type of problem:
\begin{align}\label{P-constraint}
\tag{P-constraint}
\begin{split}
    &\min_{\mathbf{a}\in \mathcal{A}}~g(\mathbf{a})\\
    \text{s.t.}&\sup_{\mathbf{q}\in\mathcal{D}_{\phi}(\mathbf{p},r)}\rho_{u,h,\mathbf{q}}(f(\mathbf{a},\mathbf{X}))\leq c,
\end{split}
\end{align}
where $g$ is convex in $\mathbf{a}$, and $c\in\mathbb{R}$ is a fixed parameter (in contrast to being an epigraph variable). 
The nominal version of \eqref{P-constraint} is defined analogously, with the only difference that the robust constraint is replaced by the nominal constraint \eqref{nom_rho_constr}.

Additionally, we would like to mention the alternative, but highly related, approach of defining a robust rank-dependent evaluation, which appears in the decision theory literature (see e.g., \citealp{dual_theory}):
\begin{align}\label{rdu_pen_1}
    \tilde{\rho}_{\text{rob}}(X)\triangleq\sup_{\mathbf{q}\geq \mathbf{0}, \mathbf{q}^T\mathbf{1}=1}\rho_{u,h,\mathbf{q}}(X)-\theta I_\phi(\mathbf{q},\mathbf{p}),\quad\theta>0.
\end{align}
In Electronic Companion~\ref{app:remark}, we provide additional details on how \eqref{rdu_pen_1} can be reformulated similar to \eqref{rob_rho_constr}, and that the minimization problem~\eqref{P} is equivalent to minimizing~\eqref{rdu_pen_1} for a specific $\theta$.

Let $[m]\triangleq\{1,\ldots,m\}$ for an integer $m$.
The following assumptions are made throughout this paper:
\begin{assumption}\label{assumption:well-defined}
The optimization problem \eqref{P} is finite: $-\infty<$\eqref{P} $<\infty$.
\end{assumption}
\begin{assumption}\label{assump: nominal_p}
The nominal probability vector $\mathbf{p}$ satisfies $p_i>0$ for all $i\in [m]$.
\end{assumption}
\begin{assumption}\label{assump:lower_semi}
The functions $\phi(t), -h(p), -u(f(\mathbf{a},\mathbf{x}))$ are lower-semicontinuous on their respective domains.   
\end{assumption}
\begin{assumption}\label{assump: finite_strict_feasibility}
    \eqref{P-constraint} is finite and contains a Slater point, i.e., there exists $\mathbf{a}_0\in \mathrm{int}(\mathcal{A})$ such that $\sup_{\mathbf{q}\in\mathcal{D}_{\phi}(\mathbf{p},r)}\rho_{u,h,\mathbf{q}}(f(\mathbf{a}_0,\mathbf{X}))< c$. 
    Furthermore, we assume $\eqref{P-constraint}> \inf_{\mathbf{a}\in \mathcal{A}}g(\mathbf{a})$.  
\end{assumption}

Assumption~\ref{assumption:well-defined} can be satisfied if e.g., the image set $\{u(f(\mathbf{a},\mathbf{x}))~|~\mathbf{a}\in \mathcal{A}, \mathbf{x}\in \mathrm{supp}(\mathbf{X})\}$ is contained in a bounded interval. 
Assumption~\ref{assump: nominal_p} constitutes a weak redundancy condition.
Assumption~\ref{assump:lower_semi} is to ensure that the optimization problem~\eqref{P} has an optimal solution in $\mathcal{A}$, and Assumption~\ref{assump: finite_strict_feasibility} is to ensure that \eqref{P-constraint} satisfies the strong duality theorem.

In this paper, the term \textit{robust solution} refers to an optimal solution of \eqref{P} or \eqref{P-constraint}. 
Similarly, a \textit{nominal solution} refers to an optimal solution of \eqref{P-Nom} or the nominal version of \eqref{P-constraint}.  

\subsection{Further Notation}
For a convex function $g:\mathbb{R}^d\to \mathbb{R}$, we denote by $g^*$ its convex \textit{conjugate} $g^*(\mathbf{y})\triangleq\sup_{\mathbf{z}\in \mathrm{dom}(g)}\{\mathbf{y}^T\mathbf{z}-g(\mathbf{z})\}$, where $\mathrm{dom}(g)\triangleq\{\mathbf{z}~|~ g(\mathbf{z})<+\infty\}$ is the effective domain of $g$. 
Note that $g^*$ is always convex, since it is the pointwise supremum of a linear function in $\mathbf{y}$. 
Furthermore, $g^*$ is non-decreasing if $\mathrm{dom}(g)\subset [0,\infty)$. 
For $\lambda>0$, the \textit{perspective} of $g$ is the function $\tilde{g}: (\mathbf{z},\lambda)\mapsto \lambda g\left(\frac{\mathbf{z}}{\lambda}\right)$.
By convention, $\tilde{g}(\mathbf{0},0)=0$ and $\tilde{g}(\mathbf{z},0)=\infty$ for $x\neq 0$. 
Note that $\tilde{g}$ is convex if $g$ is convex.
The \textit{epigraph} of a function $g$ is the set $\mathrm{Epi}(g)\triangleq\{(\mathbf{z},t):g(\mathbf{z})\leq t\}$. 
An epigraph may have a conic representation expressed in a conic inequality $\succeq_{\mathbf{K}}$, where $\mathbf{K}$ is a proper cone (i.e., pointed, closed, convex and with a non-empty interior) and $\mathbf{K}^*$ denotes its dual cone, defined as 
\begin{align*}
    \mathbf{K}^*\triangleq\{\boldsymbol{\lambda}\in \mathbb{R}^d: \boldsymbol{\lambda}^T\mathbf{z}\geq 0,\forall \mathbf{z}\in \mathbf{K}\};
\end{align*}
see \citet{Cha09} and \citet{ModernCVX} for further details.

\section{Robust Counterpart of Rank-Dependent Models}\label{sec: RC}
In this section, we show how the optimization problems \eqref{P-Nom}--\eqref{P} can be reformulated into rank-independent problems with finitely many constraints.
Our idea hinges on the insight that a rank-dependent evaluation with a concave distortion function admits a dual representation that is itself a robust optimization problem with linear probabilities and convex uncertainty set. 
Although the dual representation holds only for concave distortion functions, we show in Section~\ref{sec:NonConcave} how the same idea can be extended to encompass convex and inverse $S$-shaped distortion functions. 
This makes it possible to conduct robust optimization for a broad class of rank-dependent models.

\subsection{Reformulation into the Robust Counterpart}\label{sec: Reformulation}
We start by reformulating the constraints \eqref{rob_rho_constr}--\eqref{nom_rho_constr} with concave distortion functions, since they form the basis of the more complicated convex and inverse $S$-shaped cases described in Section~\ref{sec:NonConcave}. 
In Sections~\ref{sec: RC}--\ref{sec:SolvingAlgorithms}, without further mentioning, we assume the utility function to be concave; Section~\ref{sec:NonConcave} allows $u$ to be non-concave.
The reformulation relies on utilizing the following dual representation, which involves a \textit{composite uncertainty set}. 
\begin{theorem}\label{thm:ReformulationWholeProb}
Let $h:[0,1]\to [0,1]$ be a concave distortion function.
Then, for all $(\mathbf{a},c)\in \mathbb{R}^{n_a+1}$, we have that \eqref{rob_rho_constr} is satisfied if and only if 
\vspace{-0.1cm}
\begin{align}\label{rob_constr_ref}
\sup_{(\mathbf{q},\mathbf{\bar{q}})\in \mathcal{U}_{\phi,h}(\mathbf{p})}\sum^m_{i=1}-\bar{q}_iu(f(\mathbf{a},\mathbf{x}_i))\leq c,
\end{align}
where $\mathcal{U}_{\phi,h}(\mathbf{p})$ is a convex composite uncertainty set given by 
\vspace{-0.1cm}
\begin{align}\label{uncertaintyset}
\mathcal{U}_{\phi,h}(\mathbf{p})\triangleq\left\{(\mathbf{q},\bar{\mathbf{q}})\in \mathbb{R}^{2m} \;
    \begin{tabular}{|l}
      $\sum^m_{i=1}q_i=\sum^m_{i=1}\bar{q}_i=1$\\
      $\sum^m_{i=1}p_i\phi\left(\frac{q_i}{p_i}\right)\leq r$\\
      $\sum_{i\in J}\bar{q}_{i}\leq h\left(\sum_{i\in J} q_{i}\right),~\forall J\subset [m]$\\
      $q_i,\bar{q}_i\geq 0,~\forall i\in [m]$
\end{tabular}
\right\}.
\end{align}
As a special case of interest, \eqref{nom_rho_constr} is satisfied if and only if 
\begin{align}\label{nom_constr_ref}
    \sup_{\mathbf{\bar{q}}\in M_{h}(\mathbf{p})}\sum^m_{i=1}-\bar{q}_iu(f(\mathbf{a},\mathbf{x}_i))\leq c,
\end{align}
where $M_h(\mathbf{p})$ is the set induced by $h$:
\begin{align}\label{defMhq}
    M_h(\mathbf{p})\triangleq\left\{\bar{\mathbf{q}}\in \mathbb{R}^{m} \;
    \begin{tabular}{|l}
      $\sum^m_{i=1}\bar{q}_i=1$\\
      $\sum_{i\in J}\bar{q}_{i}\leq h\left(\sum_{i\in J} p_{i}\right),~\forall J\subset [m]$\\
      $\bar{q}_i\geq 0,~ \forall i\in [m]$
\end{tabular}
\right\}.
\end{align}
\end{theorem}
As a consequence, Theorem~\ref{thm:ReformulationWholeProb} implies that the robust problem~\eqref{P} is equivalent to the following rank-independent problem that is more suitable for robust optimization techniques:
\begin{align}\label{P-ref}
\tag{P-ref}
     \begin{split}
        \min_{\mathbf{a}\in \mathcal{A}} \sup_{(\mathbf{q},\bar{\mathbf{q}})\in \mathcal{U}_{\phi,h}(\mathbf{p})}-\sum^m_{i=1}\bar{q}_iu(f(\mathbf{a},\mathbf{x}_i)).
    \end{split}
\end{align}
Similarly, the nominal problem~\eqref{P-Nom} can be reformulated to
\begin{align}\label{P-Nom-ref}
\tag{P-Nom-ref}
     \begin{split}
        \min_{\mathbf{a}\in \mathcal{A}}\sup_{\bar{\mathbf{q}}\in M_{h}(\mathbf{p})}-\sum^m_{i=1}\bar{q}_iu(f(\mathbf{a},\mathbf{x}_i)).
    \end{split}
\end{align}


Moreover, the equivalence between~\eqref{rob_rho_constr} and~\eqref{rob_constr_ref} features an interpretation somewhat similar to that in \citet{Bertsimas09}, which connects the shape of an uncertainty set to decision theory.
Indeed, this equivalence can be viewed as a device for constructing a myriad of preference-based uncertainty sets $\mathcal{U}_{\phi,h}(\mathbf{p})$ for robust optimization. 
In Figure~\ref{fig: formUset} of Electronic Companion~\ref{App:Shape_UC}, we provide an illustrative visualization of the widely varying shapes of uncertainty sets $\mathcal{U}_{\phi,h}(\mathbf{p})$ that can be generated by selecting specific examples of the deterministic, univariate functions $h$ and $\phi$. 


The following theorem establishes how the semi-infinite constraints \eqref{rob_constr_ref} and \eqref{nom_constr_ref} can be further reformulated into finitely many constraints, yielding a (more) tractable robust counterpart problem.
\begin{theorem}\label{thm:RC}
Let $h:[0,1]\to [0,1]$ be a concave distortion function. 
Then, we have that, for all $(\mathbf{a},c)\in \mathbb{R}^{n_a+1}$, the inequality \eqref{rob_constr_ref} holds if and only if there exist $\alpha,\beta,\gamma,$ $(\nu_j)^{2^m-2}_{j=1},(\lambda_j)^{2^m-2}_{j=1}\in \mathbb{R}$, such that
\begin{align}\label{RC: final}
    \begin{cases}
    \alpha+\beta+\gamma r+\sum^m_{i=1}p_i\gamma \phi^*\left(\frac{-\alpha+\sum_{j:i\in I_j}\nu_j}{\gamma}\right)&\hspace{-0.33cm}+\sum^{2^m-2}_{j=1}\lambda_j(-h)^*\left(\frac{-\nu_j}{\lambda_j}\right)\leq c,\\
    -u(f(\mathbf{a},\mathbf{x}_i))-\beta-\sum_{j:i\in I_j}\lambda_j\leq 0,&\forall i\in [m]\\
    \lambda_j,\gamma \geq 0,& \forall j\in [2^m-2],
    \end{cases}
\end{align}
where $I_1,\ldots, I_{2^m-2}$ are all subsets of $[m]$, except the empty set and the set $[m]$ itself. 

In the nominal case, we have that the inequality \eqref{nom_constr_ref} holds if and only if there exist $\beta\in \mathbb{R}$, $(\lambda_j)_j\geq 0$ such that
\begin{align}\label{RC:nominal}
    \begin{cases}
    \beta+\sum^{2^m-2}_{j=1}\lambda_jh\left(\sum_{k\in I_j}p_k\right)\leq c,\\
    -u(f(\mathbf{a},\mathbf{x}_i))-\beta-\sum_{j: i\in I_j}\lambda_j\leq 0,~\forall i\in [m]\\
    \beta\in \mathbb{R}, \lambda_j\geq 0,~ \forall j\in [2^m-2].
    \end{cases}
\end{align}
\end{theorem}
\begin{remark}\label{rem:extendedh}
For technical reasons, the conjugate $(-h)^*$ in Theorem~\ref{thm:RC} is taken with respect to the domain $[0,\infty)$: $(-h)^*(y)=\sup_{t\geq 0}\{yt+h(t)\}$, by defining $h(x)=1$ for all $x>1$.
\end{remark}
\begin{remark}
In Table~\ref{tab:h_conjugates}, we provide $(-h)^*$ explicitly for a collection of canonical examples. 
\eqref{P-ref} may also be further reformulated using the ``optimistic dual counterpart'' \citep{Gorissen2014}.
This approach is particularly useful if $(-h)^*$ cannot be computed analytically, as is the case e.g., for $h(p)=\Phi(\Phi^{-1}(p)+\nu)$, $\nu>0$, with $\Phi$ the standard normal cdf; see \citet{Wang2000} and \citet{Goovaerts2008}. 
Further details are provided in Electronic Companion~\ref{app: OPDual}.
\end{remark}

\subsection{Conic Representability of the Robust Counterpart}\label{sec: Tractability}
In this subsection, we explore the conic representability of the robust counterpart reformulated in Theorem~\ref{thm:RC}. 
Following \citet{ModernCVX}, a set $\mathcal{S}\subset \mathbb{R}^n$ is conic representable by a cone $\mathbf{K}$ if and only if
\[\mathbf{x}\in \mathcal{S}\Leftrightarrow \exists~ \mathbf{w}, \mathbf{A}, \mathbf{b}: \mathbf{A}\begin{pmatrix}\mathbf{x}\\\mathbf{w} \end{pmatrix}-\mathbf{b}\succeq_\mathbf{K}\mathbf{0}.\]
A function is said to be conic representable if its epigraph is. 
The reformulated robust counterpart \eqref{RC: final} in Theorem~\ref{thm:RC} contains constraints that are expressed in the perspective of the univariate conjugate functions $(-h)^*$ and $\phi^*$. 
We focus on the conic representability of these constraints: 
\begin{align*}
    \lambda(-h)^*\left(\frac{-\nu}{\lambda}\right)\leq z,\ \text{and}\ \gamma\phi^*\left(\frac{s}{\gamma}\right)\leq t.
\end{align*}
In practice, the derivation of the conjugate functions might be difficult. 
However, the epigraph representation of a conjugate function does not always require an explicit form of the conjugate function itself. 
The following two lemmas provide a generic approach for determining the conic representation of the epigraph of the perspective function through that of the function itself. 
Interestingly, this procedure does not require any derivation of the conjugate function. 
In Tables~\ref{tab:h_conjugates} and~\ref{tab:phi_conjugate}, we provide the explicit conic representation for many canonical examples of the conjugate functions $(-h)^*$ and $\phi^*$, where the representations are composed of a combination of standard cones such as the quadratic, power, and exponential. 
These explicit conic representations are useful for the implementation of these constraints in standard optimization software such as CVXPY. 
Details on the derivations of the epigraph representations can be found in Electronic Companion~\ref{App:RCderivation}.

The following two lemmas originate from \citet{ModernCVX} (Propositions~2.3.2 and~2.3.4), where they were stated only for the quadratic cones. 
However, they can be extended to general cones.
\begin{lemma}\label{lem:conic_conjugate}
If $f$ is conic representable with a cone $\mathbf{K}$, i.e., there exist $(\mathbf{A},\mathbf{v},\mathbf{B},\mathbf{b})$ such that
\begin{align*}
    \mathrm{Epi}(f)=\{(\mathbf{x},t):\exists~ \mathbf{w}: \mathbf{A}\mathbf{x}+t\mathbf{v}+\mathbf{B}\mathbf{w}+\mathbf{b}\succeq_\mathbf{K}\mathbf{0}\},
\end{align*}
and $\mathbf{A}\mathbf{x}+t\mathbf{v}+\mathbf{B}\mathbf{w}+\mathbf{b}\succ_\mathbf{K}\mathbf{0}$ for some $(\mathbf{x},t,\mathbf{w})$, 
then $f^*$ is conic representable with the dual cone $\mathbf{K}^*$:
\begin{align*}
    \mathrm{Epi}(f^*)=\{(\mathbf{y},s):\exists~ \boldsymbol{\xi}\in \mathbf{K}^*: \mathbf{A}^T\boldsymbol\xi = -\mathbf{y}, \mathbf{B}^T\boldsymbol\xi=\mathbf{0}, \mathbf{v}^T\boldsymbol\xi =1, s\geq \mathbf{b}^T\boldsymbol{\xi}\}.
\end{align*}
In particular, $f$ and $f^*$ are representable by the same cone if $\mathbf{K}$ is self-adjoint.
\end{lemma}
\begin{lemma}\label{lem:conic_pers}
If $f$ is conic representable with a cone $\mathbf{K}$, then so is its perspective.
\end{lemma}
We provide an example illustrating how the ideas in the proofs of Lemmas~\ref{lem:conic_conjugate} and~\ref{lem:conic_pers} can be systematically applied to determine the conic representation of the epigraph of the perspective transformation of a distortion function.
\begin{example}
Consider 
 $h(p)=
 p^r(1-\log(p^r)),$ 
\ $0<r<1$.   
Instead of deriving a closed-form expression of $(-h)^*$, which can be challenging in specific cases, we determine a conic representation of the epigraph $\lambda (-h)^*\left(\frac{-\nu}{\lambda}\right)\leq t$ by first determining a conic representation of $\mathrm{Epi}(-h)$. 
We have
\begin{align*}
    (p,t)\in \mathrm{Epi}(-h) \Leftrightarrow \exists~ w\geq 0: -w+w\log(w)\leq t, w\leq p^r, w\leq 1,
\end{align*}
due to the fact that $x\mapsto -x+x\log(x)$ is decreasing on $[0,1]$. 
Following the proof of Lemma~\ref{lem:conic_conjugate}, we have that $(y,s)\in \mathrm{Epi}((-h)^*)$ if and only if 
\begin{align*}
    -s\leq \min_{p,w\geq 0, t\in \mathbb{R}}\{-yp+t~|-w+w\log(w)\leq t, w\leq p^r, w\leq 1\}.
\end{align*}
Since this is a bounded convex problem satisfying Slater's condition, the duality theorem implies that the minimization problem is equal to
\begin{align*}
&\max_{\xi_1,\xi_2,\xi_3\geq 0}-\xi_3+\inf_{p,w\geq 0, t\in \mathbb{R}}\{-yp-\xi_2p^r+(\xi_2+\xi_3-\xi_1)w+\xi_1w\log(w)+(1-\xi_1)t\}\\
    &=\max_{\xi_2,\xi_3\geq 0}-\xi_3+\inf_{p,w\geq 0}\{-yp-\xi_2p^r+(\xi_2+\xi_3-1)w+w\log(w)\} \\
    &=\max_{\xi_2,\xi_3\geq 0}\{-\xi_3-e^{-\xi_2-\xi_3}+\left(r^{\frac{1}{1-r}}-r^{\frac{r}{1-r}}\right)\xi_2^{\frac{1}{1-r}}|y|^{1-\frac{1}{1-r}}~|~y\leq 0\},
\end{align*}
where in the last equality we computed $\inf_{p\geq 0}\{-yp-\xi_2p^r\}$ explicitly.\footnote{$
    \inf_{p\geq 0}-yp-\xi_2p^r=\begin{cases}
        \left(r^{\frac{1}{1-r}}-r^{\frac{r}{1-r}}\right)\xi_2^{\frac{1}{1-r}}|y|^{1-\frac{1}{1-r}}&y<0,\quad\xi_2>0\\
        0&y\leq 0,\quad\xi_2=0\\
        -\infty&y\geq 0,\quad\xi_2>0.
    \end{cases}$} 
This gives that $(y,s)\in \mathrm{Epi}((-h)^*)$ if and only if there exist $\xi_2,\xi_3,\xi_4\geq 0$ such that  
\begin{align*}
    \xi_3+e^{-\xi_2-\xi_3}+\left(r^{\frac{r}{1-r}}-r^{\frac{1}{1-r}}\right)\xi_4\leq s,\quad\xi_2\leq |y|^r\xi_4^{1-r},\quad y\leq 0.
\end{align*}
Lemma~\ref{lem:conic_pers} then yields that $\lambda (-h)^*\left(\frac{-\nu}{\lambda}\right)\leq z$ if and only if there exist $\xi_2,\xi_3,\xi_4\geq 0$ such that
\begin{align*}
     \xi_3+\lambda e^{-(\xi_2+\xi_3)/\lambda}+\left(r^{\frac{r}{1-r}}-r^{\frac{1}{1-r}}\right)\xi_4\leq z,\quad\xi_2\leq |\nu|^r\xi_4^{1-r},\quad\nu \geq 0,
\end{align*}
which is a combination of the power cone and the exponential cone.
\end{example}
In some cases, one can also calculate a conjugate function by writing it as an inf-convolution. 
For example, if $f$ is a sum of individual $f_i$'s: $f(\mathbf{x})=\sum^n_{i=1}f_i(\mathbf{x})$, then the conjugate of a sum is the inf-convolution of the sum of the conjugates (see Theorem~2.1, \citealp{DickBoek}):
\begin{align*}
    \left(\sum^n_{i=1}f_i\right)^*(\mathbf{s})=\inf_{\{\mathbf{v}^i\}_{i\in [n]}}\left\{\sum^n_{i=1}f_i^*(\mathbf{v}^i)~\middle |~ \sum^n_{i=1}\mathbf{v}^i=\mathbf{s}\right\},
\end{align*}
where the infimum can eventually be omitted when considering the epigraph formulation.

\section{Solving the Robust Problem I: A Cutting-Plane Method}\label{sec:ReducingExp1}
As shown in Theorem~\ref{thm:RC}, the number of reformulated constraints grows exponentially, as $2^m$, where $m$ is the dimension of the probability vector. 
If $m$ is of small or moderate size, then the reformulations in \eqref{RC: final} and \eqref{RC:nominal} can be applied to solve the minimization problems \eqref{P} and \eqref{P-Nom} exactly. 
However, for larger $m$, this becomes computationally intractable. 
In this section, we show how the reformulation of \eqref{P} to \eqref{P-ref} (and similarly for \eqref{P-Nom} to \eqref{P-Nom-ref}) by Theorem~\ref{thm:ReformulationWholeProb} enables us to devise a suitable cutting-plane method, which circumvents this curse of dimensionality.  

\subsection{The Cutting-Plane Algorithm}\label{subsec:cuttingplane}
The standard cutting-plane algorithm is one of the most applied algorithms to solve robust optimization problems and has shown great efficiency in many applications (e.g., \citealp{Boyd09}, 
\citealp{CutvsRef}).
The general idea is simple: 
Approximate the uncertainty set $\mathcal{U}$ with a suitably chosen subset $\mathcal{U}_j\subset \mathcal{U}$ and solve the corresponding robust problem, which is often a simpler problem, similar to the nominal problem. 
If the solution is feasible according to the robust evaluation with respect to the original set $\mathcal{U}$, then the process is terminated.
Otherwise, the worst-case parameter in $\mathcal{U}$ associated with the current solution is added to $\mathcal{U}_j$, and the process is repeated. 

In our case, we apply the cutting-plane method to the reformulated problem \eqref{P-ref}, where the uncertainty set is given by the \textit{composite uncertainty set} $\mathcal{U}_{\phi,h}(\mathbf{p})$. 
After obtaining a candidate solution using a suitable subset $\mathcal{U}_{j}$, we verify its feasibility with respect to the original uncertainty set $\mathcal{U}_{\phi,h}(\mathbf{p})$. 
This involves solving the following optimization problem, for a given solution $\mathbf{a}_*$:
\begin{align*}
    \sup_{(\mathbf{q},\mathbf{\bar{q}})\in \mathcal{U}_{\phi,h}(\mathbf{p})}\sum^m_{i=1}-\bar{q}_iu(f(\mathbf{a}_*,\mathbf{x}_i)).
\end{align*}
At first sight, this seems problematic due to the $2^m$ number of constraints in $\mathcal{U}_{\phi,h}(\mathbf{p})$. Fortunately, this can be avoided, since the equivalent robust rank-dependent evaluation is the optimization problem: 
\begin{align}\label{robustcheck_1}
     \sup_{\mathbf{q}\in \mathcal{D}_\phi(\mathbf{p},r)}\rho_{u,h,\mathbf{q}}(f(\mathbf{a}_*,\mathbf{X})).
\end{align}
Problem~\eqref{robustcheck_1} can be solved efficiently. 
Indeed, for a given solution $\mathbf{a}_*$, we can assess the ranking of the outcomes:  $u(f(\mathbf{a}_j,\mathbf{x}_{(1)}))\geq \ldots \geq u(f(\mathbf{a}_j,\mathbf{x}_{(m)}))$. 
Then, by rewriting the alternating sum in \eqref{ranked_sum}, problem \eqref{robustcheck_1} becomes equal to the following convex optimization problem:
\begin{align}\label{ranked_opt_robustcheck}
    \sup_{\mathbf{q}\in \mathcal{D}_\phi(\mathbf{p},r)}\sum^{m}_{k=1}h\left(\sum^m_{j=k}q_{(j)}\right)\left(u(f(\mathbf{a}_*,\mathbf{x}_{(j-1)}))-u(f(\mathbf{a}_*,\mathbf{x}_{(j)}))\right),
\end{align}
where $u(f(\mathbf{a}_*,\mathbf{x}_{(0)})):=0$. The optimal solution of \eqref{ranked_opt_robustcheck} is a probability vector $\mathbf{q}^*\in \mathcal{D}_{\phi}(\mathbf{p},r)$. 
This probability vector further yields a $\bar{\mathbf{q}}^*$ such that $\bar{q}_{(i)}^*=h\left(\sum^{m}_{k=i}q^*_{(k)}\right)-h\left(\sum^{m}_{k=i+1}q^*_{(k)}\right)$ and $\bar{q}^*_{(m)}=h(q^*_{(m)})$, which is precisely the probability vector that constitutes the rank-dependent sum \eqref{ranked_sum}. 
The following lemma justifies that we may add the probability vectors $(\mathbf{q}^*,\bar{\mathbf{q}}^*)$ as the worst-case probabilities at each iteration of the cutting-plane procedure.

\begin{lemma}\label{lem: qq_in_U}
    We have that $(\mathbf{q}^*,\bar{\mathbf{q}}^*)\in \mathcal{U}_{\phi,h}(\mathbf{p})$.
\end{lemma}

We can now describe the cutting-plane procedure in full detail, which we do more precisely in Algorithm~\ref{cutting_plane_algo}.   
\begin{algorithm}[htp]
\caption{\textbf{Cutting-Plane Method}}\label{cutting_plane_algo}
\begin{algorithmic}[1]
\State Start with $\mathcal{U}_1=\{(\mathbf{p},\mathbf{p})\}$. 
Fix a tolerance parameter $\epsilon_{\mathrm{tol}}>0$.
\State\label{step2}At the $j$-th iteration, solve the following problem with the uncertainty set $\mathcal{U}_j$:
\begin{align}\label{cutting_plane_pb_j}
    \begin{split}
        \min_{\mathbf{a}\in \mathcal{A}}\sup_{(\mathbf{q},\mathbf{\bar{q}})\in \mathcal{U}_j}-\sum^m_{i=1}\bar{q}_iu(f(\mathbf{a},\mathbf{x}_i)).
    \end{split}
\end{align}
\State\label{step3}Let $(\mathbf{a}_j,c_j)$ be the optimal solution and objective value of \eqref{cutting_plane_pb_j}. 
Determine a ranking of the realizations: 
\begin{align*}
    -u(f(\mathbf{a}_j,\mathbf{x}_{(1)}))\leq \ldots \leq -u(f(\mathbf{a}_j,\mathbf{x}_{(m)})).
\end{align*}
Then, solve the optimization problem \eqref{ranked_opt_robustcheck}, which gives an optimal objective value $v_j$ and a solution $(\mathbf{q}^*_j,\bar{\mathbf{q}}^*_j)$.
\State\label{step4}If $v_j- c_j\leq \epsilon_{\mathrm{tol}}$, then the solution is accepted and the process is terminated.
\State\label{step5}If not, set $\mathcal{U}_{j+1} = \mathcal{U}_j \cup \{(\mathbf{q}^*_j,\mathbf{\bar{q}}^*_j)\}$ and repeat steps~\ref{step2}--\ref{step5}.
\end{algorithmic}
\end{algorithm}
We note that the cutting-plane algorithm always computes a \emph{lower} bound $c_j$ on the optimal objective value of \eqref{P-ref}, since solving \eqref{cutting_plane_pb_j} at each iteration $j$ with respect to a subset $\mathcal{U}_{j}\subset \mathcal{U}_{\phi,h}(\mathbf{p})$ is a less conservative problem. 
Moreover, the lower bound is improved iteratively, since $\mathcal{U}_j\subset \mathcal{U}_{j+1}$. 
Furthermore, by solving \eqref{ranked_opt_robustcheck}, which is a robust rank-dependent evaluation of a particular feasible solution, we obtain an \emph{upper} bound $v_j$ on \eqref{P-ref} at each iteration. 
Hence, the final step of the cutting-plane algorithm yields an upper and a lower bound on the exact optimal objective value \eqref{P-ref}, with a gap value $v_j-c_j\leq \epsilon_{\mathrm{tol}}>0$ that can be chosen arbitrarily small.

The natural, formal question that arises is whether the cutting-plane algorithm actually terminates after finitely many iterations, thus yielding convergent upper and lower bounds as $\epsilon_{\mathrm{tol}}\to 0$. 
The following theorem, which has its roots in \citet{Boyd09}, 
states that, for any $\epsilon_{\mathrm{tol}}>0$, the termination is indeed guaranteed.
We note that \citet{Boyd09} imposes Lipschitz continuity conditions, which we avoid in our setting. 
\begin{theorem}\label{thm:cuttingplane_conv}
Let $h$ be a concave distortion function.
Suppose $\sup_{\mathbf{a}\in \mathcal{A},i\in [m]}u(f(\mathbf{a}, \mathbf{x}_i))<\infty$. 
Then, for all $\epsilon_{\mathrm{tol}}>0$, Algorithm~\ref{cutting_plane_algo} terminates after finitely many iterations.
\end{theorem}

\begin{remark}
Algorithm~\ref{cutting_plane_algo} also applies to the nominal case \eqref{P-Nom}, where $r\equiv 0$ in the set $\mathcal{U}_{\phi,h}(\mathbf{p})$. 
Furthermore, Theorem~\ref{thm:cuttingplane_conv} still holds in this case, since its proof does not depend on the choice of $r$. 
\end{remark} 

\subsection{Robust Optimization with a Rank-Dependent Evaluation in the Constraint}
In the previous subsection, we have shown that our cutting-plane algorithm yields convergent upper and lower bounds to problems \eqref{P-Nom}--\eqref{P}, where the rank-dependent evaluation appears in the objective function. 
In this subsection, we discuss how the cutting-plane algorithm can be adapted to the robust problem \eqref{P-constraint}, where the RDU model appears in the constraint, to provide a convergent lower bound. 
The adaptation of Algorithm~\ref{cutting_plane_algo} to problem~\eqref{P-constraint} is presented in full detail in Algorithm~\ref{Algo: cutting_plane_constraint} in Electronic Companion~\ref{app:con}.\footnote{In a nutshell, we replace \eqref{cutting_plane_pb_j} in Algorithm~\ref{cutting_plane_algo} by \eqref{P-constraint} adapted to the set $\mathcal{U}_j$ and replace $c_j$ in step \ref{step4} by the constraint parameter $c$.}  
The following theorem establishes the convergence of the cutting-plane algorithm.
\begin{theorem}\label{thm:cuttingplane_conv_constr}
Let $h$ be concave.
Suppose $\sup_{\mathbf{a}\in \mathcal{A},i\in [m]}u(f(\mathbf{a}, \mathbf{x}_i))<\infty$. 
Then, for all $\epsilon_{\mathrm{tol}}>0$, Algorithm~\ref{Algo: cutting_plane_constraint} terminates after finitely many iterations.
Moreover, if the function $\mathbf{a}\mapsto u(f(\mathbf{a},\mathbf{x}))$ is concave for all $\mathbf{x}\in \mathbb{R}^l$, then the optimal objective value of the final solution obtained from Algorithm~\ref{Algo: cutting_plane_constraint} converges to that of \eqref{P-constraint}, as $\epsilon_{\mathrm{tol}}\to 0$.
\end{theorem}

Thus, using the cutting-plane algorithm, we can obtain lower bounds that converge to the exact optimal objective value of \eqref{P-constraint}, as $\epsilon_{\mathrm{tol}}\to 0$. 
Naturally, one would also like to obtain an upper bound on \eqref{P-constraint}, as well as a feasible solution of \eqref{P-constraint}. 
We note that this is not guaranteed by the cutting-plane algorithm as described in Algorithm~\ref{Algo: cutting_plane_constraint}, since it only provides a solution that is feasible within a tolerance $\epsilon_{\mathrm{tol}}>0$. 
With this objective in mind, we explore the rank-dependent nature of $\rho_{u,h,\mathbf{q}}(.)$ and propose a method to obtain an upper bound and a feasible solution of \eqref{P-constraint}, which requires solving an optimization problem with only $3m+3$ number of constraints. 
We first state the following definition.
\begin{definition}\label{def: ranked_index}
    Given an $\mathbf{a}\in \mathcal{A}$, we let $\mathcal{I}(\mathbf{a})$ denote the set of all permutations $(i_1,\ldots, i_m)$ of the index vector $(1,\ldots,m)$ such that the ranking $ -u(f(\mathbf{a},\mathbf{x}_{i_1}))\leq \ldots \leq -u(f(\mathbf{a},\mathbf{x}_{i_m}))$ holds. 
\end{definition}
The idea is that for any given solution $\mathbf{a}_0\in \mathcal{A}$, we can determine a ranking of the realizations $\{-u(f(\mathbf{a}_0,\mathbf{x}_i))\}^m_{i=1}$ and obtain a vector of indices $(i_1,\ldots, i_m)\in \mathcal{I}(\mathbf{a}_0)$. 
A natural upper bound on \eqref{P-constraint} can then be computed by solving the following optimization problem induced by the ranking $(i_1,\ldots, i_m)$ of $\mathbf{a}_0$:
\begin{align}\label{ranked_upper}
    \begin{split}
        U^*(\mathbf{a}_0)\triangleq\min_{\mathbf{a}\in \mathcal{A}}\left\{g(\mathbf{a})~\middle |~ 
        \sup_{(\mathbf{q},\bar{\mathbf{q}})\in \mathcal{U}^{(i_1,\ldots,i_m)}_{\phi,h}(\mathbf{p})}-\sum^m_{i=1}\bar{q}_iu(f(\mathbf{a},\mathbf{x}_i))\leq c\right\},
    \end{split}
\end{align}
where the rank-dependent uncertainty set is 
\begin{align}
\nonumber
    \begin{split}
    \mathcal{U}_{\phi,h}^{(i_1,\ldots,i_m)}(\mathbf{p})\triangleq\left\{(\mathbf{q},\bar{\mathbf{q}})\in \mathbb{R}^{2m} \;
    \begin{tabular}{|l}
      $\sum^m_{i=1}q_i=\sum^m_{i=1}\bar{q}_i=1$\\
      $\sum^m_{i=1}p_i\phi\left(\frac{q_i}{p_i}\right)\leq r$\\
      $\sum^m_{j= k}\bar{q}_{i_j}\leq h\left(\sum^m_{j=k} q_{i_j}\right),~\forall k\in [m]$\\
      $q_i,\bar{q}_i\geq 0,~\forall i\in [m]$
\end{tabular}
\right\}.
    \end{split}
\end{align}
Indeed, since $\mathcal{U}_{\phi,h}(\mathbf{p})\subset \mathcal{U}_{\phi,h}^{(i_1,\ldots,i_m)}(\mathbf{p})$, solving \eqref{ranked_upper} (which can be done using Theorem~\ref{thm:RC}) yields an upper bound and a feasible solution of \eqref{P-constraint}. 
In particular, if $\mathbf{a}_0$ is a good approximation of the optimal solution of \eqref{P-constraint}, such that the ranking coincides, then $U^*(\mathbf{a}_0)$ is exactly equal to the optimal objective value of \eqref{P-constraint}. 
This observation is made precise in the following lemma and is pivotal for developing a convergent upper bound for \eqref{P-constraint}.
\begin{lemma}\label{lem: optimal_ranking}
Let $\mathbf{a}^*$ be a minimizer of \eqref{P-constraint}. 
Then, for any $\mathbf{a}_0\in \mathcal{A}$ such that $\mathcal{I}(\mathbf{a}_0)\subset \mathcal{I}(\mathbf{a}^*)$, we have that the value $U^*(\mathbf{a}_0)$ as defined in \eqref{ranked_upper} is equal to the optimal objective value of \eqref{P-constraint}. 
\end{lemma}
Lemma~\ref{lem: optimal_ranking} suggests that if one approximates the optimal solution of \eqref{P-constraint} with a sequence of cutting-plane solutions as $\epsilon_{\mathrm{tol}}\to 0$, then under certain continuity conditions, the upper bounds computed in \eqref{ranked_upper} also converge to the exact value. 
This is established in the following theorem.  
\begin{theorem}\label{thm: convergence_cutting_plane_constraint}
    Let $h$ be concave.
    Assume $\sup_{\mathbf{a}\in \mathcal{A},i\in [m]}u(f(\mathbf{a}, \mathbf{x}_i))<\infty$.
    If the functions $g(\mathbf{a})$ and $ u(f(\mathbf{a},\mathbf{x}_i))$ are continuous on $\mathcal{A}$, for all $i\in [m]$, and $\mathbf{a}_n$ is a sequence of solutions obtained from Algorithm~\ref{Algo: cutting_plane_constraint} with $\epsilon_{\mathrm{tol},n}\to 0$, then $(U^*(\mathbf{a}_n))_{n\geq 1}$ is a sequence of upper bounds that converges to the optimal objective value of \eqref{P-constraint}.
\end{theorem}

\subsection{Choquet Expected Utility}\label{sec:ChoquetEU}
In this subsection, we briefly discuss how our cutting-plane method can also be extended to solve optimization problems for a more general rank-dependent model, namely the Choquet expected utility model \citep{Schmeidler1986,Schmeidler1989}. 
Let $c:\mathcal{F}\to [0,1]$ be a monotone set function\footnote{$c(A)\leq c(B),$ if $A\subset B$.} such that $c(\emptyset)=0,~c(\Omega)=1$. 
The Choquet expected utility model evaluates a random variable $X$ by the Choquet integral:
\begin{align}
    \int -X \,\mathrm{d}c\triangleq\int^\infty_0 c(-X>t)\,\mathrm{d}t+\int^0_{-\infty}(c(-X>t)-1)\,\mathrm{d}t.
\end{align}
The rank-dependent evaluation $\rho_{u,h,\mathbb{Q}}$ as defined in \eqref{def:distortionriskmeasure} is a special case of the Choquet expected utility, where $c(A)=h(\mathbb{Q}(A))$. 
If $c$ is a submodular set function,\footnote{$c(A\cap B)+c(A\cup B)\leq c(A)+ c(B),\ \forall A,B\in \mathcal{F}$.} then the Choquet expected utility can also be written as a worst-case expectation (see \citealp{Denneberg}), similar to Theorem~\ref{thm:ReformulationWholeProb}:
\begin{align}\label{Choquet_EU}
     \int -X \,\mathrm{d}c=\sup_{\bar{\mathbf{q}}\in \mathcal{U}_{c}}-\sum^m_{i=1}\bar{q}_ix_i,\quad
     \mathcal{U}_{c}\triangleq\left\{\bar{\mathbf{q}}\in \mathbb{R}^{m} \;
    \begin{tabular}{|l}
      $\sum^m_{i=1}\bar{q}_i=1$\\
      $\sum_{i\in J}\bar{q}_{i}\leq c(J),~\forall J\subset [m]$\\
      $\bar{q}_i\geq 0,~ \forall i\in [m]$
\end{tabular}
\right\}.
\end{align}
The uncertainty set $\mathcal{U}_{c}$ has essentially the same structure as the set $\mathcal{U}_{\phi,h}(\mathbf{p})$ defined in \eqref{uncertaintyset} and consists of only linear constraints. 
Therefore, the same procedure of the cutting-plane algorithm can be extended naturally to solve the more general rank-dependent minimization problem $\min_{\mathbf{a}\in \mathcal{A}}\int -f(\mathbf{a},\mathbf{X})\,\mathrm{d}c$. 
By contrast, the piecewise-linear approximation methods, which we will discuss in Section~\ref{sec:SolvingAlgorithms}, constitute a more restrictive, but as we will explicate in some cases more efficient, approach.

\section{Solving the Robust Problem II: Piecewise-Linear Approximation }\label{sec:SolvingAlgorithms}
In this section, we develop a different method for computing lower and upper bounds on the optimal objective value of \eqref{P-ref}, which relies on a suitable approximation of the distortion functions, and is leveraged in Section~\ref{sec:NonConcave}.
\subsection{Piecewise-Linear Distortion Functions}\label{subsec: piecelin_exact}
In this subsection, we show that if the distortion function $h$ is concave and piecewise-linear, then the exponential complexity of the constraints in $\mathcal{U}_{\phi,h}(\mathbf{p})$ can be reduced to the order of $m\cdot K$, where $m$ is the dimension of the probability vector, and $K$ is the number of linear pieces of $h$. 

More precisely, we consider a function $h=\min_{j\in [K]}h_j$, with affine functions $h_j(p)=l_j\cdot p+b_j$, such that the slopes $l_1>\cdots>l_K$ are decreasing and the intercepts $b_1<\cdots<b_K$ are increasing. 
The affine functions are assumed to be defined on a set of support points $0=x_0<x_1<\cdots <x_K=1$, such that $h(p)=h_j(p),~ \text{if}~p\in [x_{j-1},x_j]$, for all $j\in [K]$. 
Furthermore, we impose $b_1\equiv 0$ and $l_K+b_K\equiv 1$, so that $h(0)=0$ and $h(1)=1$. 
We refer to this type of function as a \textit{$K$-piecewise-linear distortion function}. 
Clearly, this function is non-decreasing and concave.

We prove the following lemma.
\begin{lemma}\label{lem:PL_UC}
    Let $h$ be a $K$-piecewise-linear distortion function.
    Then, $(\mathbf{q},\bar{\mathbf{q}})\in \mathcal{U}_{\phi,h}(\mathbf{p})$ if and only if there exists $\mathbf{t}\triangleq(\{t_{i,j}\}^m_{i=1})^{K}_{j=1}\in \mathbb{R}^{mK}$ such that the variables $(\mathbf{q},\bar{\mathbf{q}},\mathbf{t})$ satisfy the constraints
    \begin{align}\label{PL_UC}
        \begin{cases}
            q_i, \bar{q}_i, t_{i,j}\geq 0,~\forall i \in [m],~ \forall j\in [K]\\
            \bar{q}_i\leq l_j\cdot q_i+t_{i,j},~\forall i\in [m],~ \forall j\in [K]\\
            \sum^m_{i=1}t_{i,j}\leq b_j,~\forall j\in [K]\\
            \sum^m_{i=1}q_i=\sum^m_{i=1}\bar{q}_i=1\\
            \sum^m_{i=1}p_i\phi\left(\frac{q_i}{p_i}\right)\leq r.
        \end{cases}       
    \end{align}
\end{lemma}
Hence, for $K$-piecewise-linear distortion functions, we obtain the following reformulation of the robust counterpart~\eqref{rob_constr_ref}.
\begin{theorem}\label{thm:RCpiecewiselin}
Let $h$ be a K-piecewise-linear distortion function. 
Then, we have that $(\mathbf{a},c)$ satisfies the constraint~\eqref{rob_constr_ref} if and only if there exist variables $\alpha,\beta,\gamma,\{\lambda_{ij}\}_{i\in [m],j\in [K]},\{\nu_j\}^K_{j=1}\in \mathbb{R}$ such that
\begin{align}\label{RC:piecewiseaffine}
\begin{cases}
    \alpha +\beta+\gamma r +\sum^K_{j=1}\nu_jb_j+\sum^m_{i=1}p_i\gamma\phi^*\left(\frac{-\alpha+\sum^K_{j=1}\lambda_{ij}l_j}{\gamma}\right)\leq c\\
    -u(f(\mathbf{a},\mathbf{x}_i))-\beta-\sum^K_{j=1}\lambda_{ij}\leq 0,~\forall i\in [m]\\
    \lambda_{ij}\leq \nu_j,~\forall i\in [m],~\forall j\in [K]\\
    \gamma, \lambda_{ij},\nu_j\geq 0,~\forall i\in [m],~\forall j\in [K].
\end{cases}    
\end{align}
In the nominal case, we have that \eqref{nom_constr_ref} holds if and only if there exist variables $\beta,\lambda_{ij},\nu_j\in \mathbb{R}$ such that \eqref{RC:piecewiseaffine} holds with $\alpha=\gamma\equiv 0$, whence the first line becomes $\beta+\sum^K_{j=1}\nu_jb_j+\sum^m_{i=1}\sum^K_{j=1}\lambda_{ij}l_jp_i\leq c$.
\end{theorem}
Thus, for $K$-piecewise-linear distortion functions, the exponential complexity in the uncertainty set $\mathcal{U}_{\phi,h}(\mathbf{p})$ can be reduced. 
This result also provides a tractable approximation for general non-piecewise-linear concave distortion functions, as we outline in the following subsection, exploiting that each concave function can be approximated uniformly by a piecewise-linear concave function.

\subsection{Piecewise-Linear Approximation of General Concave Distortion Functions}\label{subsec:approx_affine}
For a general concave $h$, we can reduce the complexity in $\mathcal{U}_{\phi,h}(\mathbf{p})$ by approximating $h$ with a piecewise-linear function $h_\epsilon$, such that $\max_{x\in[0,1]}|h(x)-h_\epsilon(x)|\leq \epsilon$ for any given $\epsilon>0$. 
An upper approximation $h_\epsilon\geq h$ and a lower approximation $h_\epsilon\leq h$ yield an upper and lower bound on the optimal objective value of \eqref{P-ref}, respectively. 
Uniform approximation of general concave functions with piecewise-linear functions has been studied in the literature (see, e.g., \citealp{PieceLin1}, \citealp{PieceLin2}). 

More specifically, using the concavity of $h$, one can approximate $h$ from below by choosing a set of support points $\{x_i\}_{i=0}^K$ and considering the concave piecewise-linear function that connects the values $\{h(x_i)\}_i$. 
Given an $\epsilon>0$, the required number of support points $K$ can be minimized.
We briefly describe our approach of determining the minimal $K$ support points $\{x_i\}^K_{i=0}$ such that $h$ is lower approximated by a piecewise-linear function with a prescribed error $\epsilon$. 
Set $x_0=0$. 
At the $i$-th iteration, 
we choose the next support point $x_{i+1}$ such that the maximal error is equal to $\epsilon$:
\begin{align}\label{h_error}
    e_i(x_{i+1})\triangleq\sup_{x\in [x_i,x_{i+1}]}\left\{h(x)-\frac{h(x_{i+1})-h(x_i)}{x_{i+1}-x_i}(x-x_i)-h(x_i)\right\}=\epsilon.
\end{align}
If $e_i(1)\leq \epsilon$, then we simply choose $x_{i+1}=1$. 
Otherwise, we choose $x_{i+1}$ such that \eqref{h_error} holds. 
By construction, the piecewise-linear approximation $h_\epsilon$ induced by this set of support points has a maximum approximation error of $\epsilon$ and satisfies $h_\epsilon\leq h$, due to concavity. 
The existence of such $x_{i+1}$ at each iteration is verified in the following lemma. 
Moreover, it implies that we can solve \eqref{h_error} for $x_{i+1}$ using the bisection method.
\begin{lemma}\label{lem:PL_h_construct}
    Let $h$ be a continuous, increasing, strictly concave function on $[0,1]$. 
    Then, $e_i(x_{i+1})$ is an increasing and continuous function in $x_{i+1}$. 
    In particular, for any $\epsilon>0$ and iteration $i$, if $e_i(1)>\epsilon$, then there exists a $x_{i+1}\in (x_i,1)$ such that $e_i(x_{i+1})=\epsilon$.
\end{lemma}

\subsection{The Piecewise-Linear Approximation Method and its Convergence}
Given a piecewise-linear lower approximation $h_\epsilon\leq h$ of $h$, we can approximate the uncertainty set $\mathcal{U}_{\phi,h}(\mathbf{p})$ with $\mathcal{U}_{\phi,h_\epsilon}$ and apply Theorem~\ref{thm:RCpiecewiselin}. 
This yields a lower bound on \eqref{P-ref}.  
Similarly, the concave distortion function $\tilde{h}_\epsilon$, which we define as
\begin{align}\label{upper_distortion}
    \tilde{h}_\epsilon(p)\triangleq\begin{cases}
        0,& p=0\\
        \min\{h_\epsilon(p)+\epsilon,1\},& 0<p\leq 1,
    \end{cases}
\end{align}
yields an upper approximation of $h$ with uniform error $\epsilon$. 
Hence, we can also approximate $\mathcal{U}_{\phi,h}(\mathbf{p})$ by $\mathcal{U}_{\phi,\tilde{h}_\epsilon}(\mathbf{p})$ and apply Theorem~\ref{thm:RCpiecewiselin} (note that the constraints \eqref{RC:piecewiseaffine} for $\tilde{h}_{\epsilon}$ are the same as those for $h_{\epsilon}$, but with $b_j$ replaced by $b_j+\epsilon$). 
This yields an upper bound for \eqref{P-ref} since $\mathcal{U}_{\phi,h}(\mathbf{p})\subset\mathcal{U}_{\phi,\tilde{h}_\epsilon}(\mathbf{p})$. 
Therefore, we obtain an upper and lower bound using the piecewise-linear approximation method that we summarize in Algorithm~\ref{algo_PL_method}. 
We also show that both bounds converge to the optimal objective value of the exact problem \eqref{P-ref}, if the approximation error $\epsilon$ approaches zero.  
Thus, Algorithm~\ref{algo_PL_method} terminates for any parameter $\delta>0$. 
This also applies, \textit{mutatis mutandis}, to problem~\eqref{P-constraint}.

\begin{algorithm}[htp]
\caption{\textbf{Piecewise-Linear Approximation Method}}\label{algo_PL_method}
\begin{algorithmic}[1]
\State Set an $\epsilon>0$. 
Approximate $h$ from below by a $K$-piecewise-linear distortion function $h_\epsilon$ such that: 
(1) The number of pieces, $K$, is minimized; 
(2) $h_{\epsilon}(x)\leq h(x),~\forall x \in [0,1]$; (3) $\sup_{x\in [0,1]}h(x)-h_{\epsilon}(x)\leq \epsilon$.
\State Solve problem~\eqref{P-ref} with $\mathcal{U}_{\phi,h}(\mathbf{p})$ replaced by $\mathcal{U}_{\phi,h_\epsilon}(\mathbf{p})$, using Theorem~\ref{thm:RCpiecewiselin}. 
This gives a lower bound $L^*$ and an optimal solution $\mathbf{a}^*$. 
Do the same with $\tilde{h}_\epsilon$ to obtain an upper bound $U^*$.
\State If $U^*-L^*<\delta$ for some prescribed $\delta>0$, then we take $\mathbf{a}^*$ as the final solution.
\State Otherwise, set $\epsilon\rightarrow \epsilon/2$ and perform all previous steps to obtain new bounds and solutions.
\end{algorithmic}
\end{algorithm}

\begin{theorem}\label{thm:convergencePL}
Let $h$ be concave. 
Suppose $\sup_{\mathbf{a}\in \mathcal{A},i\in [m]}u(f(\mathbf{a},\mathbf{x}_i))<\infty$. 
Then, Algorithm~\ref{algo_PL_method} terminates after finitely many iterations, for any $\epsilon,\delta>0$. If the function $\mathbf{a}\mapsto u(f(\mathbf{a},\mathbf{x}))$ is concave for all $\mathbf{x}\in \mathbb{R}^l$, 
then this also holds when Algorithm~\ref{algo_PL_method} is applied to problem~\eqref{P-constraint}.
\end{theorem}


\section{Non-Concave Distortion Functions}\label{sec:NonConcave}

In the previous sections, we have discussed how to optimize a rank-dependent model when the distortion function is concave. 
We now extend these ideas to convex and inverse $S$-shaped distortion functions, which are often found in empirical work (see e.g., \citealp{Wakker}, \citealp{Prelec}). 
Optimization problems with non-concave distortion functions are challenging due to their non-convex nature. 
Moreover, rank-dependent models with non-concave distortion functions lack a dual representation, which was a pivotal tool that enabled us to reformulate the rank-dependent problem. 
In the recent literature (e.g., \citealp{clmpaper} and \citealp{ConcentrationPaper}), uncertainty sets with a specific structure have been identified for which optimizing the robust distortion risk measure with non-concave distortion function is equivalent to optimizing the same model with $h$ replaced by its concave envelope $\hat{h}$.\footnote{The concave envelope of a function $h$ is the smallest concave function that dominates $h$.} 
However, this equivalence is typically not satisfied in our  setting, as we state in the following proposition.
\begin{proposition}\label{prop:impossiblility}
Let $\mathcal{U}$ be a compact set of probability vectors. 
Enumerate the realizations of $X$ as $x_1>\ldots >x_m$ and suppose that $\mathcal{U}$ does not contain the probability vector for which $q_{1}=1$, which is concentrated on $x_1$. 
Then, there exists a continuous distortion function $h$ such that
\[\sup_{\mathbf{q}\in \mathcal{U}}\rho_{h,\mathbf{q}}(X)<\sup_{\mathbf{q}\in \mathcal{U}}\rho_{\hat{h},\mathbf{q}}(X).\]
\end{proposition}
Therefore, it is necessary to study how to reformulate a robust or nominal rank-dependent evaluation (constraint), in the case of a convex or inverse $S$-shaped distortion function. 
We will first study how to reformulate a nominal constraint of the form \eqref{nom_rho_constr}.

\subsection{The Nominal Problem}

Although rank-dependent models do not admit a dual representation for general non-concave distortion functions, we can still express them as inner robust optimization problems if the distortion function is convex or inverse $S$-shaped. 
The idea is to treat the concave and convex parts of the inverse $S$-shaped function separately, where the convex part is transformed into a concave part via the dual function $\bar{h}(p)\triangleq 1-h(1-p)$. 
This result is stated in the following theorem.
Henceforth, the utility function is allowed to be non-concave.

\begin{theorem}\label{thm: inverse-s-shape}
Let $h:[0,1]\to [0,1]$ be an inverse $S$-shaped distortion function with $p^0\in [0,1]$ as in Definition~\ref{def:inverse_S_shape}. 
Define the dual function $\bar{h}(p)\triangleq 1-h(1-p)$. 
Then, we have that, for all $(\mathbf{a},c)\in \mathbb{R}^{n_a+1}$, \eqref{nom_rho_constr} is satisfied if and only if there exist $z\in \mathbb{R}, \bar{\mathbf{q}}\in N^{\text{cv}}_{\bar{h}}(\mathbf{p})$ such that
\begin{align}\label{dual_repr_non_concave}
\begin{cases}
    \displaystyle z+ \sum^m_{i=1}-\bar{q}_iu(f(\mathbf{a},\mathbf{x}_i))\leq c\\
    \displaystyle\sup_{\mathbf{q}\in M^{\text{ca}}_h(\mathbf{p})}\sum^m_{i=1}-q_iu(f(\mathbf{a},\mathbf{x}_i))\leq z,   
\end{cases}
\end{align}
where
\begin{align}\label{uc_concave}
M^{\text{ca}}_h(\mathbf{p})\triangleq\left\{\mathbf{q}\in \mathbb{R}^m \;
    \begin{tabular}{|l}
      $q_i\geq 0$\\
      $\sum_{i\in J}q_{i}\leq h\left(\sum_{i\in J} p_i\right),~\forall J\subset [m]:\sum_{i\in J} p_i\leq p^0 $\\
      $\sum^m_{i=1}q_i=h(p^0)$
\end{tabular}
\right\},
\end{align}
and 
\begin{align}\label{uc_convex}
N^{\text{cv}}_{\bar{h}}(\mathbf{p})\triangleq\left\{\bar{\mathbf{q}}\in \mathbb{R}^m \;
    \begin{tabular}{|l}
      $\bar{q}_i\geq 0$\\
      $\sum_{i\in J}\bar{q}_{i}\leq \bar{h}\left(\sum_{i\in J} p_j\right),~\forall J\subset [m]:\sum_{i\in J} p_i\leq 1-p^0$\\
      $\sum^m_{i=1}\bar{q}_i=\bar{h}(1-p^0)$
\end{tabular}
\right\}.
\end{align}
\end{theorem}

\begin{remark}\label{rem: convex_h}
    In particular, if $h$ is convex, we have that
    \begin{align}
        \rho_{u,h,\mathbf{p}}(f(\mathbf{a},\mathbf{X}))=-\sup_{\bar{\mathbf{q}}\in N^{\text{cv}}_{\bar{h}}(\mathbf{p})}\sum^m_{i=1}\bar{q}_iu(f(\mathbf{a},\mathbf{x}_i)),
    \end{align}
    with $p^0\equiv 0$.
\end{remark}
Theorem~\ref{thm: inverse-s-shape} shows that \eqref{nom_rho_constr} can be reformulated into rank-independent constraints~\eqref{dual_repr_non_concave}, where the supremum constraint can be reformulated further as in Theorem~\ref{thm:RC}. 
However, both \eqref{uc_concave} and \eqref{uc_convex} still feature an exponential number of constraints. 
As outlined in the previous section, we can circumvent this using piecewise-linear approximation. 
Let $h$ be a piecewise-linear, inverse $S$-shaped distortion function with its concave and convex parts specified by the concave functions $h_l,\bar{h}_l$ (note that $\bar{h}_l$ is the dual of the convex part of $h$), respectively, on the domains $[0,p^0]$ and $[0,1-p^0]$. 
Due to concavity and monotonicity, they can be expressed as minima of affine functions:
\begin{align}\label{pw_concave_convex}
    h_l(p)=\min\{l^{(1)}_1p+b^{(1)}_1,\ldots, l^{(1)}_{K_1}p+b^{(1)}_{K_1}\},\quad \bar{h}_l(p)=\min\{l^{(2)}_1p+b^{(2)}_1,\ldots, l^{(2)}_{K_2}p+b^{(2)}_{K_2}\},
\end{align}
defined on the support points
\begin{align*}
    0=s^{(1)}_0\leq s^{(1)}_1,\ldots,s^{(1)}_{K_1-1}\leq s^{(1)}_{K_1}= p^0,\quad 0=s^{(2)}_0\leq s^{(2)}_1,\ldots,s^{(2)}_{K_2-1}\leq s^{(2)}_{K_2}=1-p^0,
\end{align*}
such that, for all $k\in \{1,\ldots, K_1\}, v\in \{1,\ldots,K_2\}$,
\begin{align}\label{condition:pl_h}
    h_l(p)=l^{(1)}_kp+b^{(1)}_k, \forall p\in [s^{(1)}_{k-1},s^{(1)}_k],~\bar{h}_l(p)=l^{(2)}_vp+b^{(2)}_v, \forall p\in [s^{(2)}_{v-1},s^{(2)}_v].
\end{align}
The following theorem establishes a reformulation of the constraints in \eqref{dual_repr_non_concave}, when $h$ is piecewise-linear.
\begin{theorem}\label{thm: bilinear_nom}
    Let $h$ be a piecewise-linear, inverse $S$-shaped distortion function with its concave and convex part specified by $h_l,\bar{h}_l$ as in \eqref{pw_concave_convex}. 
    Then, for any $(\mathbf{a},c)\in \mathbb{R}^{n_a+1}$, \eqref{nom_rho_constr} is satisfied if and only if there exist variables $\lambda_{ik},\nu_k,t_{ik},\bar{q}_i\geq 0$, $\beta\in \mathbb{R}$ such that 
    \begin{align}\label{RC: invS}
    \begin{cases}
        \beta\cdot h(p^0)+\sum^{K_1}_{k=1}\nu_kb^{(1)}_k+\sum^m_{i=1}\sum^{K_1}_{k=1}\lambda_{ik}l^{(1)}_kp_i-\sum^m_{i=1}\bar{q}_iu(f(\mathbf{a},\mathbf{x}_i))\leq c\\
    -u(f(\mathbf{a},\mathbf{x}_i))-\beta-\sum^{K_1}_{k=1}\lambda_{ik}\leq 0,~\forall i\in [m]\\
        \lambda_{ik}\leq \nu_k,~\forall i\in [m],~\forall k \in [K_1]\\
        \bar{q}_i\leq l^{(2)}_kp_i+t_{ik},~\forall i\in [m],~ \forall k\in [K_2]\\
        \sum^{m}_{i=1}t_{ik}\leq b^{(2)}_k,~\forall k\in [K_2] \\ 
        \sum^m_{i=1}\bar{q}_i=\bar{h}(1-p^0).
    \end{cases}
\end{align}
\end{theorem}
Importantly, this yields a 
problem with only $O(m\cdot K_1\vee K_2)$ constraints.
We note that, not surprisingly, \eqref{RC: invS} contains a non-convex constraint due to the product term $\bar{q}_iu(f(\mathbf{a},\mathbf{x}_i))$. 
Hence, the computation of~\eqref{nom_rho_constr} with the reformulated constraints in \eqref{RC: invS} requires a mixed-integer nonlinear programming (MINLP) solver, such as BARON that uses the branch and reduce search method to obtain a global optimum (see e.g., \citealp{ryoo1996branch}). 
In particular, if the set $\mathcal{A}$ is polyhedral (i.e., $\mathcal{A}= \{\mathbf{a}\in \mathbb{R}^{n_a}:\mathbf{D}\mathbf{a}\leq \mathbf{d}\}$), $f(\mathbf{a},\mathbf{x})=\mathbf{a}^T\mathbf{x}$ and $u$ is piecewise-linear, then the nomimal problem~\eqref{nom_rho_constr} with constraints \eqref{RC: invS} can also be solved efficiently by Gurobi (\citealp{gurobi}), using bilinear and special ordered sets of type 2 (SOS2) constraints; see the precise formulation in~\eqref{RC: invS_SOS2} in Electronic Companion~\ref{app:SOS2}, further elaboration on the SOS2 constraints in \eqref{sos2_nonconcave}, and a numerical study on the performance of Gurobi in this setting in Section~\ref{subsec:portfolio_inv_S}. 
\subsection{Robust Rank-Dependent Models with Inverse \texorpdfstring{$S$}{TEXT}-Shaped Distortion}
In this subsection, we study the robust constraint~\eqref{rob_rho_constr}
when $h$ is inverse-$S$-shaped. 
We note that then \eqref{rob_rho_constr} cannot be reformulated using a composite uncertainty set as in \eqref{rob_constr_ref}, since the convex part of the distortion function gives rise to a sup-inf term when combining the set $\mathcal{D}_\phi(\mathbf{p},r)$ with \eqref{uc_convex}.
However, we can still circumvent this using a suitable cutting-plane approach, where we iteratively solve a nominal problem using \eqref{RC: invS}, and compute the robust rank-dependent evaluation for each nominal solution. 

Indeed, for any given feasible solution $\mathbf{a}_*\in \mathcal{A}$, we can assess the ranking of the outcomes $u(f(\mathbf{a}^T,\mathbf{x}_{(1)}))\geq\ldots\geq u(f(\mathbf{a}^T,\mathbf{x}_{(m)}))$. 
Then, we calculate the robust rank-dependent evaluation:
\begin{align}\label{robustcheck_non_concave}
   \sup_{\mathbf{q}\in \mathcal{D}_\phi(\mathbf{p},r)}\sum^{m}_{i=1}h\left(\sum^m_{k=i}q_{(k)}\right)\left(u(f(\mathbf{a}_*),\mathbf{x}_{(i-1)}))-u(f(\mathbf{a}_*,\mathbf{x}_{(i)}))\right).
\end{align}
Problem~\eqref{robustcheck_non_concave} is, quite naturally, non-convex and can again be computed using a solver such as BARON. 
If $h$ is piecewise-linear and the divergence function $\phi$ is linear or quadratic, then one can also solve \eqref{robustcheck_non_concave} using SOS2 constraints in Gurobi.
Specifically, given a set of support points $(t_j,h(t_j))^K_{j=0}$, where $0=t_0<t_1<\cdots < t_{K-1}<t_K=1$, the SOS2 constraints can be formulated as follows:
\begin{align}\label{sos2_nonconcave}
    \begin{split}
\max_{\substack{\mathbf{q}\in \mathcal{D}_\phi(\mathbf{p},r)\\ \boldsymbol{\lambda}\in \mathbb{R}^{K\times m}\\
\mathbf{z}\in \mathbb{R}^m}}&~-u(f(\mathbf{a}_*,\mathbf{x}_{(1)}))+\sum^m_{i=2}z_{i}(u(f(\mathbf{a}_*,\mathbf{x}_{(i-1)}))-u(f(\mathbf{a}_*,\mathbf{x}_{(i)})))\\
\text{subject to}&~ \sum^m_{k=i}q_{(k)} = \sum^K_{j=1}\lambda_{ij}t_j,~ \forall i\in \{2,\ldots, m\}\\
                 &z_{i} = \sum^K_{j=1}\lambda_{ij}h(t_j),~\forall i\in  \{2,\ldots, m\}\\
                 &\sum^K_{j=1}\lambda_{ij}=1,~\forall i\in [m]\\
                 &\lambda_{ij}\geq 0, \mathrm{SOS2},~\forall i\in [m],~\forall j\in [K]. 
    \end{split}
\end{align}
The SOS2 constraints control, for each $i\in [m]$, the variables $\{\lambda_{ij}\}^K_{j=1}$ in a way such that only two adjacent $\lambda_{ij},\lambda_{i,j+1}$ can be non-zero. 
Since the support points $\{t_j\}^K_{j=0}$ are ordered, the two non-zero adjacent variables $\lambda_{ij},\lambda_{i,j+1}$ ensure that we are only optimizing within each interval $[t_j,t_{j+1}]$.

Thus, we can now introduce Algorithm~\ref{algo: cutting_plane_inv_S}, the cutting-plane approach that enables us to solve the robust problem \eqref{P} for piecewise-linear, inverse $S$-shaped distortion functions.

\begin{algorithm}[htp]
\caption{\textbf{Cutting-Plane Method with Inverse $S$-Shaped Distortion Function}}\label{algo: cutting_plane_inv_S}
\begin{algorithmic}[1]
\State Start with $\mathcal{U}_1=\{\mathbf{p}\}$. Fix a tolerance parameter $\epsilon_{\mathrm{tol}}>0$.
\State \label{step2_invS}At the $j$-th iteration, solve the following problem with the uncertainty set $\mathcal{U}_j$, to obtain a solution $\mathbf{a}_j$ and optimal objective value $c_j$:
\begin{align}\label{cut_iter_invS}
    \begin{split}
        \min_{\mathbf{a}\in \mathcal{A}}\sup_{\mathbf{q}\in \mathcal{U}_j}\rho_{u,h,\mathbf{q}}(f(\mathbf{a},\mathbf{X})),
    \end{split}
\end{align}
by applying the reformulation in Theorem~\ref{thm: bilinear_nom} to each element of $\mathcal{U}_j$.
\State \label{step3_invS}Determine the robust rank-dependent evaluation $v_j$ of $\mathbf{a}_j$ by solving \eqref{sos2_nonconcave}, which gives an optimal solution $\mathbf{q}^*_j$.
\State \label{step4_invS}If $v_j-c_j\leq \epsilon_{\mathrm{tol}}$, then the solution $\mathbf{a}_j$ is accepted and the process terminated.
\State \label{step5_invS}If not, set $\mathcal{U}_{j+1} = \mathcal{U}_j \cup \{\mathbf{q}^*_j\}$, and repeat steps \ref{step2_invS}--\ref{step5_invS}.
\end{algorithmic}
\end{algorithm}
The following theorem establishes that this cutting-plane method terminates after finitely many iterations.
\begin{theorem}\label{thm: cut_end_invS}
    Let $h$ be an inverse $S$-shaped, piecewise-linear distortion function. 
    Suppose that $\sup_{\mathbf{a}\in \mathcal{A}, i\in [m]}|u(f(\mathbf{a},\mathbf{x}_i))|<\infty$. 
    Then, for all $\epsilon_{\mathrm{tol}}>0$, Algorithm~\ref{algo: cutting_plane_inv_S} terminates after finitely many iterations.
\end{theorem}
Next, we examine the convergence of the piecewise-linear approximation to problems~\eqref{P-Nom} and~\eqref{P}. 
In conjunction with Theorem~\ref{thm: cut_end_invS}, it constitutes, in a sense, the master theorem of the paper.
\begin{theorem}\label{thm: invS_PL_error}
     Let $h$ be an inverse $S$-shaped distortion function and let $h_\epsilon$ be 
    such that $\sup_{p\in [0,1]}|h(p)-h_\epsilon(p)|\leq \epsilon$.
    Suppose that $\sup_{\mathbf{a}\in \mathcal{A}, i\in [m]}|u(f(\mathbf{a},\mathbf{x}_i))|<\infty$. 
    Then, the optimal objective value of \eqref{P-Nom} for $h_\epsilon$ converges to that of $h$, as $\epsilon\to 0$.

    Similarly, for the robust problem~\eqref{P}, Algorithm~\ref{algo: cutting_plane_inv_S}, if applied to $h_\epsilon$, yields an optimal objective value that converges to the exact optimal objective value of \eqref{P} for $h$, as both $\epsilon_{\mathrm{tol}},\epsilon\to 0$. 
\end{theorem}
We note that when $h$ is an inverse-$S$-shaped distortion function, the piecewise-linear approximation specified by $\{h_l,\bar{h}_l\}$ in \eqref{pw_concave_convex} provides neither a lower bound nor an upper bound on $h(p)$, for all $p\in[0,1]$. 
Indeed, $h_l$ is a lower approximation of $h$ on $[0,p^0]$, whereas $1-\bar{h}_l(1-p)$ is an upper approximation of $h$ on $[p_0,1]$. 
Nonetheless, if one aims to bound $h$ on the entire $[0,1]$, then we can apply the same device as in \eqref{upper_distortion}: either translate $h_l$ with an $\epsilon$ (the maximal error) to obtain an upper piecewise-linear approximation of $h$ on $[0,p^0]$, or do the same with the dual function $\bar{h}_l$ for a lower approximation on $[p^0,1]$.
Theorem~\ref{thm: invS_PL_error} then implies that the corresponding lower and upper bounds on the optimal objective value of \eqref{P-Nom} (or \eqref{P}), computed from these piecewise-linear approximations of $h$, will converge to the exact optimal value, as $\epsilon\to 0$.

Finally, we note that Algorithm~\ref{algo: cutting_plane_inv_S} can also be adapted (similar to Algorithm~\ref{Algo: cutting_plane_constraint} in the Electronic Companion) to solve problem \eqref{P-constraint}, where the rank-dependent evaluation with inverse-$S$-shaped distortion function is in the constraint. 
Hence, we also present the following theorems with similar statements as in Theorems~\ref{thm: cut_end_invS} and~\ref{thm: invS_PL_error}.

\begin{theorem}\label{thm: inv_S_constraint_cut}
    Let $h$ be an inverse $S$-shaped, piecewise-linear distortion function. 
    Suppose that $u(f(\mathbf{a},\mathbf{x}_i))$ is continuous in $\mathbf{a}\in \mathcal{A}$ for all $i\in [m]$. 
    Then, for any $\epsilon_{\mathrm{tol}}>0$, Algorithm \ref{algo: cutting_plane_inv_S} terminates after finitely many iterations when applied to \eqref{P-constraint}. 
    Furthermore, if $h$ is also continuous, then the final solution $\mathbf{a}_{\epsilon_{\mathrm{tol}}}$ of Algorithm~\ref{algo: cutting_plane_inv_S} gives a lower bound $g(\mathbf{a}_{\epsilon_{\mathrm{tol}}})$ that converges to the optimal objective value of \eqref{P-constraint} as $\epsilon_{\mathrm{tol}}\to 0$. 
\end{theorem}

\begin{theorem}\label{thm: inv_S_constraint}
    Let $h$ be an inverse $S$-shaped distortion function. 
    Suppose that $u(f(\mathbf{a},\mathbf{x}_i))$ is continuous in $\mathbf{a}\in \mathcal{A}$ for all $i\in [m]$. 
    If $g$, appearing in the objective, and $h$ are continuous, then for any distortion function $h_\epsilon$ such that $h_{\epsilon}(p)\leq h(p)$ for all $p\in[0,1]$ and $\sup_{p\in [0,1]}|h(p)-h_{\epsilon}(p)|\leq \epsilon$, the optimal objective value of \eqref{P-constraint} for $h_\epsilon$ converges to that of $h$, as $\epsilon\to 0$.
\end{theorem}

\section{Numerical Examples}\label{sec: NumericalExp}
In this section, we illustrate the various methods developed in this paper by applying them to two canonical examples of optimization problems: the \emph{newsvendor} problem and the \emph{portfolio choice} problem. 
In both examples, we compare the solution of the robust problem \eqref{P} to the solution of the nominal problem \eqref{P-Nom}. 
We use the phrase ``robust/nominal solution'' to refer to the solution of the robust/nominal problem.
\subsection{Robust Single-Item Newsvendor}
In the single-item newsvendor problem, a seller is uncertain about the demand for a certain product, and has to decide in advance how many units of the product have to be ordered and stocked in the inventory. 
Let $d_i\geq 0$ denote the realization of the demand in state $i$. 
Furthermore, let: $c$ be the ordering cost of one unit, $v>c$ be the selling price, $s<c$ be the salvage value per unsold item returned to the factory, and $l$ be the loss per unit of unmet demand, which may include both the cost of a lost sale and a penalty for the lost customer goodwill. 
Finally, let $y$ be the number of items ordered, i.e., the decision variable.
The profit function $\pi(d_i,y)$ is defined as
\begin{align}
    \pi(d_i,y)&\triangleq v\min\{d_i,y\}+s(y-d_i)_+-l(d_i-y)_+-cy\nonumber\\
    &= (s-v)(y-d_i)_+-l(d_i-y)_++(v-c)y.
\end{align}

\begin{figure}[t]
    \centering
     \caption{{\small Single-item newsvendor problem. 
     This figure displays the worst-case evaluation $\mathrm{WC}(\alpha_0)\triangleq\sup_{\mathbf{q}\in \mathcal{D}_{\phi}(\mathbf{p},r(n))}\mathrm{CVaR}_{1-\alpha_0}(.)$ under the robust and nominal solutions, for a range of values of $r(n) = \chi^2_{2,0.95}/(2n)$ and $\alpha_0= 0.4, 0.3, 0.2, 0.1$.}}\label{fig: robvsnonrob_snews}  
    \subfloat{%
  \includegraphics[width=0.4\textwidth]{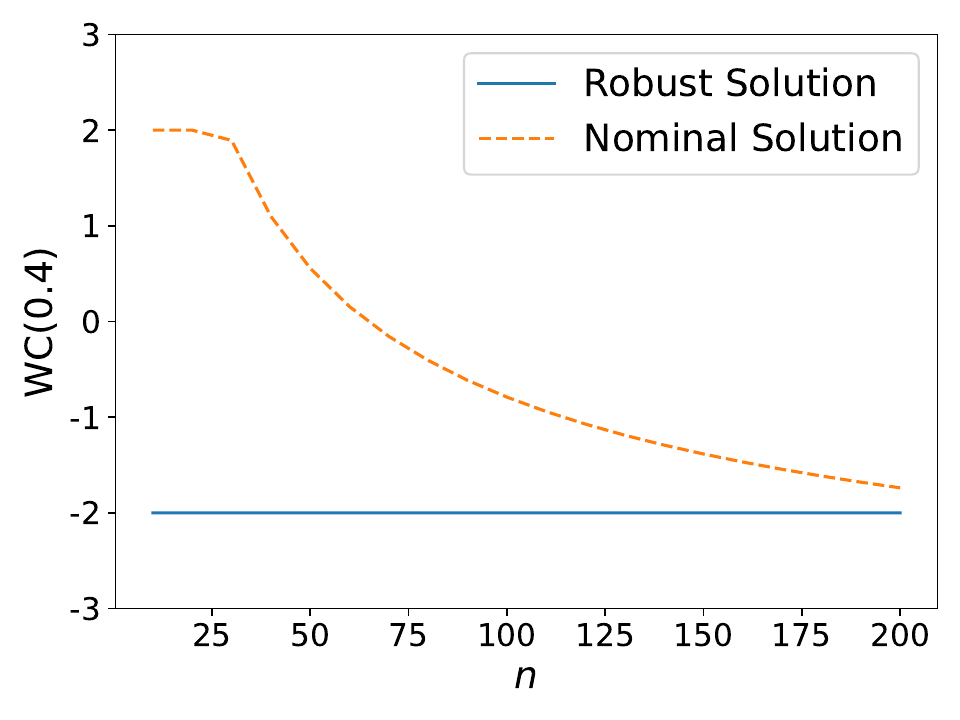}%
}%
\subfloat{%
  \includegraphics[width=0.4\textwidth]{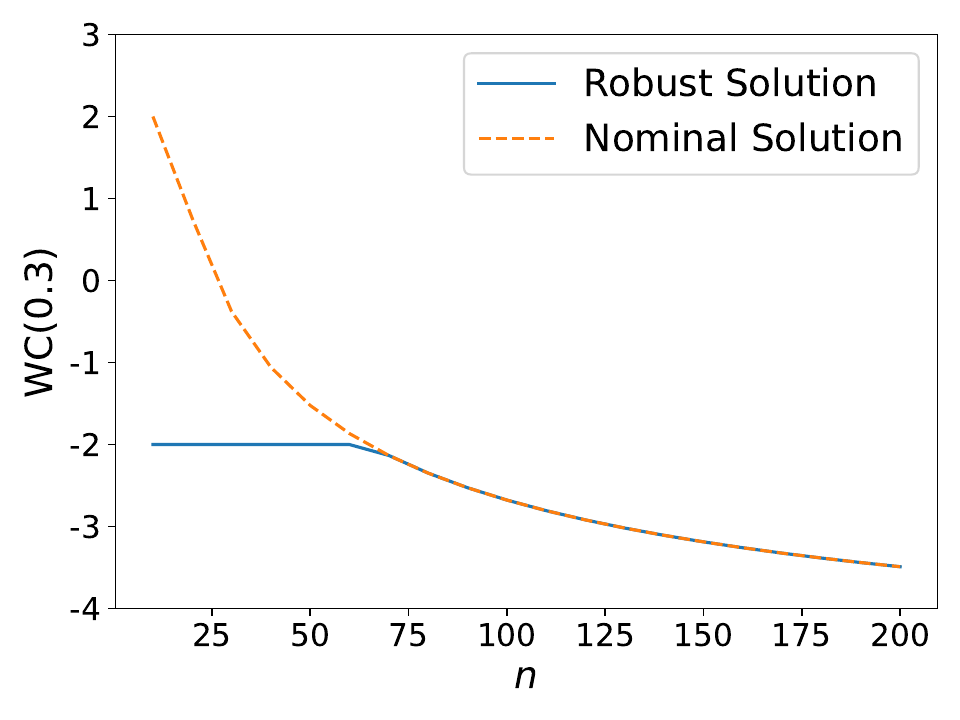}%
}

\subfloat{%
  \includegraphics[width=0.4\textwidth]{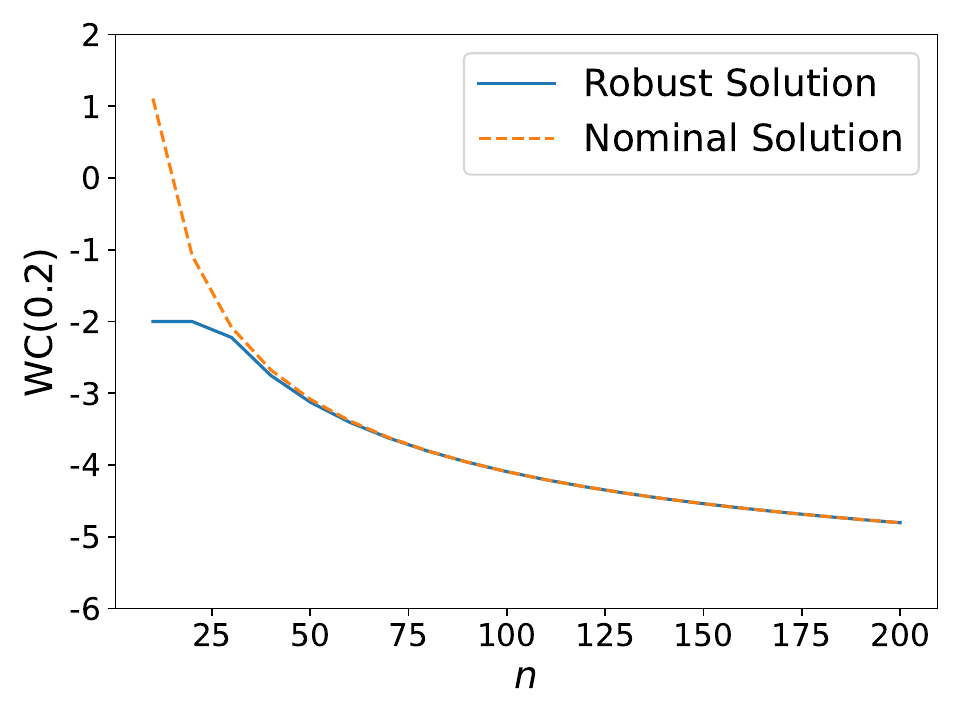}%
}
\subfloat{%
  \includegraphics[width=0.4\textwidth]{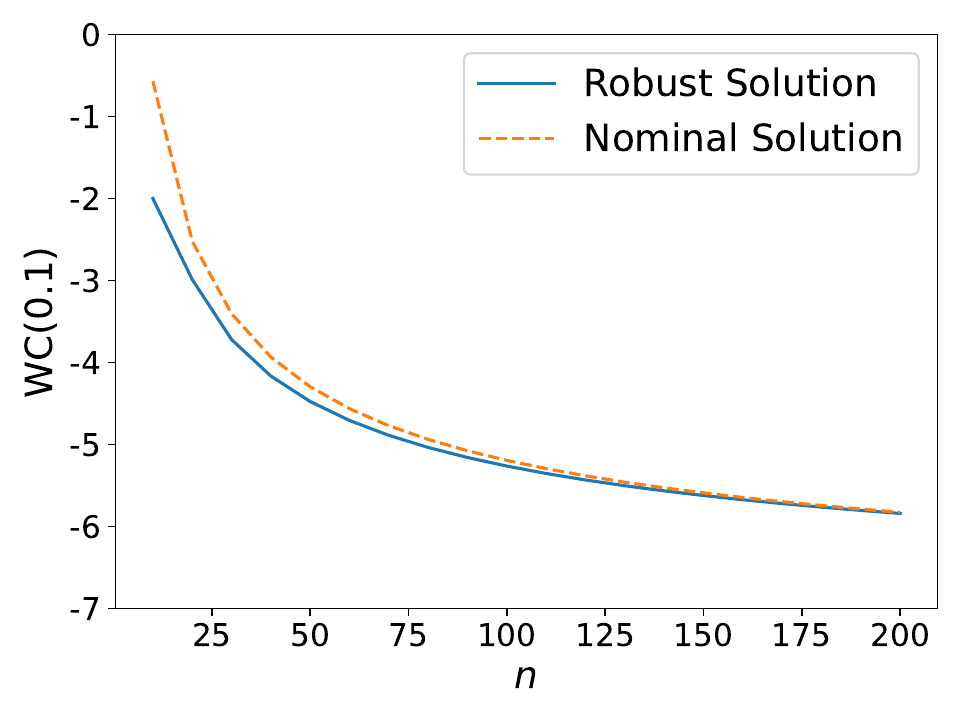}%
}
\end{figure}
Note that $\pi(d_i,y)$ is concave in $y$ since it is a sum of concave piecewise-linear functions in $y$ for all $i$. 
Assume that the demand $d_i\in \{4,8,10\}$, which corresponds to low, medium, and high demand, with nominal probabilities $\mathbf{p}=\{0.375,0.375,0.25\}$.  
Furthermore, we assume that the number of items ordered will not exceed the maximum demand, i.e., $y\leq 10$. 
Finally, we set the parameter values $c=4,v=6,s=2,l=4$.

We solve both the robust problem \eqref{P} and the nominal problem $\eqref{P-Nom}$, where the decision variable $y$ is subject to the constraint $0\leq y\leq 10$, $\mathbf{d}$ is the uncertain parameter, and $\pi(\mathbf{d},y)$ is the random objective function.
The nominal and robust newsvendor problem is studied in \cite{EeckhoudtMS} and \cite{BenTal11}, respectively, assuming expected utility preferences.
We choose $\mathrm{CVaR}_{1-\alpha_0}$ (see \eqref{def:CVaR}) to be our rank-dependent evaluation $\rho_{u,h,\mathbf{q}}$. 
This corresponds to a piecewise-linear distortion function $h(p)=\min\{p/(1-\alpha_0),1\}$ and a linear utility function $u(x)=x$. 
For the robust problem \eqref{P}, we choose the KL-divergence function $\phi(t)=t\log(t)-t+1$, and set the radius of the divergence set $\mathcal{D}_{\phi}(\mathbf{p},r)$ to be $r=\chi^2_{2,0.95}/(2n)$ as in \eqref{r_choice}, where $n$ is the 
sample size that we assume that $\mathbf{p}$ is estimated from.
A larger value of the objective function indicates a larger (i.e., worse) evaluation of the utility loss; a value of $\alpha_0=0.4$ means that the worst 60\% of the utility loss distribution is considered. 

As the cardinality of  realizations of the uncertain parameter $d$ is merely 3, we can readily invoke Theorem~\ref{thm:RC} to solve both the reformulated nominal and robust problems \eqref{P-ref} and \eqref{P-Nom-ref} exactly, without reducing the number of constraints. 
We then perform the following experiment: 
For each $\alpha_0$, we first obtain a nominal solution to \eqref{P-Nom}. 
Then, for each radius $r(n)=\chi^2_{2,0.95}/(2n)$, we obtain a robust solution to \eqref{P}. 
To compare the robust solution with the nominal solution, we calculate their worst-case evaluation $\sup_{\mathbf{q}\in \mathcal{D}_{\phi}(\mathbf{p},r(n))}\mathrm{CVaR}_{1-\alpha_0}(.)$ under the radius $r(n)$. 
This is repeated for a range of $n= 10, 20,\ldots, 200$, and $\alpha_0 = 0.4, 0.3, 0.2, 0.1$. 
As $n$ increases, the $\phi$-divergence radius $r$ decreases, and thus there is less ambiguity. 

The results are displayed in Figure~\ref{fig: robvsnonrob_snews}.
Quite naturally, we observe that the differences between the worst-case evaluations of both solutions decrease as $n$ increases. 
We also observe a decrease in the value of the worst-case evaluation as $\alpha_0$ decreases, reflecting that the $\mathrm{CVaR}_{1-\alpha_0}$ risk measure itself is monotonically decreasing in $\alpha_0$. 
Furthermore, we see that for lower values of $\alpha_0$, the differences between the worst-case evaluations of both solutions are already small for small sample size $n$. 
Moreover, we observe a flat, constant worst-case evaluation of the robust solution at level $-2$. 
This is the largest (i.e., worst) value that the robust rank-dependent evaluation attains for a robust solution $y^*=7$. 
This conservativeness occurs when $n$ is relatively small.

To further illustrate the differences between the robust and nominal solutions, we use the Hit-and-Run algorithm (see Electronic Companion~\ref{App:HitandRun} for further details) to sample $5\mathord{,}000$ probability vectors $\mathbf{q}$ from the KL-divergence uncertainty set for $n=50$. 
For each sampled vector $\mathbf{q}$, we calculate the $\mathrm{CVaR}_{1-\alpha_0}$ evaluation, 
with $\alpha_0=0.4$.
\begin{figure}[t]
    \centering
    \caption{{\small Single-item newsvendor problem (\textit{continued}): 
    This figure displays, for each sampled probability vector $\mathbf{q}$ from the KL-divergence uncertainty set for $n=50$, the corresponding $\mathrm{CVaR}_{1-\alpha_0}(.)$ evaluation with $\alpha_0=0.4$ of the robust and nominal solutions.}}\label{fig: robvsnonrob_cvar_news} 
     \includegraphics[height=2.2in]{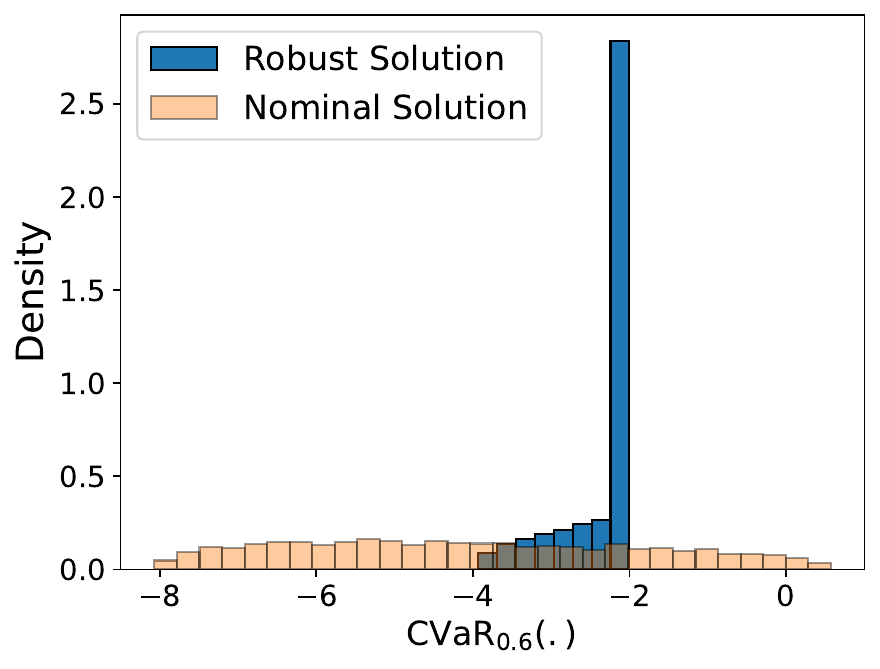}
\end{figure}
As shown in Figure~\ref{fig: robvsnonrob_cvar_news}, the evaluation of the nominal solution exhibits a large variance. 
It also exceeds the largest (i.e., worst) evaluation of the robust solution. 
On the other hand, the robust solution shows less variance and is concentrated at the value $-2$, which as mentioned above, is the most conservative evaluation that the robust optimal solution attains.

\subsection{Robust Multi-Item Newsvendor}
We next examine the multi-item newsvendor problem, where each item in a set of items features its own uncertain demand.
Let $d^{(j)}_{i}$ be the $i$-th realization of the $j$-th item's demand, $i=1,2,3$. 
We take $j=1,2,3$ and consider the sum of the individual profit functions
\begin{align}
    \pi_{\mathrm{tot}}(\mathbf{d},\mathbf{y})=\pi_1(d^{(1)},y_1)+\pi_2(d^{(2)},y_2)+\pi_3(d^{(3)},y_3),
\end{align}
where
\begin{align}
     \pi_j(d^{(j)},y_j)&=v_j\min\{d^{(j)},y_j\}+s_j(y_j-d^{(j)})_+-l_j(d^{(j)}-y_j)_+-c_jy_j,
\end{align}
with $v_j,s_j,l_j,c_j$ the parameters corresponding to item $j$. 
We assume that the demand takes on the same possible values for all items, i.e., $d^{(j)}_{i}\in \{4,8,10\}$ for all $j$. 
Since each realization of $d^{(j)}$ contributes to a possible realization of $\pi_{\mathrm{tot}}(\mathbf{y},\mathbf{d})$, there are in total $m=3^3=27$ possible realizations. 
We solve again problems~\eqref{P} and~\eqref{P-Nom} with the same preference specification as in the single-item problem.
Since the $\mathrm{CVaR}_{1-\alpha_0}$ risk measure has a piecewise-linear distortion function, we can apply Theorem~\ref{thm:RCpiecewiselin} to solve \eqref{P-ref} and \eqref{P-Nom-ref} in this higher-dimensional setting.
We take the nominal probability $\mathbf{p}\in \mathbb{R}^{27}$ to be the probability of each combination of realizations of $(d^{(1)},d^{(2)},d^{(3)})$, which is the product of the probabilities $\mathbf{p}^{(1)},\mathbf{p}^{(2)},\mathbf{p}^{(3)}$ of each individual realization. 
The parameters we use 
are given in Table~\ref{tab: mnews_param}.
\begin{table}[t]
    \centering
    \caption{{\small Multi-item newsvendor problem. 
    This table displays the parameters used in the multi-item newsvendor problem.}}
    \label{tab: mnews_param}
    \begin{tabular}{c|cccc|ccc}
         $\mathrm{Item}(j)$ &$c$& $v$ & $s$ & $l$ & $p^{(j)}_1$ & $p^{(j)}_2$ & $p^{(j)}_3$ \\
        \hline
         1&4 & 6 & 2 & 4 & 0.375 & 0.375 & 0.25 \\
         2&5 & 8 & 2.5 & 3 & 0.25 & 0.25 & 0.5 \\
        3 &4& 5 & 1.5 & 4 & 0.127 & 0.786 & 0.087 
    \end{tabular}
    
\end{table}

Similar to the single-item problem, we investigate the differences between the worst-case evaluations of the robust and nominal solutions, for a range of $n$ and $\alpha_0$. 
We choose $\alpha_0=0.4, 0.3$ to compare the results with the single-item problem, and $\alpha_0=0.9, 0.8$ to explore higher values of $\alpha_0$. 
From Figure~\ref{fig: robvsnonrob_mnews}, we observe that in the multi-item problem, the worst-case evaluation is much lower (i.e., better) than in the single-item problem, which speaks to the diversification benefits of a multi-item inventory. 
The overall pattern is similar to the single-item case: for relatively low $\alpha_0$, the difference in worst-case evaluation between the robust solution and the nominal solution is smaller than for larger $\alpha_0$. 
The same holds as the sample size $n$ increases.
\begin{figure}[t]
    \centering
       \caption{{\small Multi-item newsvendor problem. 
       This figure displays the worst-case evaluation $\mathrm{WC}(\alpha_0)\triangleq\sup_{\mathbf{q}\in \mathcal{D}_{\phi}(\mathbf{p},r(n))}\mathrm{CVaR}_{1-\alpha_0}(.)$ under the robust and nominal solutions, for a range of values of $r(n) = \chi^2_{2,0.95}/(2n)$ and $\alpha_0=0.9, 0.8, 0.4, 0.3$.}}\label{fig: robvsnonrob_mnews}
     \subfloat{%
 \includegraphics[width=0.4\textwidth]{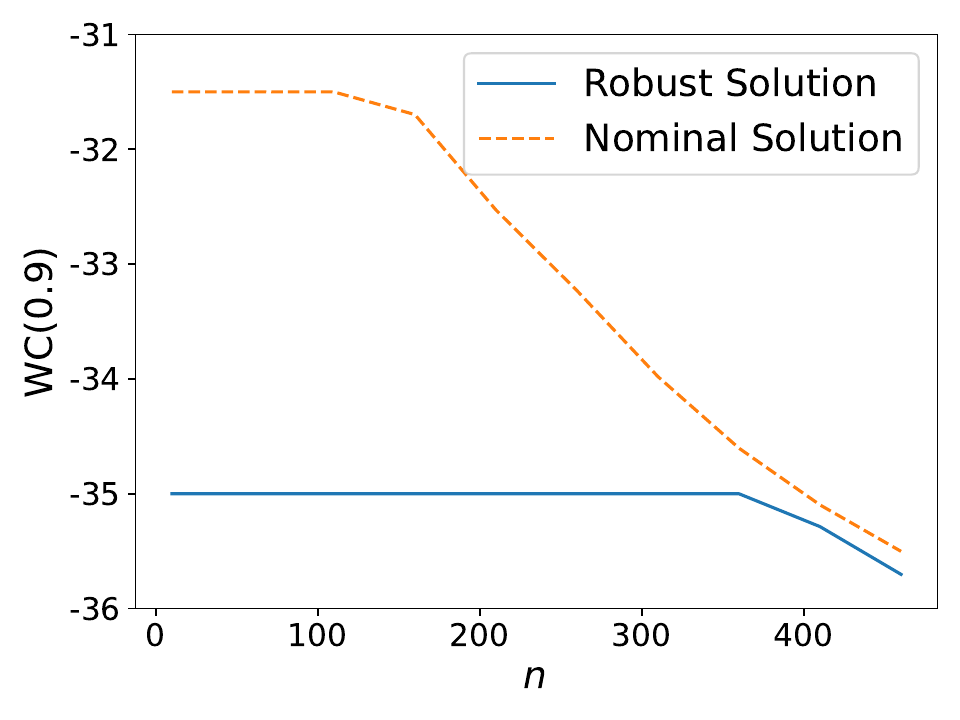}%
}%
\subfloat{%
  \includegraphics[width=0.4\textwidth]{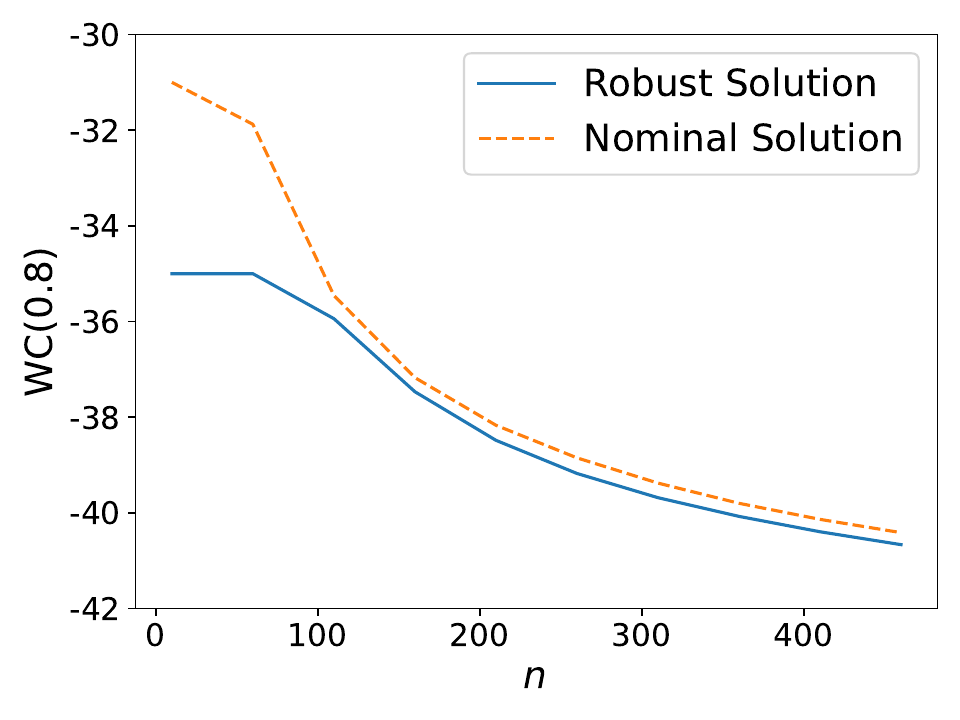}%
}%

\subfloat{%
  \includegraphics[width=0.4\textwidth]{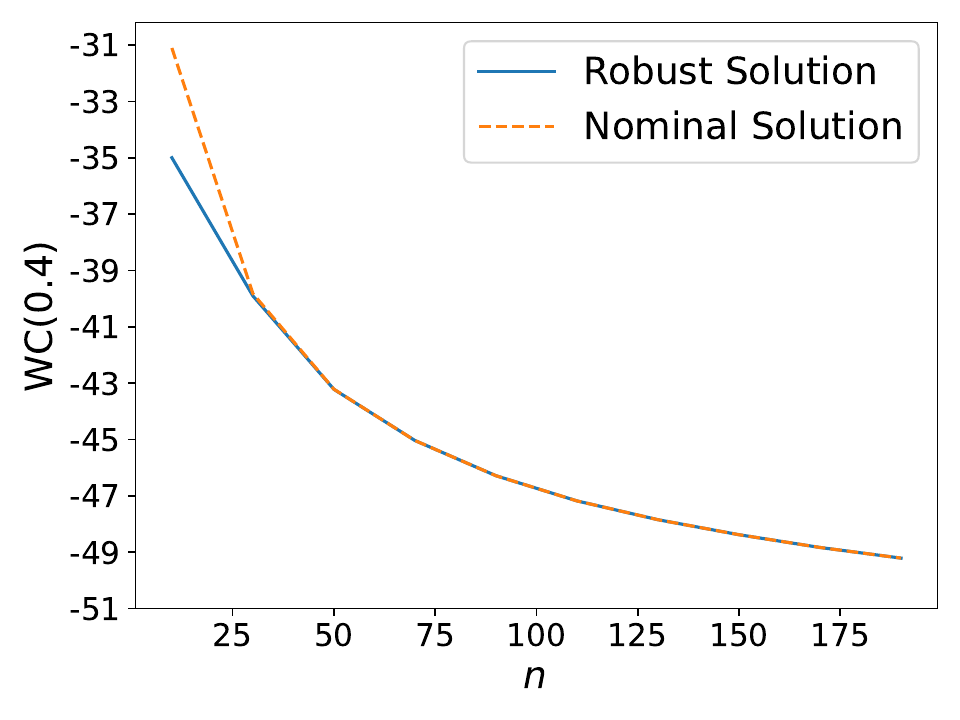}%
}%
\subfloat{%
  \includegraphics[width=0.4\textwidth]{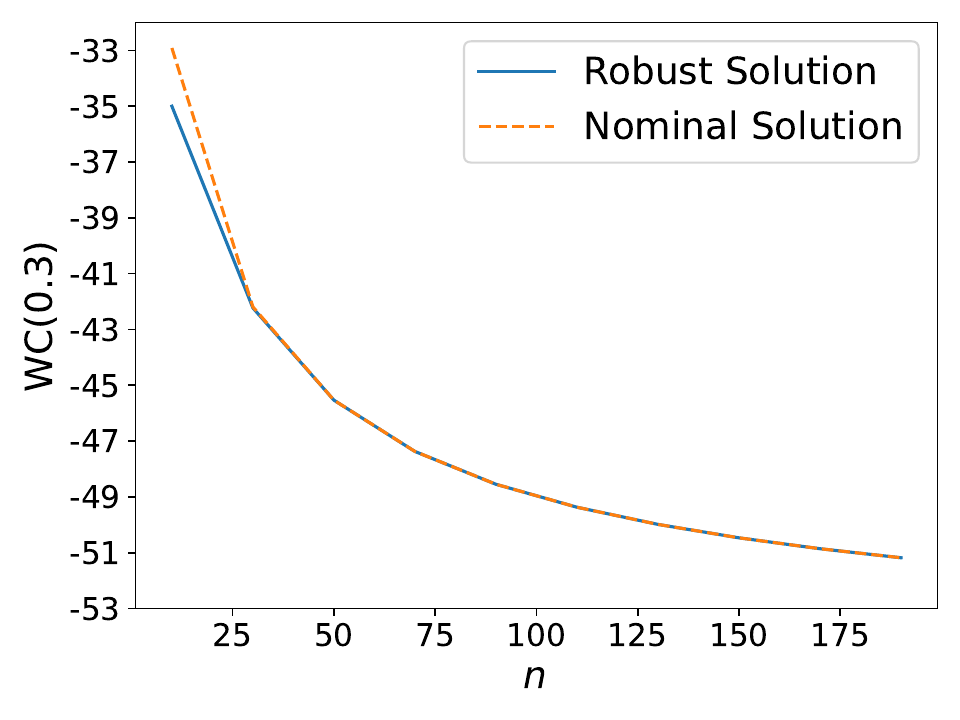}%
}
\end{figure}

\subsection{Robust Portfolio Choice with Concave Distortion Function}
In this subsection, we investigate the performance of the cutting-plane algorithm and the piecewise-linear approximation method described in Algorithms~\ref{cutting_plane_algo} and~\ref{algo_PL_method}, by studying robust portfolio optimization problems.
We consider the returns of six portfolios formed on size and book-to-market ratio (2$\times$3) obtained from Kenneth French's data library.\footnote{\url{https://mba.tuck.dartmouth.edu/pages/faculty/ken.french/data_library.html}} 
We use monthly returns from January 1984 to January 2014. 
This gives us a total of $m=360$ return realizations for each of the six portfolios. 
We denote by $\mathbf{r}\in\mathbb{R}^{6}$ the uncertain return of the six portfolios, with empirical realizations $\mathbf{r}_1,\ldots,\mathbf{r}_{360}$.
We solve the following nominal and robust portfolio optimization problems:
\begin{align}
    \label{P-nom-num1}\min_{\mathbf{a}\in \mathcal{A}}&~\rho_{u,h,\mathbf{p}}(1+\mathbf{a}^T\mathbf{r})\\
    \label{P-rob-num1}\min_{\mathbf{a}\in \mathcal{A}}&\sup_{\mathbf{q}\in \mathcal{D}_{\phi}(\mathbf{p},r)}\rho_{u,h,\mathbf{q}}(1+\mathbf{a}^T\mathbf{r}),
\end{align}
where $1+\mathbf{a}^T\mathbf{r}$ is the total wealth after one period with initial wealth normalized to unity, which is evaluated under the rank-dependent utility model. 
The decision variable $\mathbf{a}\in \mathbb{R}^6$ is subject to the constraints $\mathcal{A}=\{\mathbf{a}\in \mathbb{R}^6~|~ \sum^6_{j=1}a_j=1, a_j\geq 0\}$. 
The nominal probability $\mathbf{p}$ is set to be the empirical distribution with $p_i=\frac{1}{360}$. 
We use the modified chi-squared divergence function $\phi(x)=(x-1)^2$, with radius $r=\frac{1}{n}\chi^2_{359,0.95}$. 
We choose the concave distortion function $h(p)=1-(1-p)^2$, motivated by two decision-theoretic papers: \citet{ELS20} and \citet{EL21}. 
Furthermore, we choose the exponential utility function $u(x)=1-e^{-x/\lambda}$ with $\lambda=10$.

We solve problems~\eqref{P-nom-num1}--\eqref{P-rob-num1} with the cutting-plane method described in Algorithm~\ref{cutting_plane_algo}, and compare its performance to the piecewise-linear method described in Algorithm~\ref{algo_PL_method},
formally relying on Theorems~\ref{thm:cuttingplane_conv} and~\ref{thm:convergencePL}. 
As shown in Table~\ref{tab:ct_plane_PL}, both algorithms yield a very small gap between the upper and lower bounds they provide. 
The piecewise-linear approximation method is typically more efficient, but this is not surprising, since the cutting-plane method is a more general method and utilizes less structure of the problems at hand. 
\begin{table}[H]
\centering
\caption{{\small Robust portfolio choice. 
This table displays the results of Algorithms~\ref{cutting_plane_algo} and~\ref{algo_PL_method} when applied to the robust problem \eqref{P-rob-num1} and the nominal problem \eqref{P-nom-num1}. 
The respective tolerance parameters are set to $\epsilon_{\mathrm{tol}}=0.0001$ and $\epsilon=0.001$. 
}  }\label{tab:ct_plane_PL}
\subfloat{\centering 
\begin{tabular}{ccccc}
\multicolumn{5}{c}{\textbf{Panel A: The cutting-plane method}}\\ \hline
  Problem&  Lower Bound&Upper Bound& \# Cuts&Run Time 
    \\ \hline 
    Robust &-0.08953&  -0.08948&5 &54 sec 
    \\
    Nominal & -0.09403 & -0.09397 &5 & 22 sec 
    \\ 
    \end{tabular}}

\medskip
\subfloat{\centering
\begin{tabular}{ccccc}
\multicolumn{4}{c}{\textbf{Panel B: Piecewise-linear approximation}}\\ \hline
  Problem&  Lower Bound&Upper Bound& Run Time 
    \\ \hline 
    Robust &-0.08951&  -0.08948&  14 sec 
    \\
    Nominal & -0.09401 & -0.09398 &7 sec 
    \\ 
    \end{tabular}}
\end{table}
Additionally, for the cutting-plane algorithm, we calculate the worst-case rank-dependent evaluation of the nominal solution, which is equal to $-0.08938$. 
We also calculate the evaluation of the robust solution under the nominal distribution $\mathbf{p}$, which is equal to $-0.09395$. 
As we can see from the table, the worst-case evaluation of the nominal solution is not much larger than that of the robust solution (cf.\ the first rows of the Panels~A and~B), suggesting that the nominal solution is already ``near-optimal'' for the robust problem \eqref{P-rob-num1}. 
Similarly, the robust solution is also ``near-optimal'' for the nominal problem~\eqref{P-nom-num1}. 
It appears that when the dimension $m$ of the state space is large, the robust and nominal problems do not yield highly different solutions in this example.
  
\subsection{Robust Portfolio Choice with Inverse \texorpdfstring{$S$}{TEXT}-shaped Distortion Function}\label{subsec:portfolio_inv_S}
In this subsection, we investigate portfolio optimization problems \eqref{P-nom-num1}--\eqref{P-rob-num1} when $h$ is an inverse $S$-shaped distortion function.
In particular, we examine the \cite{Prelec} distortion function:
\begin{align}\label{Prelec}
    h_{\alpha}(p)\triangleq 1-\exp(-(-\log(1-p))^\alpha),\quad 0<\alpha<1.
\end{align}
We focus primarily on $\alpha= 0.6,\, 0.75$, which are  broadly consistent with common empirical findings (\citealp{Wakker}, p.\ 260). 
We consider a linear utility function, $u(x)=x$, to isolate the effect of inverse $S$-shaped probability weighting. 
To obtain an upper and lower bound on problems \eqref{P-nom-num1}--\eqref{P-rob-num1}, we approximate Prelec's function from above and below, using piecewise-linear functions. 
The approximation procedure is carried out by applying the method described in Section~\ref{subsec:approx_affine} separately to the concave and convex parts of Prelec's function, where we minimize the number of linear pieces under a pre-specified approximation error.\footnote{The approximation error is chosen to be $\epsilon=0.003$ leading to 19, 13, and 6 linear pieces for $\alpha=0.6,\, 0.75,\, 0.95$.} 
The shape of Prelec's function and its respective upper and lower piecewise-linear approximations are displayed in Figure~\ref{fig: PL_prelec}.  
\begin{figure}[t]
    \begin{minipage}{\textwidth}
    \centering
    \caption[]{{\small Prelec's distortion. This figure displays Prelec's distortion function \eqref{Prelec} and its upper and lower piecewise-linear approximations (dashed) for $\alpha=0.6,\, 0.75$.
    The approximation error is set to $\epsilon = 0.003$. }}
    \label{fig: PL_prelec}
    \vskip-0.43cm
     \includegraphics[height=2.7in]{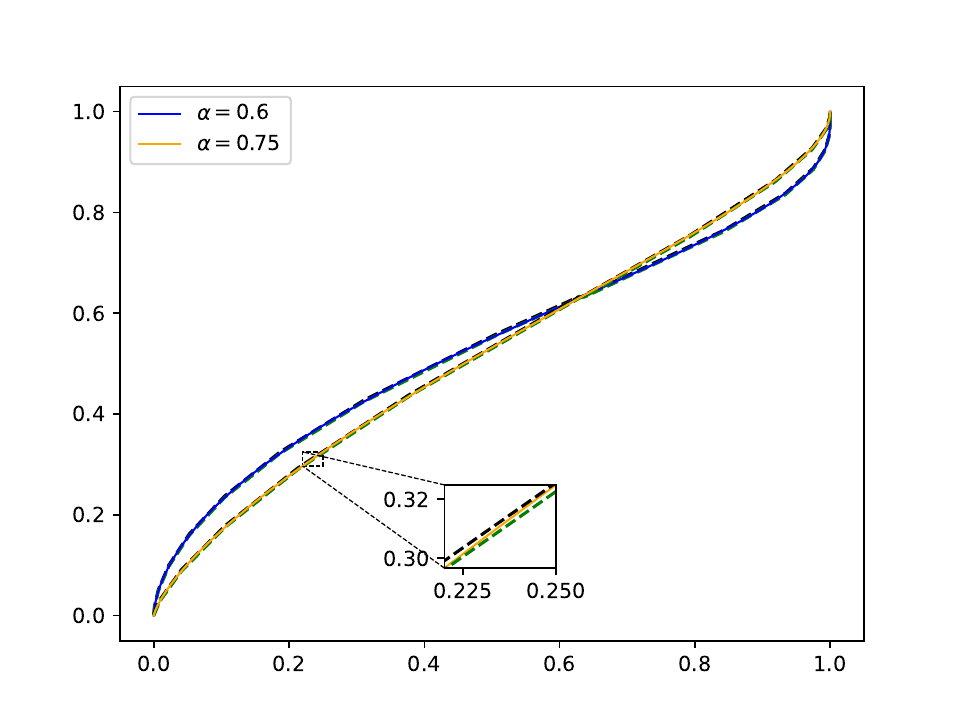}
    \vskip-0.36cm 
   \end{minipage}
\end{figure}

We generate 100 return data for 5 assets using the same numerical scheme as in \citet{EsfahaniKuhn}, where the returns for each asset $j\in \{1,\ldots,5\}$ are composed of two factors:
$r_j=\psi+\gamma_j$,    
with a systematic risk factor $\psi\sim N(0, 0.02)$ and an idiosyncratic risk factor $\gamma_j\sim N(0.03j,0.25j)$, for the $j$-th asset.
Here, $N(\mu,\sigma)$ denotes the Gaussian distribution. 
By construction, the asset with a higher index $j$ has a higher expected return and standard deviation.

We obtain lower and upper bounds on the nominal problem~\eqref{P-nom-num1} by implementing the constraints \eqref{RC: invS} of Theorem~\ref{thm: bilinear_nom} for the lower and upper piecewise-linear approximation of $h$, respectively, in the Gurobi solver (Version 11.0.3). 
Next, we invoke Theorems~\ref{thm: cut_end_invS}--\ref{thm: invS_PL_error}.
Hence, a lower bound for the robust problem \eqref{P-rob-num1} is obtained by calculating the lower bound implied by the cutting-plane method described in Algorithm~\ref{algo: cutting_plane_inv_S} (where we set $\epsilon_{\mathrm{tol}}= 0.001$), while using the lower approximation of $h$. 
Then, with the solution obtained from the cutting-plane method, we determine an upper bound on \eqref{P-rob-num1} by calculating its worst-case evaluation using the SOS2-constraints as in \eqref{sos2_nonconcave}, where we use the upper approximation of $h$. 
For the robust problem, we consider 
the total variation divergence $\phi(t)=|t-1|$. 
We choose the radius for the robust problem 
to be $r=\frac{1}{m}\chi^2_{m-1,0.95}$, where $m=100$. 

The results are displayed in Table~\ref{tab:Prelec_variational_and_quadratic}. 
We observe very tight upper and lower bounds. 
Furthermore, Gurobi solves both the nominal and robust problems \eqref{P-nom-num1}--\eqref{P-rob-num1} efficiently. 
As we can observe from Panel~A, 
the evaluation of the nominal solution decreases with $\alpha$. 
This is not surprising. 
Indeed, as we can see in Figure~\ref{fig: PL_prelec}, Prelec's function has the property that if $\alpha_0<\alpha_1$, then $h_{\alpha_0}(p)>h_{\alpha_1}(p)$ for the concave segment $p\in [0,1-1/e]$, which covers more than half of the interval $[0,1]$.
Interestingly, Panel~B of Table~\ref{tab:Prelec_variational_and_quadratic} reveals that the robust evaluation of the robust solution can exhibit an opposite, mildly increasing relationship with respect to $\alpha$. 
This is because in the robust case, if the radius $r$ is sufficiently large, then the worst-case probability $q^*_{(m)}$ of the worst-case realization can be such that $q^*_{(m)} \geq 1-1/e$ (recall that $(m)$ denotes the index of the highest ranked realization, as in \eqref{robustcheck_non_concave}). 
When this happens, all distorted decumulative probabilities in the objective function of \eqref{robustcheck_non_concave} are increasing in $\alpha$, since then all decumulative probabilities are on the interval $[1-1/e,1]$. 
If the radius $r$ would be set sufficiently small, then the same decreasing pattern with respect to $\alpha$ as in Panel~A would also emerge in the robust case. 

\begin{table}[t]
\centering
\caption{{\small Robust portfolio choice with inverse $S$-shaped distortion function. 
Upper and lower bounds (UB and LB) obtained for the nominal problem \eqref{P-nom-num1} and the robust problem \eqref{P-rob-num1} when $h$ is Prelec's distortion function with $\alpha\in \{0.6,\, 0.75,\, 0.95\}$ (see Figure~\ref{fig: PL_prelec}) and the divergence function is given by $\phi(t)=|t-1|$. 
The cardinality of return realizations is $m=100$.} }\label{tab:Prelec_variational_and_quadratic}
\subfloat{\begin{tabular}{cccccc}
\multicolumn{5}{c}{\textbf{Panel A: Solutions Nominal Problem}}\\  \hline
  $\alpha$&LB&UB&Run Time (LB)& Run Time (UB)& \\ \hline 
    0.6&-1.144&-1.142&1.34 sec&  1.21 sec\\
    0.75 & -1.153&-1.152& 0.44 sec&0.90 sec \\
    0.95 & -1.162&-1.160& 0.43 sec&0.30 sec \\
    \end{tabular}} 

\subfloat{\begin{tabular}{cccccc}
\multicolumn{6}{c}{\textbf{Panel B: Solutions Robust Problem}}\\  \hline
  $\alpha$&LB & UB&\# Cuts& Run Time (LB)&Run Time (UB)  \\ \hline 
    0.6&-1.041 &-1.040&9&  314 sec&2.77 sec\\
    0.75 &-1.039 &-1.038&9&142 sec& 1.45 sec\\
    0.95 & -1.036 &-1.035& 9 & 98 sec & 0.37 sec
    \end{tabular}} 

\end{table}

\section{Concluding Remarks} \label{sec: Conclusion}
In this paper, we have shown that nominal and robust optimization problems involving rank-dependent models can be reformulated into rank-independent, tractable optimization problems.
When the distortion function is concave, we have demonstrated that this reformulation admits a conic representation, which we have explicitly derived for canonical distortion and divergence functions.
Whereas the multiplicity of constraints in the reformulation increases exponentially with the dimension of the underlying probability space, we have developed two types of algorithms, and combinations thereof, to circumvent this curse of dimensionality. 
We have formally established that the upper and lower bounds the algorithms generate converge to the optimal objective value.
Finally, we have illustrated the good performance of our methods in two examples involving concave as well as inverse $S$-shaped distortion functions, efficiently yielding nominal and robust solutions that generate very tight upper and lower bounds on the optimal objective values.
%

As a direction for future research, one can investigate 
whether the approach developed in this paper can be extended to encompass robust optimization problems with general law-invariant convex functionals (\cite{Follmer2011}, S.~4.5), which also admit a dual representation. 

\singlespacing
\bibliographystyle{apalike}
\bibliography{sample}

\begin{thebibliography}{}

\bibitem[Allais, 1953]{Allais1953}
Allais, M. (1953).
\newblock Le comportement de l'homme rationnel devant le risque: Critique des
  postulats et axiomes de l'\'ecole {A}m\'ericaine.
\newblock {\em Econometrica}, 21:503--546.

\bibitem[Beck and Ben-Tal, 2009]{PrimalWorstDualBest}
Beck, A. and Ben-Tal, A. (2009).
\newblock {Duality in robust optimization: {P}rimal worst equals dual best}.
\newblock {\em Operations Research Letters}, 37:1--6.

\bibitem[Ben-Tal et~al., 2013]{BenTal11}
Ben-Tal, A., den Hertog, D., de~Waegenaere, A., Melenberg, B., and Rennen, G.
  (2013).
\newblock {Robust solutions of optimization problems affected by uncertain
  probabilities}.
\newblock {\em Management Science}, 59(2):341--357.

\bibitem[Ben-Tal and Nemirovski, 2019]{ModernCVX}
Ben-Tal, A. and Nemirovski, A. (2019).
\newblock {\em Lectures on Modern Convex Optimization}.
\newblock SIAM.

\bibitem[Ben-Tal and Teboulle, 1986]{NCE}
Ben-Tal, A. and Teboulle, M. (1986).
\newblock {Expected utility, penalty functions, and duality in stochastic
  nonlinear programming}.
\newblock {\em Management Science}, 32(11):1445--1466.

\bibitem[Ben-Tal and Teboulle, 1987]{CE_phi_div}
Ben-Tal, A. and Teboulle, M. (1987).
\newblock {Penalty functions and duality in stochastic programming via
  $\phi$-divergence functionals}.
\newblock {\em Mathematics of Operations Research}, 12(2):224--240.

\bibitem[Berge, 1963]{Berge}
Berge, C. (1963).
\newblock {\em {Topological Spaces: Including a Treatment of Multi-Valued
  Functions, Vector Spaces, and Convexity}}.
\newblock Courier Corporation.

\bibitem[Bertsimas and Brown, 2009]{Bertsimas09}
Bertsimas, D. and Brown, D.~B. (2009).
\newblock {Constructing uncertainty sets for robust linear optimization}.
\newblock {\em Operations Research}, 57(6):1483--1495.

\bibitem[Bertsimas and den Hertog, 2022]{DickBoek}
Bertsimas, D. and den Hertog, D. (2022).
\newblock {\em Robust and Adaptive Optimization}.
\newblock Dynamic Ideas LLC Belmont, Massachusetts.

\bibitem[Bertsimas et~al., 2016]{CutvsRef}
Bertsimas, D., Dunning, I., and Lubin, M. (2016).
\newblock {Reformulation versus cutting-planes for robust optimization}.
\newblock {\em Computational Management Science}, 13(2):195--217.

\bibitem[Bélisle et~al., 1993]{HitRun}
Bélisle, C.~J., Romeijn, H.~E., and Smith, R.~L. (1993).
\newblock {Hit-and-Run algorithms for generating multivariate distributions}.
\newblock {\em Mathematics of Operations Research}, 18(2):255--266.

\bibitem[Cai et~al., 2025]{clmpaper}
Cai, J., Li, J. Y.-M., and Mao, T. (2025).
\newblock {Distributionally robust optimization under distorted expectations}.
\newblock {\em Operations Research}, 73(2):969--985.

\bibitem[Carlier and Dana, 2006]{carlier2006law}
Carlier, G. and Dana, R.-A. (2006).
\newblock Law invariant concave utility functions and optimization problems
  with monotonicity and comonotonicity constraints.
\newblock {\em Statistics \& Risk Modeling}, 24(1):127--152.

\bibitem[Chares, 2009]{Cha09}
Chares, P.~R. (2009).
\newblock {\em Cones and Interior-Point Algorithms for Structured Convex
  Optimization Involving Powers and Exponentials}.
\newblock Ecole polytechnique de Louvain, Université Catholique de Louvain.

\bibitem[Cherny and Madan, 2009]{Cherny09}
Cherny, A. and Madan, D. (2009).
\newblock {New measures for performance evaluation}.
\newblock {\em The Review of Financial Studies}, 22(7):2571--2606.

\bibitem[Chew et~al., 1987]{HKS1987}
Chew, S., Karni, E., and Safra, Z. (1987).
\newblock {Risk aversion in the theory of expected utility with rank dependent
  probabilities}.
\newblock {\em Journal of Economic Theory}, 42:370--381.

\bibitem[Choquet, 1954]{Choquet1954}
Choquet, G. (1954).
\newblock Theory of capacities.
\newblock {\em Annales de l’Institut Fourier}, 5:131--295.

\bibitem[Cox, 1971]{PieceLin2}
Cox, M. (1971).
\newblock {An algorithm for approximating convex functions by means by first
  degree splines}.
\newblock {\em The Computer Journal}, 14(3):272--275.

\bibitem[Csisz\'{a}r, 1975]{Csiszar1975}
Csisz\'{a}r, I. (1975).
\newblock $i$-divergence geometry of probability distributions and minimization
  problems.
\newblock {\em Annals of Probability}, 3:146--158.

\bibitem[Delage et~al., 2022]{Delage2022}
Delage, E., Guo, S., and Xu, H. (2022).
\newblock {Shortfall risk models when information of loss function is
  incomplete}.
\newblock {\em Operations Research}, 70(6):3511--3518.

\bibitem[Denneberg, 1990a]{Denneberg1990}
Denneberg, D. (1990a).
\newblock Distorted probabilities and insurance premiums.
\newblock {\em Methods of Operations Research}, 63:3--5.

\bibitem[Denneberg, 1990b]{denneberg1990premium}
Denneberg, D. (1990b).
\newblock Premium calculation: {W}hy standard deviation should be replaced by
  absolute deviation.
\newblock {\em ASTIN Bulletin: The Journal of the IAA}, 20(2):181--190.

\bibitem[Denneberg, 1994]{Denneberg}
Denneberg, D. (1994).
\newblock {\em Non-Additive Measure and Integral}.
\newblock Springer.

\bibitem[Eeckhoudt et~al., 1995]{EeckhoudtMS}
Eeckhoudt, L., Gollier, C., and Schlesinger, H. (1995).
\newblock The risk-averse (and prudent) newsboy.
\newblock {\em Management Science}, 41(5):786--794.

\bibitem[Eeckhoudt and Laeven, 2021]{EL21}
Eeckhoudt, L.~R. and Laeven, R.~J. (2021).
\newblock {Dual moments and risk attitudes}.
\newblock {\em Operations Research}, 70(3):1330--1341.

\bibitem[Eeckhoudt et~al., 2020]{ELS20}
Eeckhoudt, L.~R., Laeven, R.~J., and Schlesinger, H. (2020).
\newblock {Risk apportionment: {T}he dual story}.
\newblock {\em Journal of Economic Theory}, 185:104971.

\bibitem[Esfahani and Kuhn, 2018]{EsfahaniKuhn}
Esfahani, P.~M. and Kuhn, D. (2018).
\newblock {Data-driven distributionally robust optimization using the
  Wasserstein metric: {P}erformance guarantees and tractable reformulations}.
\newblock {\em Mathematical Programming}, 171:115--166.

\bibitem[F\"{o}llmer and Schied, 2016]{Follmer2011}
F\"{o}llmer, H. and Schied, A. (2016).
\newblock {\em Stochastic Finance}.
\newblock Berlin: De Gruyter, 4th edition.

\bibitem[Gilboa and Schmeidler, 1989]{MaxminEU}
Gilboa, I. and Schmeidler, D. (1989).
\newblock {Maxmin expected utility with non-unique prior}.
\newblock {\em Journal of Mathematical Economics}, 18(2):141--153.

\bibitem[Goovaerts and Laeven, 2008]{Goovaerts2008}
Goovaerts, M.~J. and Laeven, R.~J. (2008).
\newblock Actuarial risk measures for financial derivative pricing.
\newblock {\em Insurance: Mathematics and Economics}, 42(2):540--547.

\bibitem[Gorissen et~al., 2014]{Gorissen2014}
Gorissen, B.~L., Blanc, H., Ben-Tal, A., and den Hertog, D. (2014).
\newblock {Technical Note--Deriving robust and globalized robust solutions of
  uncertain linear programs with general convex uncertainty sets}.
\newblock {\em Operations Research}, 62(3):672--679.

\bibitem[{Gurobi Optimization, LLC}, 2023]{gurobi}
{Gurobi Optimization, LLC} (2023).
\newblock {Gurobi Optimizer Reference Manual}.

\bibitem[Hansen and Sargent, 2001]{HansenSargent01}
Hansen, L.~P. and Sargent, T.~J. (2001).
\newblock Robust control and model uncertainty.
\newblock {\em American Economic Review}, 91:60--66.

\bibitem[Hansen and Sargent, 2007]{HansenSargent07}
Hansen, L.~P. and Sargent, T.~J. (2007).
\newblock {\em Robustness}.
\newblock Princeton University Press.

\bibitem[He and Zhou, 2011]{Quantile_Method}
He, X.~D. and Zhou, X.~Y. (2011).
\newblock {Portfolio choice via quantiles}.
\newblock {\em Mathematical Finance}, 21(2):203--231.

\bibitem[Huber, 1981]{Huber1981}
Huber, P.~J. (1981).
\newblock {\em Robust Statistics}.
\newblock Wiley, New York.

\bibitem[Imamoto and Tang, 2008]{PieceLin1}
Imamoto, A. and Tang, B. (2008).
\newblock {Optimal piecewise linear approximation of convex functions}.
\newblock {\em Proceedings of the World Congress on Engineering and Computer
  Science}.

\bibitem[Laeven and Stadje, 2023]{dual_theory}
Laeven, R.~J. and Stadje, M. (2023).
\newblock {A rank-dependent theory for decision under risk and ambiguity}.
\newblock {\em Technical Report}.

\bibitem[Maccheroni et~al., 2006]{Maccheroni}
Maccheroni, F., Marinacci, M., and Rustichini, A. (2006).
\newblock {Ambiguity aversion, robustness, and the variational representation
  of preferences}.
\newblock {\em Econometrica}, 74(6):1447--1498.

\bibitem[Muliere and Scarsini, 1989]{MuliereScarsini1989}
Muliere, P. and Scarsini, M. (1989).
\newblock A note on stochastic dominance and inequality measures.
\newblock {\em Journal of Economic Theory}, 49(2):314--323.

\bibitem[Mutapcic and Boyd, 2009]{Boyd09}
Mutapcic, A. and Boyd, S. (2009).
\newblock {Cutting-set methods for robust convex optimization with pessimizing
  oracles}.
\newblock {\em Optimization Methods \& Software}, 24(3):381--406.

\bibitem[Pardo, 2006]{Pardo06}
Pardo, L. (2006).
\newblock {\em {Statistical Inference Based on Divergence Measures}}.
\newblock Chapman \& Hall/ CRC Boca Raton.

\bibitem[Pesenti et~al., 2020]{ConcentrationPaper}
Pesenti, S., Wang, Q., and Wang, R. (2020).
\newblock {Optimizing distortion riskmetrics with distributional uncertainty}.
\newblock {\em Paper}.
\newblock Available on ArXiv.

\bibitem[Postek et~al., 2016]{SiamPaper}
Postek, K., den Hertog, D., and Melenberg, B. (2016).
\newblock {Computationally tractable counterparts of distributionally robust
  constraints on risk measures}.
\newblock {\em SIAM Review}, 58(4):603--650.

\bibitem[Prelec, 1998]{Prelec}
Prelec, D. (1998).
\newblock {The probability weighting function}.
\newblock {\em Econometrica}, 66(3):497--527.

\bibitem[Quiggin, 1982]{Quiggin1982}
Quiggin, J. (1982).
\newblock {A theory of anticipated utility}.
\newblock {\em Journal of Economic Behavior and Organization}, 3(4):323--343.

\bibitem[Rockafellar, 1970]{Rockafellar}
Rockafellar, R.~T. (1970).
\newblock {\em Convex Analysis}.
\newblock Princeton University Press, Princeton, NJ.

\bibitem[Ryoo and Sahinidis, 1996]{ryoo1996branch}
Ryoo, H.~S. and Sahinidis, N.~V. (1996).
\newblock A branch-and-reduce approach to global optimization.
\newblock {\em Journal of Global Optimization}, 8(2):107--138.

\bibitem[Schmeidler, 1986]{Schmeidler1986}
Schmeidler, D. (1986).
\newblock {Integral representation without additivity}.
\newblock {\em Proceedings of the American Mathematical Society}, 97:255--261.

\bibitem[Schmeidler, 1989]{Schmeidler1989}
Schmeidler, D. (1989).
\newblock {Subjective probability and expected utility without additivity}.
\newblock {\em Econometrica}, 57(3):571--587.

\bibitem[Tversky and Kahneman, 1992]{TverskyKahneman1992}
Tversky, A. and Kahneman, D. (1992).
\newblock Advances in prospect theory: {C}umulative representation of
  uncertainty.
\newblock {\em Journal of Risk and Uncertainty}, 5:297--323.

\bibitem[von Neumann and Morgenstern, 1944]{NeunmannMorgen}
von Neumann, J. and Morgenstern, O. (1944).
\newblock {\em {Theory of Games and Economic Behavior}}.
\newblock Princeton University Press, Princeton, 3rd edition.

\bibitem[Wakker, 2010]{Wakker}
Wakker, P.~P. (2010).
\newblock {\em Prospect Theory for Risk and Rationality}.
\newblock Cambridge University Press.

\bibitem[Wald, 1950]{Wald1950}
Wald, A. (1950).
\newblock {\em Statistical Decision Functions}.
\newblock Wiley, New York.

\bibitem[Wang, 1995]{Wang1995}
Wang, S. (1995).
\newblock Insurance pricing and increased limits ratemaking by proportional
  hazards transforms.
\newblock {\em Insurance: Mathematics and Economics}, 17(1):43--54.

\bibitem[Wang, 2000]{Wang2000}
Wang, S.~S. (2000).
\newblock A class of distortion operators for pricing financial and insurance
  risks.
\newblock {\em The Journal of Risk and Insurance}, 67(1):15--36.

\bibitem[Wang and Xu, 2023]{WangXu2021}
Wang, W. and Xu, H. (2023).
\newblock {Preference robust distortion risk measure and its application}.
\newblock {\em Mathematical Finance}, 33(2).

\bibitem[Wiesemann et~al., 2014]{KuhnWieseman}
Wiesemann, W., Kuhn, D., and Sim, M. (2014).
\newblock {Distributionally robust convex optimization}.
\newblock {\em Operations Research}, 62(6):1358--1376.

\bibitem[Yaari, 1987]{Yaari1987}
Yaari, M.~E. (1987).
\newblock The dual theory of choice under risk.
\newblock {\em Econometrica}, 55(1):95--115.

\end{thebibliography}


\setcounter{equation}{0}
\appendix
\renewcommand{\thesection}{EC.\arabic{section}}
\counterwithout{equation}{section}
\renewcommand{\theequation}{EC.\arabic{equation}}
\theoremstyle{plain}
\newtheorem{mylemma}{Lemma}[section]
\renewcommand{\themylemma}{EC.\arabic{section}.\arabic{mylemma}}
\theoremstyle{plain}
\newtheorem{mytheorem}{Theorem}[section]
\renewcommand{\themytheorem}{EC.\arabic{section}.\arabic{mytheorem}}
\renewcommand{\theproposition}{EC.\arabic{section}}
\theoremstyle{plain}
\newtheorem{myproposition}{Proposition}[section]
\renewcommand{\themyproposition}{EC.\arabic{section}.\arabic{myproposition}}

\renewcommand{\theassumption}{EC.\arabic{section}}
\theoremstyle{plain}
\newtheorem{myassumption}{Assumption}[section]
\renewcommand{\themyassumption}{EC.\arabic{section}.\arabic{myassumption}}

\renewcommand{\thedefinition}{EC.\arabic{section}}
\theoremstyle{plain}
\newtheorem{mydefinition}{Definition}[section]
\renewcommand{\themydefinition}{EC.\arabic{section}.\arabic{mydefinition}}

\begin{landscape}
\section*{Online Appendix}
\section{Tables~\ref{tab:h_conjugates} and~\ref{tab:phi_conjugate}}\label{sec:tab}
\begin{table}[H]	
\centering
{\footnotesize\begin{tabular}{l|l|l|l|l|}
\textit{Distortion family} & \textit{Function}
& \textit{Conjugate} & \textit{Epigraph of perspective} & \textit{Conic}\\
&$h(p)$, $p\in [0,1]$ & $(-h)^*(y)$, $y\leq 0$& $\lambda(-h)^*\left(\frac{-\nu}{\lambda}\right)\leq z,~\lambda>0$& \textit{representation}\\ \hline

Expectation&$p$&$0$, $y\leq -1$&$0\leq z$, $\nu\geq \lambda$&CQ\\ \hline

$\mathrm{CVaR}_{1-\alpha}$&$
    \min\{\frac{p}{1-\alpha},1\}$, $\alpha\in[0,1)$ &  $\max\{(1-\alpha)y+1,0\}$ & $\max\{-(1-\alpha)\nu +\lambda,0\}\leq z$,~$\nu \geq 0$&CQ      \\ \hline
     Proportional Hazard&$\begin{cases}
     p^r\\
     r\in (0,1)
     \end{cases}$& $r^{\frac{1}{1-r}}(1+\frac{1}{r})|y|^\frac{r}{r-1}$&$\begin{cases}
          \lambda(r^{\frac{r}{1-r}}-r^\frac{1}{1-r})^{1-r}\leq z^{1-r}\nu^r\\
          \nu \geq 0
      \end{cases}$ &PC\\\hline
      Absolute Deviation&$\begin{cases}
    (1+r)p&p<1/2\\
    (1-r)p+r&p\geq 1/2\\
    r\in (0,1)
\end{cases}$&$\max\{(y+1+r)/2,0\},~ y\leq -(1-r)$&$\begin{cases}\max\{(-\nu+\lambda+\lambda r)/2,0\}\leq z\\ -\nu\leq -(1-r)\lambda \end{cases}$&CQ\\ \hline
     Gini Principles& $\begin{cases}
        (1+r)p-rp^2\\
        r\in (0,1)
     \end{cases}$ &$\begin{cases}
        \frac{1}{4r}\max\{y+1+r,0\}^2 & y+1\leq r\\
        y+1 & y+1>r
     \end{cases}$&$\begin{cases}
        z = z_1+z_2\\
        -\nu+\lambda(1-r)= \xi_1+\xi_2\\
        (z_1+\lambda)\geq \sqrt{\frac{w^2}{r}+(z_1-\lambda)^2},~z_2\geq \xi_2\\
        \xi_1+2r\leq w\\
        \xi_1\leq 0, \xi_2, w, \nu \geq 0
    \end{cases}$& CQ\\ \hline
      Dual Moments&$\begin{cases}
          1-(1-p)^n\\
          n>1
      \end{cases}$  &$\begin{cases}
       y+ c(n)\cdot \min\{|y|,n\}^{\frac{n}{n-1}}+1\\
       c(n)= \left(n^{-\frac{1}{n-1}}-n^{-\frac{n}{n-1}}\right)
    
      \end{cases}$& $\begin{cases}
        \lambda-\xi_3+(n^{-\frac{1}{n-1}}-n^{-\frac{n}{n-1}})\xi_4\leq z\\
        \xi_2\leq \xi_4^{\frac{n-1}{n}}\cdot \lambda^{1-\frac{n-1}{n}}\\
        -\nu+\xi_3\leq 0,~ \xi_2\geq \xi_3,\\ \xi_2,\xi_3,\xi_4\geq 0.
    \end{cases}$& PC  \\ \hline
      $\mathrm{MAXMINVAR}$&$\begin{cases}
          (1-(1-p)^n)^{1/n}\\
          n>1
      \end{cases}$&$\begin{cases}
      y\left(1-s^{\frac{1}{n}}\right)+\left(1-s\right)^{\frac{1}{n}}\\
      s= \frac{|y|^{\frac{n}{n-1}}}{1+|y|^{\frac{n}{n-1}}}
      \end{cases}$& $\begin{cases}
          -\nu+\xi_3\leq 0,\xi_2\geq \lambda\\ -\xi_3+\|(\xi_2,\xi_3)\|_{\frac{n}{n-1}}\leq z\\
          \xi_2,\xi_3 \geq 0
      \end{cases}$&PC\\\hline
      Lookback Transform&$\begin{cases}
          p^r(1-\log(p^r))\\
          r\in (0,1)
      \end{cases}$&Closed form unknown &$\begin{cases}
          \xi_3+\lambda e^{-(\xi_2+\xi_3)/\lambda}+(r^{\frac{r}{1-r}}-r^{\frac{1}{1-r}})\xi_4\leq z\\
          \xi_2\leq \nu^r\xi_4^{1-r}\\
          \xi_2,\xi_3,\xi_4,\nu \geq 0
      \end{cases}$& $\mathrm{EXP}\times \mathrm{PC}$ \\\hline
\end{tabular}}
\caption{Canonical distortion functions with explicit expressions of their conjugates $(-h)^*$ and conic representations of the epigraphs of their perspectives. 
The cones are abbreviated as: CQ: quadratic cone, PC: power cone, EXP: exponential cone. 
Note that $(-h)^*(y)=+\infty$ for $y>0$; see Remark~\ref{rem:extendedh}.
For Expectation and $\mathrm{CVaR}_{1-\alpha}$, see \cite{Follmer2011}, Section~4.6; 
for Proportional Hazard, see \cite{Wang1995}; 
for Absolute Deviation and Gini Principles, see \cite{denneberg1990premium};
for Dual Moments, see \cite{MuliereScarsini1989} and \cite{ELS20}; 
for MAXMINVAR, see \cite{Cherny09}; 
for Lookback Transform, see \cite{Denneberg1990}. }\label{tab:h_conjugates}
\end{table}
\end{landscape}
\begin{landscape}
\begin{table}[htp]
\centering
{\footnotesize\begin{tabular}{l|l|l|l|l|}
\textit{Divergence family}& \textit{Function} &\textit{Conjugate} & \textit{Epigraph of perspective} & \textit{Conic}\\
& $\phi(x)$, $x\geq 0$ & $\phi^*(y)$, $y\in \mathbb{R}$&$\gamma\phi^*\left(\frac{s}{\gamma}\right)\leq t,~\gamma>0$&\textit{representation} \\ \hline
Kullback-Leibler&$x\log x-x+1$&$e^y-1$&$\begin{cases}
    \gamma\log(\frac{\gamma}{w})+s\leq 0\\
    w-\gamma \leq t.
\end{cases}$&EXP\\\hline
Burg entropy&$-\log x+x-1$&$-\log(1-y),~ y<1$&$\begin{cases}
    \gamma\log(\frac{\gamma}{v})\leq t\\
    v=\gamma - s,~v>0\\
\end{cases}$&EXP\\\hline
$\chi^2$-distance& $\frac{1}{x}(x-1)^2$&$2-2\sqrt{1-y},~ y<1$&$\begin{cases}
    2\gamma -2w\leq t\\
    \sqrt{w^2+\frac{1}{4}(\gamma -v)^2}\leq \frac{1}{2}(\gamma+v)\\
    w\geq 0,~v= \gamma -s,~ v>0\\
\end{cases}$&CQ\\\hline
Variation distance &$|x-1|$& $\max\{s,-1\},~ s\leq 1$&$\max\{s,-\gamma\}\leq t,~ s\leq \gamma$ &CQ\\\hline
Modified $\chi^2$-distance& $(x-1)^2$&$\max\{0,y/2+1\}^2-1$&$\begin{cases}
 \sqrt{w^2+\frac{t^2}{4}}\leq \frac{t+2\gamma}{2}\\
    0\leq w, s/2+\gamma\leq w
\end{cases}$&CQ\\ \hline
Hellinger distance& $(\sqrt{x}-1)^2$& $\frac{y}{1-y},~ y<1$&$\begin{cases}
     -\gamma+v\leq t\\ \sqrt{\gamma^2+\frac{1}{4}(v-w)^2}\leq \frac{1}{2}(v+w)\\
     w=\gamma-s,~ w>0.
\end{cases}$&CQ\\ \hline
$\chi$-divergence of order $\theta>1$& $|x-1|^\theta$& $y+(\theta -1)(\frac{|y|}{\theta})^{\frac{\theta}{\theta-1}}$& $\begin{cases}
s+(\theta-1)\theta^{\frac{\theta}{1-\theta}} w\leq t, 0\leq w\\
~|s|\leq w^{\frac{\theta}{\theta-1}}\cdot \gamma^{1-\frac{\theta}{\theta-1}}    
\end{cases}$&PC\\ \hline
Cressie and Read& $\frac{1-\theta+\theta x-x^\theta}{\theta(1-\theta)},~ 0<\theta<1$&$\begin{cases}
    \frac{1}{\theta}(1-y(1-\theta))^{\frac{\theta}{\theta -1}}-\frac{1}{\theta}\\
    y>\frac{1}{1-\theta}
\end{cases}$& $\begin{cases}
    |w|\leq (t\theta +\gamma)^{\frac{\theta-1}{\theta}}\cdot \gamma^{1-\frac{\theta-1}{\theta}}\\ w = \gamma-s(1-\theta), ~s<\frac{\gamma}{1-\theta}
\end{cases}$&PC\\ \hline
\end{tabular}}
\caption{Canonical $\phi$-divergence functions taken from Table~2 of \citet{BenTal11}, but with the explicit epigraphs of their perspective functions (provided in the fourth column) and corresponding conic representations (provided in the fifth column). 
The cones are abbreviated as: CQ: quadratic cone, PC: power cone, EXP: exponential cone.}\label{tab:phi_conjugate}
\end{table}
\end{landscape}

\section{Proofs} \label{app:proofs}
\begin{proof}[\textbf{Proof of Theorem~\ref{thm:ReformulationWholeProb}}]
Let $\mathbf{q}$ be an arbitrary probability vector and denote by $\mathbb{Q}$ the corresponding probability measure on the sigma-algebra $2^{|\Omega|}$, i.e., the set of all subsets of $\Omega$.
Recall Eqn.~\eqref{def:distortion_non_additive}.
By \citet{Denneberg}, pp.~16--17, the set function induced by the composition $h\circ \mathbb{Q}$ is monotone and submodular.
It then follows from Proposition~10.3 of \citet{Denneberg} that, for any $X:\Omega\rightarrow\mathbb{R}$, 
\begin{align}
\nonumber
 \rho_{u,h,\mathbf{q}}(X)=\sup_{\mathbf{\bar{q}}\in M_h(\mathbf{q})} \mathbb{E}_{\mathbf{\bar{q}}}[-u(X)],
\end{align}
with $M_h(\mathbf{q})$ defined in \eqref{defMhq}.

Hence, we have that
\begin{align*}
    \sup_{\mathbf{q}\in \mathcal{D}_\phi(\mathbf{p},r)}\rho_{u,h,\mathbf{q}}(f(\mathbf{a},\mathbf{X}))
    &=\sup_{\mathbf{q}\in \mathcal{D}_\phi(\mathbf{p},r)}\sup_{\mathbf{\bar{q}}\in M_h(\mathbf{q})} \mathbb{E}_{\mathbf{\bar{q}}}[-u(f(\mathbf{a},\mathbf{X}))]\\
    &=\sup_{(\mathbf{q},\bar{\mathbf{q}})\in \mathcal{U}_{\phi,h}(\mathbf{p})}-\sum^m_{i=1}\bar{q}_iu(f(\mathbf{a},\mathbf{x}_i)),
\end{align*}
where $\mathcal{U}_{\phi,h}(\mathbf{p})$ is as defined in \eqref{uncertaintyset}.
This proves the stated result.
\end{proof}
\begin{proof}[\textbf{Proof of Theorem~\ref{thm:RC}}] Suppose that a pair $(\mathbf{a},c)\in \mathbb{R}^{n_a+1}$ satisfies the inequality
\begin{align}\label{primal_prob}
\sup_{(\mathbf{q},\mathbf{\bar{q}})\in \mathcal{U}_{\phi,h}(\mathbf{p})}-\sum^m_{i=1} \bar{q}_iu(f(\mathbf{a},\mathbf{x}_i))\leq c.
\end{align}
Then, the left-hand side of \eqref{primal_prob} constitutes a maximization problem upper bounded by $c<\infty$. 
Moreover, the set $\mathcal{U}_{\phi,h}(\mathbf{p})$ contains $(\mathbf{p},\mathbf{p})$ as a Slater point, since $\phi(1)=0<r$ and $h(\sum_{k\in I_j}p_k)>\sum_{k\in I_j}p_k$ for all subsets $I_j\subset [m]$ that are not $\emptyset$ and $[m]$.\footnote{This is because $h(x)>x$ for all $x\in (0,1)$ due to concavity (excluding the trivial case where $h(x)=x$ for all $x\in[0,1]$).} 
Therefore, strong duality holds, and we consider the Lagrangian function
\begin{align*}
&L(\mathbf{a},\mathbf{q},\bar{\mathbf{q}},\alpha,\beta,(\lambda_j)_j,\gamma)\\
&=-\sum^m_{i=1}\bar{q}_iu(f(\mathbf{a},\mathbf{x}_i))-\alpha \left(\sum^m_{i=1}q_i-1\right)-\beta\left(\sum^m_{i=1}\bar{q}_i-1\right)-\sum^{2^m-2}_{j=1} \lambda_j\left(\sum_{k\in I_j}\bar{q}_k-h\left(\sum_{k\in I_j}q_k\right)\right)\\
&~~~~~~-\gamma \left(\sum^m_{i=1}p_i\phi\left(\frac{q_i}{p_i}\right)-r\right)\\
&=\alpha+\beta+\gamma r +\sum^m_{i=1}-(u(f(\mathbf{a},\mathbf{x}_i))+\beta)\bar{q}_i-\sum^{2^m-2}_{j=1} \lambda_j\sum_{k\in I_j}\bar{q}_k+\sum^m_{i=1}-\alpha q_i-\gamma \left(\sum^m_{i=1}p_i\phi\left(\frac{q_i}{p_i}\right)\right)\\
&~~~~~~ -\sum^{2^m-2}_{j=1}-\lambda_jh\left(\sum_{k\in I_j}q_k\right),
\end{align*}
for $\alpha,\beta\in \mathbb{R}$ and $\lambda_j,\gamma\geq 0$.  
We analyze $\sup_{\mathbf{q},\mathbf{\bar{q}}\geq 0}L(\mathbf{a},\mathbf{q},\mathbf{\bar{q}},\alpha,\beta,(\lambda_j)_j,\gamma)$, which excluding the constant $\alpha+\beta+\gamma r$ is equal to:
\begin{align*}
&\sup_{\mathbf{q},\bar{\mathbf{q}}\geq 0}\sum^m_{i=1}-(u(f(\mathbf{a},\mathbf{x}_i))+\beta)\bar{q}_i-\sum^{2^m-2}_{j=1} \lambda_j\sum_{k\in I_j}\bar{q}_k+\sum^m_{i=1}-\alpha q_i-\gamma p_i\phi\left(\frac{q_i}{p_i}\right)\\
&\qquad-\sum^{2^m-2}_{j=1} \lambda_j(-h)\left(\sum_{k\in I_j}q_k\right)\\
    &=\sup_{\bar{\mathbf{q}}\geq 0}\sum^m_{i=1}-(u(f(\mathbf{a},\mathbf{x}_i))+\beta)\bar{q}_i-\sum^{2^m-2}_{j=1} \lambda_j\sum_{k\in I_j}\bar{q}_k+\sup_{\mathbf{q}\geq 0}\sum^m_{i=1}-\alpha q_i-\gamma p_i\phi\left(\frac{q_i}{p_i}\right)\\
    &\qquad-\sum^{2^m-2}_{j=1} \lambda_j(-h)\left(\sum_{k\in I_j}q_k\right).
\end{align*}
We examine both supremum terms separately. 
The first supremum gives:
\begin{align*}
     \sup_{\bar{\mathbf{q}}\geq 0}\sum^m_{i=1}-(u(f(\mathbf{a},\mathbf{x}_i))+\beta)\bar{q}_i-\sum^{2^m-2}_{j=1} \lambda_j\sum_{k\in I_j}\bar{q}_k
     &=\sup_{\bar{\mathbf{q}}\geq 0}\sum^m_{i=1}-\left(u(f(\mathbf{a},\mathbf{x}_i))+\beta+\sum_{j:i\in I_j}\lambda_j\right)\bar{q}_i\\
     &=\begin{cases} 
     0&\text{if}~u(f(\mathbf{a},\mathbf{x}_i))+\beta+\sum_{j:i\in I_j}\lambda_j\geq 0,~\forall i\\
     \infty&\text{else}.
     \end{cases}
\end{align*}
The second supremum gives:
\begin{align*}
     &\sup_{\mathbf{q}\geq 0}\sum^m_{i=1}-\alpha q_i-\gamma p_i\phi\left(\frac{q_i}{p_i}\right)-\sum^{2^m-2}_{j=1} \lambda_j(-h)\left(\sum_{k\in I_j}q_k\right)\\
     &=\sup_{\substack{\mathbf{q} \geq \mathbf{0}\\w_1,\ldots, w_{2^m-2}\geq 0}}\left\{\sum^m_{i=1}-\alpha q_i-\gamma p_i\phi\left(\frac{q_i}{p_i}\right)-\sum^{2^m-2}_{j=1} \lambda_j(-h)(w_j)~\middle|~w_j=\sum_{k\in I_j}q_k\right\}\\
     &=\inf_{\nu_1,\ldots, \nu_{2^m-2}}\sup_{\substack{\mathbf{q}\geq \mathbf{0}\\w_1,\ldots, w_{2^m-2}\geq 0}}\sum^m_{i=1}-\alpha q_i-\gamma p_i\phi\left(\frac{q_i}{p_i}\right)-\sum^{2^m-2}_{j=1} \lambda_j(-h)(w_j)-\sum^{2^m-2}_{j=1}\nu_j\left(w_j-\sum_{k\in I_j}q_k\right)\\
    &=\inf_{\nu_1,\ldots,\nu_{2^m-2}}\sup_{\mathbf{q}\geq 0}\sum^m_{i=1}\left(-\alpha+\sum_{j: i\in I_j}\nu_j\right) q_i-\gamma p_i\phi\left(\frac{q_i}{p_i}\right)+\sup_{w_1,\ldots ,w_{2^m-2}\geq 0}\sum^{2^m-2}_{j=1}-\nu_jw_j-\lambda_j(-h)(w_j)\\
    &=\inf_{\nu_1,\ldots, \nu_{2^m-2}}\sum^m_{i=1}p_i\left(\sup_{t\geq 0}\left(-\alpha+\sum_{j: i\in I_j}\nu_j\right)t-\gamma \phi(t)\right)+\sum^{2^m-2}_{j=1}\lambda_j(-h)^*\left(\frac{-\nu_j}{\lambda_j}\right)\\
    &=\inf_{\nu_1,\ldots,\nu_{2^m-2}}\sum^m_{i=1}p_i\gamma \phi^*\left(\frac{-\alpha+\sum_{j:i\in I_j}\nu_j}{\gamma}\right)+\sum^{2^m-2}_{j=1}\lambda_j(-h)^*\left(\frac{-\nu_j}{\lambda_j}\right),
\end{align*}
with $\phi^*$ and $(-h)^*$ the conjugates of $\phi$ and $-h$. 

Therefore, strong duality implies that a pair $(\mathbf{a},c)$ satisfies \eqref{primal_prob}
if and only if
\begin{align*}
\inf_{\substack{\alpha,\beta,\nu_j\in\mathbb{R}\\\lambda_j,\gamma\geq 0}}&\alpha+\beta+\gamma r+\sum^m_{i=1}p_i\gamma \phi^*\left(\frac{-\alpha+\sum_{j:i\in I_j}\nu_j}{\gamma}\right)+\sum^{2^m-2}_{j=1}\lambda_j(-h)^*\left(\frac{-\nu_j}{\lambda_j}\right)\leq c\\
    \text{subject to}&~-u(f(\mathbf{a},\mathbf{x}_i))-\beta-\sum_{j:i\in I_j}\lambda_j\leq 0,~\forall i\in [m].    
\end{align*}
Since the infimum is attained due to the boundedness of the primal problem~\eqref{primal_prob} and the strong duality theorem, we may remove the infimum sign and obtain that the above holds if and only if there exist $\lambda_j,\gamma\geq 0,\alpha,\beta,\nu_j\in\mathbb{R}$ such that
\begin{align*}
    \begin{cases}
    \alpha+\beta+\gamma r+\sum^m_{i=1}p_i\gamma \phi^*\left(\frac{-\alpha+\sum_{j:i\in I_j}\nu_j}{\gamma}\right)+\sum^{2^m-2}_{j=1}\lambda_j(-h)^*\left(\frac{-\nu_j}{\lambda_j}\right)\leq c\\
    -u(f(\mathbf{a},\mathbf{x}_i))-\beta-\sum_{j:i\in I_j}\lambda_j\leq 0,~\forall i\in[m]\\
    \lambda_j,\gamma \geq 0\\
    \alpha, \beta, \nu_{j}\in \mathbb{R},~\forall j\in [2^m-2].
    \end{cases}
\end{align*}

For the nominal problem, the Lagrangian function is given by
\begin{align*}
    L(\mathbf{a},\bar{\mathbf{q}},\beta,(\lambda_j)_j)&=-\sum^m_{i=1}\bar{q}_iu(f(\mathbf{a},\mathbf{x}_i))-\beta\left(\sum^m_{i=1}\bar{q}_i-1\right)-\sum^{2^m-2}_{j=1} \lambda_j\left(\sum_{k\in I_j}\bar{q}_k-h\left(\sum_{k\in I_j}p_k\right)\right)\\
    &=\beta+\sum^{2^m-2}_{j=1}\lambda_jh\left(\sum_{k\in I_j}p_k\right)+\sum^m_{j=1}\bar{q}_i\left(-u(f(\mathbf{a},\mathbf{x}_i))-\beta-\sum_{j: i\in I_j}\lambda_j\right).
\end{align*}
Hence, we have
\begin{align*}
    &\sup_{\bar{\mathbf{q}}\in M_h(\mathbf{p})}\sum^m_{i=1}-\bar{q}_iu(f(\mathbf{a},\mathbf{x}_i))&\\
    &=\inf_{\beta\in \mathbb{R},\lambda_j\geq 0}\sup_{\bar{\mathbf{q}}\geq \mathbf{0}}L(\mathbf{a},\bar{\mathbf{q}},\beta,(\lambda_j)_j)\\
    &=\inf_{\beta\in \mathbb{R},\lambda_j\geq 0} \left\{\beta+\sum^{2^m-2}_{j=1}\lambda_jh\left(\sum_{k\in I_j}p_k\right)~\middle |~-u(f(\mathbf{a},\mathbf{x}_i))-\beta-\sum_{j: i\in I_j}\lambda_j\leq 0,~\forall i\in [m]\right\},
\end{align*}
where the above strong duality holds since we are solving a linear programming problem.
Therefore, we have that $\sup_{\bar{\mathbf{q}}\in M_h(\mathbf{p})}\sum^m_{i=1}-\bar{q}_iu(f(\mathbf{a},\mathbf{x}_i))\leq c$ if and only if there exist $\beta\in \mathbb{R}, (\lambda_j)_j\geq 0$ such that
\begin{align*}
    \begin{cases}
    \beta+\sum^{2^m-2}_{j=1}\lambda_jh\left(\sum_{k\in I_j}p_k\right)\leq c\\
    -u(f(\mathbf{a},\mathbf{x}_i))-\beta-\sum_{j: i\in I_j}\lambda_j\leq 0,~\forall i\in [m]\\
    \beta\in \mathbb{R}, \lambda_j\geq 0,~ \forall j\in [2^m-2].
    \end{cases}
\end{align*}
 \end{proof}
\begin{proof}[\textbf{Proof of Lemma~\ref{lem:conic_conjugate}}]
The proof follows \citet{ModernCVX}, except that we replace the quadratic cone with a general cone. 
We have
\begin{align*}
    \mathrm{Epi}(f^*)=\{(\mathbf{y},s):\mathbf{y}^T\mathbf{x}-f(\mathbf{x})\leq s,\forall \mathbf{x}\}&=\{(\mathbf{y},s):\mathbf{y}^T\mathbf{x}-t\leq s,\forall (\mathbf{x},t)\in \mathrm{Epi}(f)\}.
\end{align*}
Therefore, $(\mathbf{y},s)\in \mathrm{Epi}(f^*)$ if and only if
\begin{align*}
    \min_{\mathbf{x},t,\mathbf{w}}\{-\mathbf{y}^T\mathbf{x}+t:\mathbf{A}\mathbf{x}+t\mathbf{v}+\mathbf{B}\mathbf{w}+\mathbf{b}\succeq_\mathbf{K}\mathbf{0}\}\geq -s.
\end{align*}
This minimization problem is strictly feasible and bounded from below. 
Hence, the conic duality theorem (see \citealp{ModernCVX}) implies that it is equal to 
\begin{align*}
    \max_{\boldsymbol{\xi}}\{-\mathbf{b}^T\boldsymbol{\xi}:\mathbf{A}^T\boldsymbol{\xi} = -\mathbf{y}, \mathbf{B}^T\boldsymbol{\xi}=\mathbf{0}, \mathbf{v}^T\boldsymbol{\xi} =1, \boldsymbol{\xi}\in\mathbf{K}^*\}.
\end{align*}
Therefore, 
\begin{align*}
    \mathrm{Epi}(f^*)=\{(\mathbf{y},s):\exists~ \boldsymbol{\xi}\in \mathbf{K}^*: \mathbf{A}^T\boldsymbol{\xi} = -\mathbf{y}, \mathbf{B}^T\boldsymbol{\xi}=\mathbf{0}, \mathbf{v}^T\boldsymbol{\xi} =1, s\geq \mathbf{b}^T\boldsymbol{\xi}\}.
\end{align*}
\end{proof}
\begin{proof}[\textbf{Proof of Lemma~\ref{lem:conic_pers}}] 
Following \citet{ModernCVX},
we have
\begin{align*}
    \mathrm{Epi}(\tilde{f})=\left\{(\mathbf{x},\lambda,t):\lambda f\left(\frac{\mathbf{x}}{\lambda}\right)\leq t\right\}&=\left\{(\mathbf{x},\lambda,t):\left(\frac{\mathbf{x}}{\lambda},\frac{t}{\lambda}\right)\in \mathrm{Epi}(f)\right\}\\
    &=\{(\mathbf{x},\lambda,t):\exists~ \mathbf{w}\in \mathbb{R}^k:\mathbf{A}(\mathbf{x}/\lambda, \mathbf{w}, t/\lambda)^T-b\succeq_{\mathbf{K}}\mathbf{0}\}\\
    &=\{(\mathbf{x},\lambda,t):\exists~ \tilde{\mathbf{w}}\in \mathbb{R}^k:\mathbf{A}(\mathbf{x}/\lambda, \Tilde{\mathbf{w}}/\lambda, t/\lambda)^T-b\succeq_{\mathbf{K}}\mathbf{0}\}\\
    &=\{(\mathbf{x},\lambda,t):\exists~ \tilde{\mathbf{w}}\in \mathbb{R}^k:\mathbf{A}(\mathbf{x},\tilde{\mathbf{w}},t)^T-\lambda b\succeq_{\mathbf{K}}\mathbf{0}\}\\
    &=\{(\mathbf{x},\lambda,t):\exists~ \tilde{\mathbf{w}}\in \mathbb{R}^k:[\mathbf{A},-b](\mathbf{x},\tilde{\mathbf{w}},t,\lambda)^T\succeq_{\mathbf{K}}\mathbf{0}\}.
\end{align*}
 \end{proof}

\begin{proof}[\textbf{Proof of Lemma~\ref{lem: qq_in_U}}]
By definition, $\mathbf{q}^*\in \mathcal{D}_{\phi}(\mathbf{p},r)$. Hence, we only need to show that $\bar{\mathbf{q}}^*\in M_h(\mathbf{q}^*)$.
    Using Lemma~4.98 of \citet{Follmer2011}, we have that $\rho_{u,h,\mathbf{q}^*}(X)\geq \mathbb{E}_{\bar{\mathbf{q}}^*}[-X]$ for all random variables $X$ (where $\rho_{u,h,\mathbf{q}^*}(X)$ is as defined in \eqref{def:distortionriskmeasure} with measure $\mathbb{Q}=\mathbf{q}^*$). 
    In particular, this holds for all $X=-\mathbbm{1}_{A}$, for any measurable set $A\subset \Omega$. 
    Hence, $(\mathbf{q}^*,\bar{\mathbf{q}}^*)\in \mathcal{U}_{\phi,h}(\mathbf{p})$.
\end{proof}

\begin{proof}[\textbf{Proof of Theorem~\ref{thm:cuttingplane_conv}}]
First, we have, for any $\mathbf{a}^{1},\mathbf{a}^{2}\in \mathcal{A}$, 
\begin{align}
    \begin{split}
\left|\sum_{i=1}^m\bar{q}_iu(f(\mathbf{a}^{2},\mathbf{x}_i))-\sum_{i=1}^m\bar{q}_iu(f(\mathbf{a}^{1},\mathbf{x}_i))\right|&\leq \|\mathbf{\bar{q}}\|_2\|(u(f(\mathbf{a}^{2},\mathbf{x}_i))-u(f(\mathbf{a}^{1},\mathbf{x}_i)))^m_{i=1}\|_2\\
    &\leq \|\mathbf{\bar{q}}\|_1\|(u(f(\mathbf{a}^{2},\mathbf{x}_i))-u(f(\mathbf{a}^{1},\mathbf{x}_i)))^m_{i=1}\|_2\\
    &=\|(u(f(\mathbf{a}^{2},\mathbf{x}_i))-u(f(\mathbf{a}^{1},\mathbf{x}_i)))^m_{i=1}\|_2,
    \end{split}
\end{align}
where the first inequality follows from Cauchy-Schwarz; 
and the second inequality follows from $\|\bar{\mathbf{q}}\|_2\leq \|\bar{\mathbf{q}}\|_1$ for a probability vector $\bar{\mathbf{q}}$, since $\bar{q}_k^2\leq \bar{q}_k$ for $|\bar{q}_k|\leq 1$.

Suppose now that the cutting-plane method has not terminated at the $t$-th iteration, i.e., the optimal solution and objective value $(\mathbf{a}_t,c_t)$ violate the $\epsilon_{\mathrm{tol}}$-feasibility condition at step \ref{step4} of Algorithm~\ref{cutting_plane_algo}. 
Let $(\mathbf{q}^*_t, \mathbf{\bar{q}}^*_{t})$ be the new worst-case scenarios that are added to $\mathcal{U}_t$ at step \ref{step5} of Algorithm \ref{cutting_plane_algo}. 
Then, by definition, we have
\begin{align}\label{t_iteration}
    -\sum_{i=1}^m\bar{q}^*_{t,i}u(f(\mathbf{a}_t,\mathbf{x}_i))-c_t>\epsilon_{\mathrm{tol}}.
\end{align}
For any $s>t$, we also have that at the $s$-th iteration:
\begin{align}\label{s_iteration}
    -\sum_{i=1}^m\bar{q}^*_{t,i}u(f(\mathbf{a}_s,\mathbf{x}_i))-c_s\leq 0,
\end{align}
since this is part of the constraint at the $s$-th iteration. Let $\tilde{c}$ be the optimal objective value of \eqref{P-ref}. 
By Assumption~\ref{assumption:well-defined}, we have that $-\infty<\tilde{c}<\infty$. 
Furthermore, we have $\tilde{c}\geq c_i$ for any iteration $i$ since the cutting-plane algorithm always yields a lower bound on \eqref{P-ref}. 
Hence, we may assume that for $t$ and $s$ sufficiently large, the lower bound improvement is upper bounded, i.e., $c_s-c_t\leq \frac{1}{2}\epsilon_{\mathrm{tol}}$, since otherwise the cutting-plane will yield a lower bound that exceeds $\tilde{c}$, after finitely many iterations.  
Therefore, it follows from \eqref{t_iteration}--\eqref{s_iteration} that
\begin{align*}
    \left|\sum_{i=1}^m\bar{q}^*_{t,i}u(f(\mathbf{a}_t,\mathbf{x}_i))-\sum_{i=1}^m\bar{q}^*_{t,i}u(f(\mathbf{a}_s,\mathbf{x}_i))\right|>\epsilon_{\mathrm{tol}}-(c_s-c_t)\geq \frac{1}{2}\epsilon_{\mathrm{tol}},
\end{align*}
which implies that
\begin{align}\label{no_intersection}
   \|(u(f(\mathbf{a}_s,\mathbf{x}_i))-u(f(\mathbf{a}_t,\mathbf{x}_i)))^m_{i=1}\|_2>\frac{1}{2}\epsilon_{\mathrm{tol}}.
\end{align}
This shows that the minimum distance between any two outcomes of the cutting-plane method, when evaluated in utilities, is at least $\frac{1}{2}\epsilon_{\mathrm{tol}}$. 

The idea is now to show that if the cutting-plane method does not terminate, then there is a sequence of infinitely many cutting-plane solutions $\{\mathbf{a}_j\}^\infty_{j=1}$ for which the corresponding vector $(u(f(\mathbf{a}_j,\mathbf{x}_i)))^m_{i=1}$ remains in a bounded set. 
Since we know from above that for each of these solutions their utility values are a distance $\frac{1}{2}\epsilon_{\mathrm{tol}}$ away from each other, we conclude that the cutting-plane method must terminate since a bounded set can not contain infinitely many disjoint balls with radius $\frac{1}{2}\epsilon_{\mathrm{tol}}$, as argued in \citet{Boyd09}.

Therefore, we define
\begin{align*}
  \mathcal{T}\triangleq\left\{(u(f(\mathbf{a},\mathbf{x}_i)))^m_{i=1}~\middle |~\mathbf{a}\in \mathcal{A},~-\sum_{i=1}^mp_iu(f(\mathbf{a},\mathbf{x}_i))\leq \tilde{c}\right\}.  
\end{align*}
Then, for all iterations $t$, we have $(u(f(\mathbf{a}_t,\mathbf{x}_i)))^m_{i=1}\in \mathcal{T}$, since $\mathbf{p}\in \mathcal{U}_t$ for all $t\geq 1$, hence we have the inequality $-\sum^m_{i=1}p_iu(f(\mathbf{a}_t,\mathbf{x}_i))\leq c_t\leq \tilde{c}$. 
We show that $\mathcal{T}$ is bounded in the Euclidean 2-norm $\|.\|_2$. 
Indeed, by assumption, $M\triangleq\sup_{\mathbf{a}\in \mathcal{A},i\in [m]}u(f(\mathbf{a},\mathbf{x}_i))<\infty$. 
Hence, for all vectors $(u(f(\mathbf{a},\mathbf{x}_i)))^m_{i=1}\in \mathcal{T}$, its individual entry is always bounded from above by $M$. 
It remains to show that all its entries are also bounded from below. 
Let $p_{\min}=\min_{i=1}^mp_i>0$ (by Assumption~\ref{assump: nominal_p}) and $i_{\min}(\mathbf{a})=\mathrm{argmin}_{i}u(f(\mathbf{a},\mathbf{x}_i)$. 
Then, we have for all $(u(f(\mathbf{a},\mathbf{x}_i)))^m_{i=1}\in \mathcal{T}$,
\begin{align*}
   -p_{i_{\min}(\mathbf{a})}u(f(\mathbf{a},\mathbf{x}_{i_{\min}(\mathbf{a})}))-\sum_{i\neq i_{\min}(\mathbf{a})}p_iu(f(\mathbf{a},\mathbf{x}_i))\leq \tilde{c},
\end{align*}
which implies
\begin{align*}
    -u(f(\mathbf{a},\mathbf{x}_{i_{\min}(\mathbf{a})}))\leq& \frac{\tilde{c}+M}{p_{i_{\min}(\mathbf{a})}}\leq \frac{\tilde{c}+M}{p_{\min}}\\
    &\Leftrightarrow\\
     u(f(\mathbf{a},\mathbf{x}_{i_{\min}(\mathbf{a})}))\geq& -\frac{\tilde{c}+M}{p_{\min}}>-\infty.
\end{align*}
Hence, $\mathcal{T}$ is bounded in the Euclidean 2-norm $\|.\|_2$.
\end{proof}

\begin{proof}[\textbf{Proof of Theorem~\ref{thm:cuttingplane_conv_constr}}]
Before we proceed with the proof, we recall the notion of a KKT-vector. 
For a generic convex optimization problem 
\begin{align*}
    \min_{x\in \mathcal{X}}\{f_0(x)~|~f_i(x)\leq 0,\forall i\in [L]\},
\end{align*}
with $\mathcal{X}$ a suitable convex subset and $f_0,\ldots,f_L$ convex functions, the KKT-vector corresponding to the constraints $f_i(x)\leq 0, i\in [L]$ is a vector $\mathbf{\lambda} \in \mathbb{R}^L$ such that
\begin{align*}
    \inf_{x\in \mathcal{X}}\left\{f_0(x)+\sum^L_{i=1}\lambda_if_i(x)\right\}= \min_{x\in \mathcal{X}}\{f_0(x)~|~f_i(x)\leq 0,\forall i\in [L]\}.
\end{align*}
The existence of the KKT-vector is guaranteed when Slater's condition and the boundedness of the optimization problem are satisfied (see Theorem~28.2, \citealp{Rockafellar}).

The proof now consists of two parts. 
In the first part, we show that Algorithm~\ref{Algo: cutting_plane_constraint} terminates after finitely many iterations, for any $\epsilon_{\mathrm{tol}}>0$. 
The second part shows the convergence as $\epsilon_{\mathrm{tol}}\to 0$.

\textit{First part}:
    We reexamine the proof of Theorem~\ref{thm:cuttingplane_conv}. 
    First, we have, for any $\mathbf{a}^{1},\mathbf{a}^{2}\in \mathcal{A}$, the following inequality:
\begin{align*}
\left|\sum_{i=1}^m\bar{q}_iu(f(\mathbf{a}^{2},\mathbf{x}_i))-\sum_{i=1}^m\bar{q}_iu(f(\mathbf{a}^{1},\mathbf{x}_i))\right|\leq \|(u(f(\mathbf{a}^{2},\mathbf{x}_i))-u(f(\mathbf{a}^{1},\mathbf{x}_i)))^m_{i=1}\|_2.
\end{align*}
Let $t$ be an iteration where Algorithm~\ref{Algo: cutting_plane_constraint} has not yet terminated and 
let $(\mathbf{q}^*_t, \mathbf{\bar{q}}^*_{t})$ be the new worst-case scenarios that are added to $\mathcal{U}_t$. 
Then, we have
$    -\sum_{i=1}^m\bar{q}^*_{t,i}u(f(\mathbf{a}_t,\mathbf{x}_i))-c>\epsilon_{\mathrm{tol}}.
$ 
For any $s>t$, we also have that at the $s$-th iteration
$    -\sum_{i=1}^m\bar{q}^*_{t,i}u(f(\mathbf{a}_s,\mathbf{x}_i))-c\leq 0.
$ 
Hence, 
\begin{align*}
    \left|\sum_{i=1}^m\bar{q}^*_{t,i}u(f(\mathbf{a}_t,\mathbf{x}_i))-\sum_{i=1}^m\bar{q}^*_{t,i}u(f(\mathbf{a}_s,\mathbf{x}_i))\right|>\epsilon_{\mathrm{tol}},
\end{align*}
which implies that
\begin{align*}
   \|(u(f(\mathbf{a}_s,\mathbf{x}_i))-u(f(\mathbf{a}_t,\mathbf{x}_i)))^m_{i=1}\|_2>\epsilon_{\mathrm{tol}}.
\end{align*}
We define the set
\begin{align*}
  \mathcal{T}\triangleq\left\{(u(f(\mathbf{a},\mathbf{x}_i)))^m_{i=1}~\middle |~\mathbf{a}\in \mathcal{A},~-\sum_{i=1}^mp_iu(f(\mathbf{a},\mathbf{x}_i))\leq c\right\}.  
\end{align*}
Then, for all iterations $t$, we have $(u(f(\mathbf{a}_t,\mathbf{x}_i)))^m_{i=1}\in \mathcal{T}$, since $\mathbf{p}\in \mathcal{U}_t$ for all $t\geq 1$. 
It follows from the proof of Theorem~\ref{thm:cuttingplane_conv} that $\mathcal{T}$ must be bounded in the Euclidean 2-norm $\|.\|_2$. 
Hence, Algorithm~\ref{Algo: cutting_plane_constraint} must terminate after finitely many steps.

\textit{Second part}: Let $g(\mathbf{a}_{\epsilon_{\mathrm{tol}}})$ be the objective value of the solution $\mathbf{a}_{\epsilon_{\mathrm{tol}}}$ obtained at the final iteration of Algorithm~\ref{Algo: cutting_plane_constraint}, for a tolerance parameter $\epsilon_{\mathrm{tol}}>0$. 
Let $P(0)$ be the optimal objective value of \eqref{P-constraint}. 
Define, for any $\epsilon>0$,
\begin{align*}
    P(\epsilon)\triangleq\min_{\mathbf{a}\in \mathcal{A}}\ \left\{g(\mathbf{a})~\middle |~\sup_{(\mathbf{q},\bar{\mathbf{q}})\in \mathcal{U}_{\phi,h}(\mathbf{p})}-\sum^N_{k=1}\bar{q}_ku(x_k(\mathbf{a}))\leq c+\epsilon \right\}.
\end{align*}
By construction, we have $P(\epsilon_{\mathrm{tol}})\leq g(\mathbf{a}_{\epsilon_{\mathrm{tol}}})$ since $\mathbf{a}_{\epsilon_{\mathrm{tol}}}$ is feasible for the problem of $P(\epsilon_{\mathrm{tol}})$. 
Furthermore, the cutting-plane algorithm yields a lower bound on $P(0)$. 
Hence,
$g(\mathbf{a}_{\epsilon_{\mathrm{tol}}})\leq P(0)$ for all $\epsilon_{\mathrm{tol}}>0$. 
Therefore,
\begin{align*}
    0\leq P(0)-g(\mathbf{a}_{\epsilon_{\mathrm{tol}}})\leq P(0)-P(\epsilon_{\mathrm{tol}}).
\end{align*}
Thus, it remains to bound $P(0)-P(\epsilon_{\mathrm{tol}})$ from above. 
This can be done by utilizing the strong duality theorem, which is guaranteed by Assumption~\ref{assump: finite_strict_feasibility}. 
Therefore, we have, for any $\epsilon>0$,
\begin{align*}
     P(\epsilon)&=\sup_{\lambda\geq 0}\inf_{\mathbf{a}\in \mathcal{A}} g(\mathbf{a})+\lambda\left(\sup_{(\mathbf{q},\bar{\mathbf{q}})\in \mathcal{U}_{\phi,h}(\mathbf{p})}-\sum^N_{k=1}\bar{q}_ku(x_k(\mathbf{a}))-c-\epsilon\right)\\
     &\geq -\lambda^*\epsilon+ \inf_{\mathbf{a}\in \mathcal{A}} g(\mathbf{a})-\lambda^*\left(\sup_{(\mathbf{q},\bar{\mathbf{q}})\in \mathcal{U}_{\phi,h}(\mathbf{p})}-\sum^N_{k=1}\bar{q}_ku(x_k(\mathbf{a}))-c\right)\\
     &=-\lambda^*\epsilon+P(0),
\end{align*}
where $\lambda^*$ is the KKT-vector of the problem $P(0)$ corresponding to the supremum constraint, which is a strictly positive constant (by Assumption~\ref{assump: finite_strict_feasibility}) independent of $\epsilon$. 
Hence, we have
\begin{align*}
    P(0)-P(\epsilon_{\mathrm{tol}})\leq \lambda^*\epsilon_{\mathrm{tol}}.
\end{align*}
Therefore, the convergence follows as $\epsilon_{\mathrm{tol}}\to 0$.
\end{proof}

\begin{proof}[\textbf{Proof of Lemma~\ref{lem: optimal_ranking}}]
We first show that, for any $\mathbf{a}\in \mathcal{A}$, we have that
\begin{align}\label{equality_reduced_uc}
\sup_{(\mathbf{q},\bar{\mathbf{q}})\in \mathcal{U}_{\phi,h}(\mathbf{p})}\sum^m_{i=1}-\bar{q}_iu(f(\mathbf{a},\mathbf{x}_i))=\sup_{(\mathbf{q},\bar{\mathbf{q}})\in \mathcal{U}^{(i_1,\ldots, i_m)}_{\phi,h}(\mathbf{p})}-\sum^m_{i=1}\bar{q}_iu(f(\mathbf{a},\mathbf{x}_i)),
\end{align}
for any ranking $(i_1,\ldots, i_m)\in \mathcal{I}(\mathbf{a})$ as in Definition~\ref{def: ranked_index}.
We fix $\mathbf{a}$, denote $Y\triangleq f(\mathbf{a},\mathbf{X})$ and consider the ranking
$u(y_{(1)})\geq \ldots \geq u(y_{(m)})$. 
By the definition of $\rho_{u,h,\mathbf{q}}(Y)$ in \eqref{def:distortionriskmeasure}, we have that in our discrete setting $\rho_{u,h,\mathbf{q}}(Y)$ is equal to the rank-dependent sum
\begin{align*}
    \rho_{u,h,\mathbf{q}}(Y)=-\sum^{m}_{i=1}h\left(\sum^m_{j=i}q_{(j)}\right)(u(y_{(i)})-u(y_{(i-1)})),
\end{align*}
where $-u(y_{(0)})\triangleq 0$.  
We now claim that we have an equality between the rank-dependent sum above and the following optimization problem:
\begin{align*}
    \max_{\bar{\mathbf{q}}\in \mathbb{R}^m}&\quad\sum^m_{i=1}-\left(\sum^m_{j=i}\bar{q}_{(j)}\right)(u(y_{(i)})-u(y_{(i-1)}))\\
    \mathrm{subject~to}&\quad \sum^m_{j=i}\bar{q}_{(j)}\leq h\left(\sum^m_{j=i} q_{(j)}\right),\quad \forall i\in [m]\\
    &\quad \sum^m_{i=1}\bar{q}_i=1\\
    &\quad\bar{q}_i\geq 0,\quad \forall i\in [m].
\end{align*}
Indeed, we can define $\bar{q}^*_{(i)}\triangleq h\left(\sum^m_{j=i}q_{(j)}\right)-h\left(\sum^m_{j=i+1}q_{(j)}\right)$, where $\bar{q}^*_{(m)}\triangleq h(q_{(m)})$. 
Then, $\sum^m_{j=i}\bar{q}^*_{(i)}=h\left(\sum^m_{j=i} q_{(j)}\right)$, and $\bar{\mathbf{q}}^*$ is a probability vector since $h(1)=1$, $h(0)=0$, and $h$ is non-decreasing. 
Hence, $\bar{\mathbf{q}}^*$ is feasible for the above optimization problem. 
Furthermore, since $-(u(y_{(i)})-u(y_{(i-1)}))\geq 0$ for all $i\geq 2$, the maximum is attained at a vector $\bar{\mathbf{q}}$ such that the constraint $\sum^m_{j=i}\bar{q}_{(i)}\leq h\left(\sum^m_{j=i} q_{(j)}\right)$ is an equality for all $i$, which uniquely defines $\bar{\mathbf{q}}^*$. 
An expansion of the alternating sum $\sum^m_{i=1}-\left(\sum^m_{j=i}q_{(j)}\right)(u(y_{(i)})-u(y_{(i-1)}))$ shows that it is also equal to the sum $\sum^m_{i=1}-q_{(i)}u(y_{(i)})$. 
Therefore, \eqref{equality_reduced_uc} holds for all $\mathbf{a}\in \mathcal{A}$ and any ranking $(i_1,\ldots, i_m)\in \mathcal{I}(\mathbf{a})$. 
Let $V^*$ be the optimal objective value of \eqref{P-constraint}. 
Then, for any $\mathbf{a}_0$ such that $\mathcal{I}(\mathbf{a}_0)\subset \mathcal{I}(\mathbf{a}^*)$, we have that $U^*(\mathbf{a}_0)\leq V^*$, since $\mathbf{a}^*$ is feasible for the problem \eqref{ranked_upper} by \eqref{equality_reduced_uc}. 
On the other hand, we also have the upper bound relation $V^*\leq U^*(\mathbf{a}_0)$, since $\mathcal{U}_{\phi,h}(\mathbf{p})\subset \mathcal{U}^{(i_1,\ldots, i_m)}_{\phi,h}(\mathbf{p})$ for any index vector $(i_1,\ldots, i_m)$. 
Hence, $U^*(\mathbf{a}_0)=V^*$.
\end{proof}

\begin{proof}[\textbf{Proof of Theorem~\ref{thm: convergence_cutting_plane_constraint}}]
Let $(\mathbf{q}^*_n,\bar{\mathbf{q}}^*_n)\in \mathrm{argmax}_{(\mathbf{q},\bar{\mathbf{q}})\in \mathcal{U}_{\phi,h}(\mathbf{p})}-\sum^m_{i=1}\bar{q}_iu(f(\mathbf{a}_n,\mathbf{x}_i))$. 
Denote $g(\mathbf{a},\mathbf{q},\bar{\mathbf{q}})=-\sum^m_{i=1}\bar{q}_iu(f(\mathbf{a},\mathbf{x}_i))$. 
Since the set $\mathcal{A}\times \mathcal{U}_{\phi,h}(\mathbf{p})$ is compact,\footnote{The set $\mathcal{U}_{\phi,h}(\mathbf{p})$ is compact due to the lower-semicontinuity assumption made in Assumption~\ref{assump:lower_semi}.} we may assume that there exists a limit $(\mathbf{a}_n,\mathbf{q}^*_n,\bar{\mathbf{q}}^*_n)\to (\mathbf{a}_L,\mathbf{q}_L,\bar{\mathbf{q}}_L)\in \mathcal{A}\times \mathcal{U}_{\phi,h}(\mathbf{p})$, as $n\to \infty$. 
We will show that $\mathbf{a}_L$ must be an optimal solution for \eqref{P-constraint}. 
Indeed, since $(\mathbf{q}^*_n,\bar{\mathbf{q}}^*_n)$ are maximizers, we have that 
        \begin{align*}
c+\epsilon_{\mathrm{tol},n}\geq -\sum^m_{i=1}\bar{q}^*_{n,i}u(f(\mathbf{a}_n,\mathbf{x}_i))\geq  \max_{(\mathbf{q},\bar{\mathbf{q}})\in \mathcal{U}_{\phi,h}(\mathbf{p})}-\sum^m_{i=1}\bar{q}_{i}u(f(\mathbf{a}_n,\mathbf{x}_i)).    
        \end{align*}
Taking the limit as $n\to \infty$ yields
\begin{align*}
           c\geq -\sum^m_{i=1}\bar{q}^*_{L,i}u(f(\mathbf{a}_L,\mathbf{x}_i))\geq \max_{(\mathbf{q},\bar{\mathbf{q}})\in \mathcal{U}_{\phi,h}(\mathbf{p})} -\sum^m_{i=1}\bar{q}_{i}u(f(\mathbf{a}_L,\mathbf{x}_i)).
\end{align*}
Hence, by Theorem~\ref{thm:ReformulationWholeProb}, this implies that $\mathbf{a}_L$ is feasible for \eqref{P-constraint}. 
Since $\mathbf{a}_n$ is a sequence of solutions such that $g(\mathbf{a}_n)$ converges to the optimal objective value of \eqref{P-constraint} by Theorem~\ref{thm:cuttingplane_conv_constr}, it follows from the continuity of $g$ that $g(\mathbf{a}_L)$ equals to the optimal objective value of \eqref{P-constraint}. 
Hence, $\mathbf{a}_L$ is an optimal solution of \eqref{P-constraint}. 
Continuity of the functions $-u(f(\mathbf{a},\mathbf{x}_i))$ in $\mathbf{a}$, for all $i\in [m]$, implies that there exists some $N>0$, such that for all $n\geq N$, we have the inclusion of the ranking set: $\mathcal{I}(\mathbf{a}_n)\subset \mathcal{I}(\mathbf{a}_L)$. 
Therefore, Lemma~\ref{lem: optimal_ranking} implies that $U^*(\mathbf{a}_n)$ converges to the optimal objective value of \eqref{P-constraint}.
\end{proof}

\begin{proof}[\textbf{Proof of Lemma~\ref{lem:PL_UC}}]
If $(\mathbf{q},\bar{\mathbf{q}}, \mathbf{t})$ satisfies the constraints in \eqref{PL_UC}, then $\bar{q}_i\leq l_j\cdot q_i+t_{i,j}$ for all $i\in [m]$ and $j\in [K]$. Then, by the non-negativity of the variables $t_{i,j}$, we also have that, for any subset $I\subset [m]$ and any $j\in [K]$,
\begin{align*}
    \sum_{i\in I}\bar{q}_i\leq l_j\sum_{i\in I}q_i+\sum_{i\in I} t_{i,j}\leq l_j\sum_{i\in I}q_i+\sum^m_{i=1} t_{i,j}\leq l_j\sum_{i\in I}q_i+b_j.  
\end{align*}
Hence, $\sum_{i\in I}\bar{q}_i\leq \min_{j\in [K]}h_j(\sum_{i\in I}q_i)$ and thus $(\mathbf{q},\bar{\mathbf{q}})\in \mathcal{U}_{\phi,h}(\mathbf{p})$. 

Conversely, let $(\mathbf{q},\bar{\mathbf{q}})\in \mathcal{U}_{\phi,h}(\mathbf{p})$. 
Then, $\bar{q}_i\leq l_j\cdot q_i+t_{i,j}$ for all $i\in [m]$ and $j\in [K]$ with $t_{i,j}\triangleq\max\{\bar{q}_i-l_jq_i,0\}$. 
Moreover, we have that, for all $j\in [K]$,
\begin{align*}
    \sum^m_{i=1}t_{i,j}=\sum^m_{i=1}\max\{\bar{q}_i-l_jq_i,0\}=\sum_{i\in I_+}\bar{q}_i-\sum_{i\in I_+}l_jq_i\leq b_j,~\text{where}~ I_+\triangleq\{i: \bar{q}_i-l_jq_i\geq 0\}.
\end{align*}
Hence, $(\mathbf{q},\bar{\mathbf{q}},\mathbf{t})$ satisfies the constraints in \eqref{PL_UC}.
\end{proof}

\begin{proof}[\textbf{Proof of Theorem~\ref{thm:RCpiecewiselin}}]
We reformulate the constraint
\begin{align*}
\max_{(\mathbf{q},\mathbf{\bar{q}},\mathbf{t})\in \mathcal{\bar{U}}}-\sum^m_{i=1}\bar{q}_iu(f(\mathbf{a},\mathbf{x}_i))\leq c,
\end{align*}
where
\begin{align*}
    \mathcal{\bar{U}}\triangleq\left\{(\mathbf{q},\bar{\mathbf{q}},\mathbf{t})\in \mathbb{R}^{2m}_{\geq 0}\times \mathbb{R}^{mK}_{\geq 0} \; 
    \begin{tabular}{|l}
    $\sum^m_{i=1} q_i=1,\sum^m_{i=1} \bar{q}_i=1$, $\sum^m_{i=1}p_i\phi\left(\frac{q_i}{p_i}\right)\leq r$\\
    $\sum^m_{i=1}t_{i,j}\leq b_j,~\forall j\in [K]$\\
    $\bar{q}_i-l_jq_i\leq t_{i,j},~\forall i\in [m],\ \forall j\in [K]$
    \end{tabular}
    \right\}.
\end{align*}
We note that $(\mathbf{p},\mathbf{p},\{\max\{(l_j-1)p_i,0\}_{i,j}\})$ is a Slater point in $\bar{\mathcal{U}}$ and that the left-hand side of the constraint above constitutes a maximization problem upper bounded by $c\in \mathbb{R}$. 
Therefore, strong duality holds, and we examine the Lagrangian function
\begin{align*}
    &L(\mathbf{a},\mathbf{q}, \bar{\mathbf{q}},\alpha,\beta,\gamma,\lambda_{ij}, t_{i,j},\nu_j)\\
    &=-\sum^m_{i=1}\bar{q}_iu(f(\mathbf{a},\mathbf{x}_i))-\alpha\left(\sum^m_{i=1}q_i-1\right)-\beta\left(\sum^m_{i=1}\bar{q}_i-1\right)-\gamma\left(\sum^m_{i=1}p_i\phi\left(\frac{q_i}{p_i}\right)-r\right)\\
    &~~~~~~-\sum^m_{i=1}\sum^K_{j=1}\lambda_{ij}(\bar{q}_i-l_jq_i-t_{i,j})-\sum^K_{j=1}\nu_j\left(\sum^m_{i=1}t_{i,j}-b_j\right)\\
    &=\alpha+\beta+\gamma r+\sum^K_{j=1}\nu_jb_j-\sum^m_{i=1}\left(u(f(\mathbf{a},\mathbf{x}_i))+\beta+\sum^K_{j=1}\lambda_{ij}\right)\bar{q}_i\\
    &~~~~~~+\sum^m_{i=1}p_i\left(\left(-\alpha+\sum^K_{j=1}\lambda_{ij}l_j\right)\frac{q_i}{p_i}-\gamma \phi\left(\frac{q_i}{p_i}\right)\right)+\sum^m_{i=1}\sum^K_{j=1}(\lambda_{ij}-\nu_j)t_{i,j}.
\end{align*}
We have
\begin{align}\label{id_sup_q}
\begin{split}
   &\sup_{\bar{q}_1,\ldots, \bar{q}_m\geq 0}-\sum^m_{i=1}\left(u(f(\mathbf{a},\mathbf{x}_i))+\beta+\sum^K_{j=1}\lambda_{ij}\right)\bar{q}_i\\
   &=\begin{cases}
        0&~\text{if}~u(f(\mathbf{a},\mathbf{x}_i))+\beta+\sum^K_{j=1}\lambda_{ij}\geq 0,~\forall i\in[m]\\
        \infty&~\text{else},
    \end{cases} 
\end{split}
\end{align}
\begin{align*}
    \sup_{q_1,\ldots, q_m\geq 0}\sum^m_{i=1}p_i\left(\left(-\alpha+\sum^K_{j=1}\lambda_{ij}l_j\right)\frac{q_i}{p_i}-\gamma \phi\left(\frac{q_i}{p_i}\right)\right)=\sum^m_{i=1}p_i\gamma\phi^*\left(\frac{-\alpha+\sum^K_{j=1}\lambda_{ij}l_j}{\gamma}\right),
\end{align*}
and
\begin{align}\label{id_sup_t}
    \sup_{t_{1,1},\ldots, t_{m,K}\geq 0}\sum^m_{i=1}\sum^K_{j=1}(\lambda_{ij}-\nu_j)t_{i,j}=\begin{cases}
        0&~\text{if}~\lambda_{ij}\leq \nu_j,~ \forall i\in [m],~ \forall j\in [K]\\
        \infty&~\text{else}.
    \end{cases}
\end{align}
Therefore, employing the same arguments as in the proof of Theorem~\ref{thm:RC}, the reformulated robust counterpart is given by
\begin{align*}
\begin{cases}
    \alpha +\beta+\gamma r +\sum^K_{j=1}\nu_jb_j+\sum^m_{i=1}p_i\gamma\phi^*\left(\frac{-\alpha+\sum^K_{j=1}\lambda_{ij}l_j}{\gamma}\right)\leq c\\
    -u(f(\mathbf{a},\mathbf{x}_i))-\beta-\sum^K_{j=1}\lambda_{ij}\leq 0,~\forall i\in [m]\\
    \lambda_{ij}\leq \nu_j,~\forall i\in [m],~\forall j\in [K]\\
    \alpha,\beta\in \mathbb{R}, \gamma, \lambda_{ij},\nu_j\geq 0.
\end{cases}
\end{align*}

In the nominal case where $\mathbf{q}=\mathbf{p}$, we have to reformulate the following constraint:
\begin{align*}
    \max_{(\mathbf{\bar{q}},\mathbf{t})\in \mathcal{\bar{U}}_{\mathrm{nom}}}-\sum^m_{i=1}\bar{q}_iu(f(\mathbf{a},\mathbf{x}_i))\leq c,
\end{align*}
where
\begin{align*}
    \mathcal{\bar{U}}_{\mathrm{nom}}\triangleq\left\{(\bar{\mathbf{q}},\mathbf{t})\in \mathbb{R}^{m}_{\geq 0}\times \mathbb{R}^{mK}_{\geq 0} \; 
    \begin{tabular}{|l}
    $\sum^m_{i=1} \bar{q}_i=1$\\
    $\sum^m_{i=1}t_{i,j}\leq b_j,~\forall j\in [K]$\\
    $\bar{q}_i-l_jp_i\leq t_{i,j},~\forall i\in [m],\ \forall j\in [K]$
    \end{tabular}
    \right\}.
\end{align*}
The above maximization problem is a linear programming problem bounded from above. 
Therefore, strong duality applies, and we examine the Lagrangian function
\begin{align*}
  &L(\mathbf{a}, \bar{\mathbf{q}},\beta,\lambda_{ij}, t_{i,j},\nu_j)\\
    &=-\sum^m_{i=1}\bar{q}_iu(f(\mathbf{a},\mathbf{x}_i))-\beta\left(\sum^m_{i=1}\bar{q}_i-1\right)-\sum^m_{i=1}\sum^K_{j=1}\lambda_{ij}(\bar{q}_i-l_jp_i-t_{i,j})-\sum^K_{j=1}\nu_j\left(\sum^m_{i=1}t_{i,j}-b_j\right)\\
    &=\beta+\sum^K_{j=1}\nu_jb_j+\sum^m_{i=1}\sum^K_{j=1}\lambda_{ij}l_jp_i+\sum^m_{i=1}\bar{q}_i\left(-u(f(\mathbf{a},\mathbf{x}_i))-\beta-\sum^K_{j=1}\lambda_{ij}\right)+\sum^m_{i=1}\sum^K_{j=1}(\lambda_{ij}-\nu_j)t_{i,j}.
\end{align*}
By \eqref{id_sup_q}--\eqref{id_sup_t}, we have that
\begin{align*}
    &\inf_{\substack{\lambda_{ij},\nu_j\geq 0\\ \beta\in \mathbb{R}}}\sup_{\bar{\mathbf{q}},\mathbf{t}\geq \mathbf{0}}L(\mathbf{a}, \bar{\mathbf{q}},\beta,\lambda_{ij}, t_{i,j},\nu_j)\\
    &=\inf_{\substack{\lambda_{ij},\nu_j\geq 0\\ \beta\in \mathbb{R}}}\left\{\beta+\sum^K_{j=1}\nu_jb_j+\sum^m_{i=1}\sum^K_{j=1}\lambda_{ij}l_jp_i~\middle|~-u(f(\mathbf{a},\mathbf{x}_i))-\beta-\sum^K_{j=1}\lambda_{ij}\leq 0, \forall i,~\lambda_{ij}\leq \nu_j, \forall i,j\right\}.
\end{align*}
Hence, the reformulated robust counterpart is given by
\begin{align*}
    \begin{cases}
        \beta+\sum^K_{j=1}\nu_jb_j+\sum^m_{i=1}\sum^K_{j=1}\lambda_{ij}l_jp_i\leq c\\
        -u(f(\mathbf{a},\mathbf{x}_i))-\beta-\sum^K_{j=1}\lambda_{ij}\leq 0,~ \forall i\in [m]\\
        \lambda_{ij}\leq \nu_j,~ \forall i\in [m],~\forall j\in [K]\\
        \lambda_{ij},\nu_j\geq 0 ,~ \forall i\in [m],~\forall j\in [K].
    \end{cases}
\end{align*}
\end{proof}

\begin{proof}[\textbf{Proof of Lemma~\ref{lem:PL_h_construct}}]
    First, we have that
    \begin{align*}
         e_i(x_{i+1})=\sup_{x\in [x_i,1]}\left\{h(x)-\frac{h(x_{i+1})-h(x_i)}{x_{i+1}-x_i}(x-x_i)-h(x_i)\right\},
    \end{align*}
    since, by concavity, the supremum is only taken in the interval $[x_i,x_{i+1}]$. 
    Therefore, we can extend the feasibility region to $[x_i,1]$.

    Next, again by concavity, we have that $-\frac{h(x_{i+1})-h(x_i)}{x_{i+1}-x_i}(x-x_i)$ is an increasing function of $x_{i+1}$, for all $x\in (x_i,1]$. 
    Therefore, $e_i(x_{i+1})$ is increasing in $x_{i+1}$. 
    It is also a continuous function in $x_{i+1}$. 
    This is because the function 
    \begin{align*}
        \tilde{e}_i(y)\triangleq\sup_{x\in [x_i,1]}\left\{h(x)-y(x-x_i)-h(x_i)\right\},
    \end{align*}
    is convex in $y$; 
    indeed, it is a supremum of a linear function of $y$. 
    Hence, $\tilde{e}_i$ is continuous on the interior of its domain, which is the whole of $\mathbb{R}$, because $\tilde{e}_i(y)$ is a supremum of a continuous function on a compact interval. 
    Thus, $\tilde{e}_i(y)$ exists and is finite everywhere. 
    This implies that $e_i$ is continuous for all $x_{i+1}\in (x_i,1]$, since $e_i(x_{i+1})=\tilde{e}_i\left(\frac{h(x_{i+1})-h(x_i)}{x_{i+1}-x_i}\right)$ is a composition of continuous functions. 

    Finally, we show the existence of a $x_{i+1}$ such that $e_i(x_{i+1})=\epsilon$, for any given $\epsilon>0$, under the assumption that $e_{i}(1)>\epsilon$. 
    This is guaranteed if we can find a $z\in (x_i,1)$ such that $e_{i}(z)<\epsilon$. 
    Since $h$ has decreasing slope and is increasing, we have that for any $y\in (x_i,1)$ and all $x\in [x_i,y]$:
    \begin{align*}
        h(x)-\frac{h(y)-h(x_i)}{y-x_i}(x-x_i)-h(x_i)\leq h(y)-h(x_i).
    \end{align*}
    We note that the maximization problem in $e_i(y)$ can be restricted to the interval $[x_i,y]$. 
    Therefore, 
        $e_i(y)\leq h(y)-h(x_i)$.
    Taking $y\downarrow x_i$, it follows by the continuity of $h$ that such a point $z$ must exist.
\end{proof}

\begin{proof}[\textbf{Proof of Theorem~\ref{thm:convergencePL} (Part I)}]
We split the proof into two parts. 
We first treat the case of problem \eqref{P} and next the case of problem \eqref{P-constraint}. 
Fix an $\mathbf{a}\in \mathcal{A}$ and consider the ranked realizations 
$         u(f(\mathbf{a},\mathbf{x}_{(1)}))\geq \ldots \geq u(f(\mathbf{a},\mathbf{x}_{(m)}))$. 
Let $h_1,h_2$ be any two concave distortion functions such that $\sup_{t\in [0,1]}|h_1(t)-h_2(t)|\leq \epsilon$. 
Set $u(f(\mathbf{a},\mathbf{x}_{(0)}))\triangleq 0$. 
Then, we have
\begin{align*}
\rho_{u,h_1,\mathbf{q}}(f(\mathbf{a},\mathbf{X}))&=\sum^{m}_{i=1}h_1\left(\sum^m_{j=i}q_{(j)}\right)(u(f(\mathbf{a},\mathbf{x}_{(i-1)}))-u(f(\mathbf{a},\mathbf{x}_{(i)})))\\
    &\leq -u(f(\mathbf{a},\mathbf{x}_{(1)}))+ \sum^{m}_{i=2}\left( h_2\left(\sum^m_{j=i}q_{(j)}\right)+\epsilon\right)(u(f(\mathbf{a},\mathbf{x}_{(i-1)}))-u(f(\mathbf{a},\mathbf{x}_{(i)})))\\
    &=\rho_{u,h_{2},\mathbf{q}}(f(\mathbf{a},\mathbf{X})) + \epsilon\cdot \left(\max_{\substack{i\in [m]}}u(f(\mathbf{a},\mathbf{x}_i))-\min_{i\in [m]} u(f(\mathbf{a},\mathbf{x}_i))\right)\\
    &\leq \rho_{u,h_{2},\mathbf{q}}(f(\mathbf{a},\mathbf{X}))+ \epsilon\cdot \left(M+\max_{\substack{i\in [m]}}-u(f(\mathbf{a},\mathbf{x}_i))\right),
\end{align*}
where $M\triangleq\sup_{\mathbf{a}\in \mathcal{A},i\in [m]}u(f(\mathbf{a},\mathbf{x}_i))<\infty$. 
The idea now is to bound the term $\max_{\substack{i\in [m]}}-u(f(\mathbf{a},\mathbf{x}_i))$ on a subset $\mathcal{A}_0\subset \mathcal{A}$ of the feasible region, for which the restriction of the minimization problem \eqref{P} on $\mathcal{A}_0$ does not change the original optimal value.

We take $h_1\equiv h$ and $h_2\in \{h_\epsilon, \tilde{h}_\epsilon\}$. 
By Assumption~\ref{assumption:well-defined}, there exists $\mathbf{a}_0\in \mathcal{A}$ such that
$\sup_{\mathbf{q}\in \mathcal{D}_{\phi}(\mathbf{p},r)}\rho_{u,h,\mathbf{q}}(f(\mathbf{a}_0,\mathbf{X}))<\infty$. 
This implies $\rho_{u,h,\mathbf{p}}(f(\mathbf{a}_0,\mathbf{X}))<\infty$, hence $\max_{i\in [m]}-u(f(\mathbf{a}_0,\mathbf{x}_i))$ $<\infty$. 
Then, we also have that
\begin{align*}
    \sup_{\mathbf{q}\in \mathcal{D}_{\phi}(\mathbf{p},r)}\rho_{u,h_2,\mathbf{q}}(f(\mathbf{a}_0,\mathbf{X}))\leq \max_{i\in [m]}-u(f(\mathbf{a}_0,\mathbf{x}_i))<\infty.
\end{align*}
Therefore, we may define the finite number
\begin{align*}
    c_0\triangleq\max_{j=1,2}\left\{\sup_{\mathbf{q}\in \mathcal{D}_{\phi}(\mathbf{p},r)}\rho_{u,h_j,\mathbf{q}}(f(\mathbf{a}_0,\mathbf{X}))\right\}<\infty.
\end{align*}
Then, for all $\mathbf{a}\in \mathcal{A}$ such that $\sup_{\mathbf{q}\in \mathcal{D}_{\phi}(\mathbf{p},r)}\rho_{u,h_j,\mathbf{q}}(f(\mathbf{a},\mathbf{X}))$ $\leq c_0$ for any $j\in \{1,2\}$, we have that 
\begin{align}
    -\sum^m_{i=1}p_iu(f(\mathbf{a},\mathbf{x}_i))\leq \rho_{u,h_j,\mathbf{p}}(f(\mathbf{a},\mathbf{X}))\leq c_0,
\end{align}
due to concavity (which implies $h_j(x)\geq x, \forall x\in [0,1]$, for any $j\in \{1,2\}$). Therefore, for all such $\mathbf{a}$, we have,
\begin{align}\label{bound_wc_uf(a)}
\sup_{i\in [m]}-u(f(\mathbf{a},\mathbf{x}_i))\leq \frac{c_0+M}{p_{\min}},    
\end{align}
as shown in the proof of Theorem~\ref{thm:cuttingplane_conv}. 
Define,
\begin{align*}
    \mathcal{A}_0\triangleq\left\{\mathbf{a}\in \mathcal{A}~\middle |~\sup_{i\in [m]}-u(f(\mathbf{a},\mathbf{x}_i))\leq \frac{c_0+M}{p_{\min}}\right\}.
\end{align*}
Then, we have that for any $j=1,2$:
\begin{align*}
    \min_{\mathbf{a}\in \mathcal{A}}\sup_{\mathbf{q}\in \mathcal{D}_{\phi}(\mathbf{p},r)}\rho_{u,h_j,\mathbf{q}}(f(\mathbf{a},\mathbf{X}))=\min_{\mathbf{a}\in \mathcal{A}_0}\sup_{\mathbf{q}\in \mathcal{D}_{\phi}(\mathbf{p},r)}\rho_{u,h_j,\mathbf{q}}(f(\mathbf{a},\mathbf{X})),
\end{align*}
since any potential minimizer $\mathbf{a}$ satisfies $\sup_{\mathbf{q}\in \mathcal{D}_{\phi}(\mathbf{p},r)}\rho_{u,h_j,\mathbf{q}}(f(\mathbf{a},\mathbf{X}))\leq c_0$ and thus \eqref{bound_wc_uf(a)}. Therefore,
\begin{align*}
   \min_{\mathbf{a}\in \mathcal{A}}\sup_{\mathbf{q}\in \mathcal{D}_{\phi}(\mathbf{p},r)}\rho_{u,h_1,\mathbf{q}}(f(\mathbf{a},\mathbf{X}))&=\min_{\mathbf{a}\in \mathcal{A}_0}\sup_{\mathbf{q}\in \mathcal{D}_{\phi}(\mathbf{p},r)}\rho_{u,h_1,\mathbf{q}}(f(\mathbf{a},\mathbf{X}))\\
   &\leq \min_{\mathbf{a}\in \mathcal{A}_0}\sup_{\mathbf{q}\in \mathcal{D}_{\phi}(\mathbf{p},r)}\rho_{u,h_2,\mathbf{q}}(f(\mathbf{a},\mathbf{X}))+\epsilon \cdot \left(M+\frac{c_0+M}{p_{\min}}\right)\\
   &=\min_{\mathbf{a}\in \mathcal{A}}\sup_{\mathbf{q}\in \mathcal{D}_{\phi}(\mathbf{p},r)}\rho_{u,h_2,\mathbf{q}}(f(\mathbf{a},\mathbf{X}))+\epsilon \cdot \left(M+\frac{c_0+M}{p_{\min}}\right).
\end{align*}
By symmetry, we thus have
\begin{align*}
    \left|\min_{\mathbf{a}\in \mathcal{A}}\sup_{\mathbf{q}\in \mathcal{D}_{\phi}(\mathbf{p},r)}\rho_{u,h_1,\mathbf{q}}(f(\mathbf{a},\mathbf{X}))-\min_{\mathbf{a}\in \mathcal{A}}\sup_{\mathbf{q}\in \mathcal{D}_{\phi}(\mathbf{p},r)}\rho_{u,h_2,\mathbf{q}}(f(\mathbf{a},\mathbf{X}))\right|\leq \epsilon \cdot \left(M+\frac{c_0+M}{p_{\min}}\right),
\end{align*}
which approaches zero as $\epsilon\to 0$. 
\end{proof}

 \begin{proof}[\textbf{Proof of Theorem~\ref{thm:convergencePL} (Part II)}]
We define
    \begin{align*}
        L_\epsilon\triangleq \min_{\mathbf{a}\in \mathcal{A}}\ \left\{g(\mathbf{a})~\middle |~\sup_{\mathbf{q}\in\mathcal{D}_{\phi}(\mathbf{p},r)}\rho_{u,h_\epsilon,\mathbf{q}}(f(\mathbf{a},\mathbf{X}))\leq c \right\},\\
        U_\epsilon\triangleq \min_{\mathbf{a}\in \mathcal{A}}\ \left\{g(\mathbf{a})~\middle |~\sup_{\mathbf{q}\in\mathcal{D}_{\phi}(\mathbf{p},r)}\rho_{u,\tilde{h}_\epsilon,\mathbf{q}}(f(\mathbf{a},\mathbf{X}))\leq c \right\},
    \end{align*}
and let $P_0$ denote the optimal objective value of \eqref{P-constraint}. 
By definition, $L_\epsilon\leq P_0\leq U_\epsilon$.

We first note that $L_\epsilon$ has a nonempty feasible set, which follows from Assumption~\ref{assump: finite_strict_feasibility} and the fact that $\rho_{u,h_\epsilon,\mathbf{q}}(f(\mathbf{a},\mathbf{X}))\leq \rho_{u,h,\mathbf{q}}(f(\mathbf{a},\mathbf{X}))$ for all $\mathbf{a}\in \mathcal{A}$ and any probability vector $\mathbf{q}$. 
Then, following the first part of the proof of Theorem~\ref{thm:convergencePL}, we can show that for any $\mathbf{a}\in \mathcal{A}$ that is feasible for $L_\epsilon$, we have that, for all $\mathbf{q}$,
    \begin{align*}
        \rho_{u,h_\epsilon,\mathbf{q}}(f(\mathbf{a},\mathbf{X}))\geq \rho_{u,h,\mathbf{q}}(f(\mathbf{a},\mathbf{X}))-\epsilon\left(M+\frac{c+M}{p_{\min}}\right). 
    \end{align*}
Therefore, we also have the following implication for any $\mathbf{a}\in \mathcal{A}$: 
    \begin{align*}
     \sup_{\mathbf{q}\in\mathcal{D}_{\phi}(\mathbf{p},r)}\rho_{u,h_\epsilon,\mathbf{q}}(f(\mathbf{a},\mathbf{X}))\leq c \Rightarrow   \sup_{\mathbf{q}\in\mathcal{D}_{\phi}(\mathbf{p},r)}\rho_{u,h,\mathbf{q}}(f(\mathbf{a},\mathbf{X}))\leq c+\epsilon\left(M+\frac{c+M}{p_{\min}}\right).
    \end{align*}
Hence,
    \begin{align*}
        L_\epsilon&\geq \min_{\mathbf{a}\in \mathcal{A}}\ \left\{g(\mathbf{a})~\middle |~\sup_{\mathbf{q}\in\mathcal{D}_{\phi}(\mathbf{p},r)}\rho_{u,h,\mathbf{q}}(f(\mathbf{a},\mathbf{X}))\leq c+\epsilon\left(M+\frac{c+M}{p_{\min}}\right)\right\}\\
        &\stackrel{(*)}{\geq} \sup_{\lambda\geq 0}\inf_{\mathbf{a}\in \mathcal{A}}g(\mathbf{a})+\lambda \left(\sup_{\mathbf{q}\in\mathcal{D}_{\phi}(\mathbf{p},r)}\rho_{u,h,\mathbf{q}}(f(\mathbf{a},\mathbf{X}))-c-\epsilon\left(M+\frac{c+M}{p_{\min}}\right)\right)\\
        &\geq \inf_{\mathbf{a}\in \mathcal{A}}g(\mathbf{a})+\lambda^* \left(\sup_{\mathbf{q}\in\mathcal{D}_{\phi}(\mathbf{p},r)}\rho_{u,h,\mathbf{q}}(f(\mathbf{a},\mathbf{X}))-c-\epsilon\left(M+\frac{c+M}{p_{\min}}\right)\right)\\
        &\stackrel{(**)}{=}P_0-\lambda^*\cdot \epsilon\left(M+\frac{c+M}{p_{\min}}\right),
    \end{align*}
where in $(*)$ we used weak duality and in $(**)$ we used $\lambda^*$, the KKT-vector of \eqref{P-constraint} corresponding to the constraint $\sup_{\mathbf{q}\in\mathcal{D}_{\phi}(\mathbf{p},r)}\rho_{u,h,\mathbf{q}}(f(\mathbf{a},\mathbf{X}))\leq c$, which is a strictly positive constant by Assumption~\ref{assump: finite_strict_feasibility} and does not depend on $\epsilon$. 
Therefore, we have
    \begin{align*}
        0\leq P_0-L_\epsilon\leq \lambda^*\epsilon\left(M+\frac{c+M}{p_{\min}}\right),
    \end{align*}
which approaches zero as $\epsilon\to 0$.

Similarly, we have the following implication for any $\mathbf{a}\in \mathcal{A}$:
    \begin{align*}
     \sup_{\mathbf{q}\in\mathcal{D}_{\phi}(\mathbf{p},r)}\rho_{u,h,\mathbf{q}}(f(\mathbf{a},\mathbf{X}))\leq c-\epsilon\left(M+\frac{c+M}{p_{\min}}\right) \Rightarrow   \sup_{\mathbf{q}\in\mathcal{D}_{\phi}(\mathbf{p},r)}\rho_{u,\tilde{h}_\epsilon,\mathbf{q}}(f(\mathbf{a},\mathbf{X}))\leq c. 
    \end{align*}
Therefore, 
    \begin{align}\label{U_epsilon_bound}
        U_\epsilon &\leq \min_{\mathbf{a}\in \mathcal{A}}\ \left\{g(\mathbf{a})~\middle |~\sup_{\mathbf{q}\in\mathcal{D}_{\phi}(\mathbf{p},r)}\rho_{u,h,\mathbf{q}}(f(\mathbf{a},\mathbf{X}))\leq c-\epsilon\left(M+\frac{c+M}{p_{\min}}\right)\right\}.
    \end{align}
We note that the minimization problem on the right-hand side of \eqref{U_epsilon_bound} contains a Slater point, for $\epsilon$ sufficiently small. 
Indeed, by Assumption~\ref{assump: finite_strict_feasibility}, there exists a point $\mathbf{a}_0\in  \mathrm{int}(\mathcal{A})$, such that
    \begin{align*}
        \sup_{\mathbf{q}\in\mathcal{D}_{\phi}(\mathbf{p},r)}\rho_{u,h,\mathbf{q}}(f(\mathbf{a}_0,\mathbf{X}))\triangleq c_0<c.
    \end{align*}
Then, for $\epsilon\leq \frac{c-c_0}{2\left(M+\frac{c+M}{p_{\min}}\right)}$, we have that
    \begin{align}\label{feasible_a_0}
       \sup_{\mathbf{q}\in\mathcal{D}_{\phi}(\mathbf{p},r)}\rho_{u,h,\mathbf{q}}(f(\mathbf{a}_0,\mathbf{X}))<c-\epsilon \left(M+\frac{c+M}{p_{\min}}\right).
    \end{align}
Therefore, together with Assumption~\ref{assump: finite_strict_feasibility}, as well as the reformulation in Theorem~\ref{thm:ReformulationWholeProb}, we may apply the strong duality theorem and obtain the upper estimation
    \begin{align*}
        U_\epsilon&\leq \sup_{\lambda \geq 0}\min_{\mathbf{a}\in \mathcal{A}}g(\mathbf{a})+\lambda \left(\sup_{\mathbf{q}\in\mathcal{D}_{\phi}(\mathbf{p},r)}\rho_{u,h,\mathbf{q}}(f(\mathbf{a},\mathbf{X}))-c+\epsilon\left(M+\frac{c+M}{p_{\min}}\right)\right)\\
        &=\min_{\mathbf{a}\in \mathcal{A}}g(\mathbf{a})+\lambda^*(\epsilon) \left(\sup_{\mathbf{q}\in\mathcal{D}_{\phi}(\mathbf{p},r)}\rho_{u,h,\mathbf{q}}(f(\mathbf{a},\mathbf{X}))-c+\epsilon\left(M+\frac{c+M}{p_{\min}}\right)\right)\\
        &\leq \lambda^*(\epsilon)\cdot \epsilon\left(M+\frac{c+M}{p_{\min}}\right)+\sup_{\lambda\geq 0}\min_{\mathbf{a}\in \mathcal{A}}g(\mathbf{a})+\lambda \left(\sup_{\mathbf{q}\in\mathcal{D}_{\phi}(\mathbf{p},r)}\rho_{u,h,\mathbf{q}}(f(\mathbf{a},\mathbf{X}))-c\right)\\
        &=\lambda^*(\epsilon)\cdot \epsilon\left(M+\frac{c+M}{p_{\min}}\right)+P_0,
    \end{align*}
    where $\lambda^*(\epsilon)$ is the KKT-vector of the minimization problem on the right-hand side of \eqref{U_epsilon_bound}, corresponding to the supremum constraint, which depends on $\epsilon$. 
    As a final step, we will show that $\lambda^*(\epsilon)$ can be further bounded by a constant that does not depend on $\epsilon$, for $\epsilon$ sufficiently small. 
    Indeed, let $\mathbf{a}_0$ be the point in \eqref{feasible_a_0} and let $\epsilon\leq \frac{c-c_0}{2\left(M+\frac{c+M}{p_{\min}}\right)}$. 
    Since $P_0\leq U_\epsilon$, 
    \begin{align*}
        P_0&\leq \min_{\mathbf{a}\in \mathcal{A}}g(\mathbf{a})+\lambda^*(\epsilon) \left(\sup_{\mathbf{q}\in\mathcal{D}_{\phi}(\mathbf{p},r)}\rho_{u,h,\mathbf{q}}(f(\mathbf{a},\mathbf{X}))-c+\epsilon\left(M+\frac{c+M}{p_{\min}}\right)\right)\\
        &\leq g(\mathbf{a}_0)+\lambda^*(\epsilon)\left(\sup_{\mathbf{q}\in\mathcal{D}_{\phi}(\mathbf{p},r)}\rho_{u,h,\mathbf{q}}(f(\mathbf{a}_0,\mathbf{X}))-c+\epsilon\left(M+\frac{c+M}{p_{\min}}\right)\right)\\
        &\leq g(\mathbf{a}_0)-\lambda^*(\epsilon)\frac{c-c_0}{2},
    \end{align*}
    which implies that, for all $\epsilon\leq \frac{c-c_0}{2\left(M+\frac{c+M}{p_{\min}}\right)}$, we have
    \begin{align*}
        \lambda^*(\epsilon)\leq \frac{2(g(\mathbf{a}_0)-P_0)}{c-c_0}\triangleq C^*.
    \end{align*}
    Hence, for $\epsilon$ sufficiently small, 
    \begin{align*}
        0\leq U_\epsilon-P_0\leq C^*\cdot \epsilon\left(M+\frac{c+M}{p_{\min}}\right),
    \end{align*}
    which converges to zero as $\epsilon\to 0$.
\end{proof}

\begin{proof}[\textbf{Proof of Proposition~\ref{prop:impossiblility}}]
Choose the function $h(p)=p^2$. 
Then $h(0)=0$, $h(1)=1$ and $h(p)<p,~\forall p\in (0,1)$.
Due to the compactness of $\mathcal{U}$, there exists a maximizer $\mathbf{q}^*$ of $\sup_{\mathbf{q}\in \mathcal{U}}\rho_{h,\mathbf{q}}(X)$. 
Let $\hat{h}(p)=p$ be the concave envelope of $h$. 
By the assumptions on $\mathcal{U}$, we have that $\mathbf{q}^*_1<1$. 
Hence,
\begin{align*}
    \rho_{\hat{h},\mathbf{q}^*}(X)-\rho_{h,\mathbf{q}^*}(X)=\sum^m_{i=2}\left(\hat{h}\left(\sum^m_{k={i}}q^*_k\right)-h\left(\sum^m_{k={i}}q^*_k\right)\right)(x_{i-1}-x_{i})>0,
\end{align*}
since $\hat{h}(p)>h(p)$ for all $p\in (0,1)$ and $x_{i-1}-x_i> 0$ for all $i\in [m]$. 
Therefore, $\sup_{\mathbf{q}\in \mathcal{U}}\rho_{\hat{h},\mathbf{q}}(X)> \sup_{\mathbf{q}\in \mathcal{U}}\rho_{h,\mathbf{q}}(X)$.
\end{proof}

\begin{proof}[\textbf{Proof of Theorem~\ref{thm: inverse-s-shape}}]
It is sufficient to show that for any $X$, with outcomes $\{x_i\}^m_{i=1}$, we have
\begin{align*}
     \rho_{h,\mathbf{p}}(X)=\sup_{\mathbf{q}\in M^{\text{ca}}_h(\mathbf{p})}\sum^m_{i=1}-q_ix_i -\sup_{\bar{\mathbf{q}}\in N^{\text{cv}}_{\bar{h}}(\mathbf{p})}\sum^m_{i=1}\bar{q}_ix_i,
 \end{align*}
 since $\rho_{u,h,\mathbf{p}}(f(\mathbf{a},\mathbf{X}))=\rho_{h,\mathbf{p}}(u(f(\mathbf{a},\mathbf{X}))$.
    Let $-x_{(1)}\geq \ldots \geq -x_{(m)}$ be the ranked realizations of $X$. Let $k^*$ be the index where $h\left(\sum^{k^*}_{s=1}p_{(s)}\right)\leq h(p^0)\leq h\left(\sum^{k^*+1}_{s=1}p_{(s)}\right)$. 
    Then, 
    \begin{align*}
        &\rho_{h,\mathbf{p}}(X)&\\
        &=\sum^{k^*}_{i=1}\left(h\left(\sum^i_{s=1}p_{(s)}\right)-h\left(\sum^{i-1}_{s=1}p_{(s)}\right)\right)\cdot (-x_{(i)})+\left(h(p^0)-h\left(\sum^{k^*}_{s=1}p_{(s)}\right)\right)\cdot (-x_{(k^*+1)})\\
        &\qquad   + \left(h\left(\sum^{k^*+1}_{s=1}p_{(s)}\right)-h(p^0)\right)\cdot (-x_{(k^*+1)})+\sum^{m}_{i=k^*+2}\left(h\left(\sum^i_{s=1}p_{(s)}\right)-h\left(\sum^{i-1}_{s=1}p_{(s)}\right)\right)\cdot (-x_{(i)}),
    \end{align*}
    where an empty sum $\sum^0_{s=1}$ is zero by convention.
    We can also express the second line of the previous expression in terms of the dual function $\bar{h}(p)\triangleq 1-h(1-p)$, and rearrange the indices to obtain:
    \begin{align*}
        &\sum^{m}_{i=k^*+2}\left(h\left(\sum^i_{s=1}p_{(s)}\right)-h\left(\sum^{i-1}_{s=1}p_{(s)}\right)\right)\cdot (-x_{(i)})+ \left(h\left(\sum^{k^*+1}_{s=1}p_{(s)}\right)-h(p^0)\right)\cdot (-x_{(k^*+1)})\\
        &=\sum^{m}_{i=k^*+2}\left(\bar{h}\left(\sum^{m}_{s=i}p_{(s)}\right)-\bar{h}\left(\sum^{m}_{s=i+1}p_{(s)}\right)\right)\cdot (-x_{(i)})+ \left(\bar{h}(1-p^0)-\bar{h}\left(\sum^{m}_{s=k^*+2}p_{(s)}\right)\right)\cdot (-x_{(k^*+1)})\\
        &=-\sum^{m-k^*-1}_{i=1}\left(\bar{h}\left(\sum^i_{s=1}p_{(m-s+1)}\right)-\bar{h}\left(\sum^{i-1}_{s=1}p_{(m-s+1)}\right)\right)\cdot x_{(m-i+1)}\\
        &\qquad\qquad\qquad\qquad\qquad\qquad\qquad\qquad- \left(\bar{h}(1-p^0)-\bar{h}\left(\sum^{m-k^*-1}_{s=1}p_{(m-s+1)} \right)\right)\cdot x_{(k^*+1)},
    \end{align*}
where the empty sum $\sum^m_{s=m+1}$ is again zero. 
It remains to show that the sums are equal to their dual representations:
\begin{align}\label{dual_concave}
\begin{split}
   \sup_{\mathbf{q}\in M^{\text{ca}}_h(\mathbf{p})}\sum^m_{i=1}-q_ix_i&=\sum^{k^*}_{i=1}\left(h\left(\sum^i_{s=1}p_{(s)}\right)-h\left(\sum^{i-1}_{s=1}p_{(s)}\right)\right)\cdot (-x_{(i)})\\
   &\qquad\qquad\qquad+\left(h(p^0)-h\left(\sum^{k^*}_{s=1}p_{(s)}\right)\right)\cdot (-x_{(k^*+1)}), 
\end{split}
\end{align}
and
\begin{align}\label{dual_convex}
\begin{split}
    \sup_{\bar{\mathbf{q}}\in N^{\text{cv}}_{\bar{h}}(\mathbf{p})}\sum^m_{i=1}\bar{q}_ix_i
    &=\sum^{m-k^*-1}_{i=1}\left(\bar{h}\left(\sum^i_{s=1}p_{(m-s+1)}\right)-\bar{h}\left(\sum^{i-1}_{s=1}p_{(m-s+1)}\right)\right)\cdot x_{(m-i+1)}\\
    &\qquad\qquad\qquad+ \left(\bar{h}(1-p^0)-\bar{h}\left(\sum^{m-k^*-1}_{s=1}p_{(m-s+1)} \right)\right)\cdot x_{(k^*+1)}.
\end{split}   
\end{align}
We will first show \eqref{dual_concave}; \eqref{dual_convex} follows similarly. 
Let $h_0$ be defined as follows: 
\begin{align} \label{def: h_0}
    h_0(p)\triangleq\begin{cases}
        h(p)& p\leq p^0\\
        h(p^0)& p\geq p^0.
    \end{cases}
\end{align}
Then, $h_0$ is non-decreasing and concave. 
Let $P$ be the probability measure induced by the vector $(p_i)^m_{i=1}$. 
Then, by Lemma~\ref{lem: submodular}, $\mu_1\triangleq h_0\circ P$ is a monotone, submodular set function. 
By the definition of a Choquet integral, we have
\begin{align*}
    &\int -X \mathrm{d}\mu_1\\
    &=\sum^{k^*}_{i=1}\left(h_0\left(\sum^i_{s=1}p_{(s)}\right)-h_0\left(\sum^{i-1}_{s=1}p_{(s)}\right)\right)\cdot (-x_{(i)})+\left(h_0(p^0)-h_0\left(\sum^{k^*}_{s=1}p_{(s)}\right)\right)\cdot (-x_{(k^*+1)})\\
    &=\sum^{k^*}_{i=1}\left(h\left(\sum^i_{s=1}p_{(s)}\right)-h\left(\sum^{i-1}_{s=1}p_{(s)}\right)\right)\cdot (-x_{(i)})+\left(h(p^0)-h\left(\sum^{k^*}_{s=1}p_{(s)}\right)\right)\cdot (-x_{(k^*+1)}).
\end{align*}
Moreover, since $\mu_1$ is monotone and submodular, we have by Proposition 10.3 of \citet{Denneberg},
\begin{align*}
 \int -X \mathrm{d}\mu_1= \left\{\sum^m_{i=1}-q_ix_i~\middle |~ q_i\geq 0, \sum_{i\in J}q_{i}\leq h_0\left(\sum_{i\in J} p_i\right),~\forall J\subset [m],~\sum^m_{i=1}q_i=h_0(p^0)\right\},
\end{align*}
which involves the same feasible set as in  \eqref{uc_concave}.

Similarly, define the function
\begin{align}\label{def:h_bar0}
    \bar{h}_0(p)\triangleq\begin{cases}
        \bar{h}(p)& p\leq 1-p^0\\
        \bar{h}(1-p^0)& p\geq 1-p^0.
    \end{cases}
\end{align}
Then, $\bar{h}_0$ is concave and non-decreasing. 
Hence, by the same arguments as above, we have that $\mu_2\triangleq\bar{h}_0\circ P$ is monotone and submodular. 
Therefore,
\begin{align*}
    \int X \mathrm{d}\mu_2&=
    \sum^{m-k^*-1}_{i=1}\left(\bar{h}\left(\sum^i_{s=1}p_{(m-s+1)}\right)-\bar{h}\left(\sum^{i-1}_{s=1}p_{(m-s+1)}\right)\right)\cdot x_{(m-i+1)}\\
    &\qquad\qquad\qquad+ \left(\bar{h}(1-p^0)-\bar{h}\left(\sum^{m-k^*-1}_{s=1}p_{(s)} \right)\right)\cdot x_{(k^*+1)}.
\end{align*}
Moreover, we have
\begin{align*}
 \int X \mathrm{d}\mu_2= \sup\left\{\sum^m_{i=1}\bar{q}_ix_i~\middle |~ \bar{q}_i\geq 0, \sum_{i\in J}\bar{q}_{i}\leq \bar{h}_0\left(\sum_{i\in J} p_i\right),~\forall J\subset [m],~\sum^m_{i=1}\bar{q}_i=\bar{h}_0(1-p^0)\right\},
\end{align*}
which implies \eqref{uc_convex}.
\end{proof}

\begin{lemma}\label{lem: submodular}
    Let $\Omega$ be a set and let $\mathcal{F}=2^\Omega$ be the collection of all subsets of $\Omega$. 
    Furthermore, let $h:[0,1]\to \mathbb{R}$ be a non-decreasing concave function such that $h(0)=0$ and 
    let $P:\mathcal{F}\to [0,\infty)$ be an additive set function. 
    Then, $\mu\triangleq h\circ P$ is a monotone, submodular set function. 
\end{lemma}
\begin{proof}
Monotonicity of $\mu$ is clear due to the additivity of $P$ and the monotonicity of $h$. 
Let $A,B\in \mathcal{F}$. 
Additivity of $P$ also implies that
    \begin{align*}
        P(A\cup B)+P(A\cap B)=P(A)+P(B).
    \end{align*}
Let $a\triangleq P(A),\ b\triangleq P(B)$. 
Then, 
we have
    \begin{align*}
        P(A\cap B)\triangleq i\leq a\leq b\leq u\triangleq P(A\cup B).
    \end{align*}
Since $b-i=u-a$, 
we have that by concavity of $h$,
    \begin{align*}
        \frac{h(b)-h(i)}{b-i}\geq \frac{h(u)-h(a)}{u-a},
    \end{align*}
    which, via 
        $h(i)+h(u)\leq h(a)+h(b)$,
    implies the submodularity of $\mu$.
\end{proof}
\begin{proof}[\textbf{Proof of Theorem~\ref{thm: bilinear_nom}}]
By Theorem~\ref{thm: inverse-s-shape}, it is sufficient to show that the sets $M^{\text{ca}}_{h}(\mathbf{p})$ and $N^{\text{cv}}_{\bar{h}}(\mathbf{p})$ are the projection of the following sets on the coordinates $\mathbf{q}$ and $\bar{\mathbf{q}}$, respectively:
\begin{align}\label{uc_concave_piecewise}
M^{\text{ca}}_{h,l}(\mathbf{p})=\left\{\mathbf{q}\in \mathbb{R}^m, \mathbf{t}^{(1)}\in \mathbb{R}^{m\times K_1} \;
    \begin{tabular}{|l}
      $q_i\geq 0$\\
      $q_i\leq l^{(1)}_kp_i+t^{(1)}_{ik},~\forall i\in [m],~ \forall k\in [K_1] $\\
      $\sum^{m}_{i=1}t^{(1)}_{ik}\leq b^{(1)}_k,~\forall k\in [K_1] $\\
      $\sum^m_{i=1}q_i=h(p^0)$
\end{tabular}
\right\},
\end{align}
and 
\begin{align}\label{uc_convex_piecewise}
N^{\text{cv}}_{\bar{h},l}(\mathbf{p})=\left\{\bar{\mathbf{q}}\in \mathbb{R}^m, \mathbf{t}^{(2)}\in \mathbb{R}^{m\times K_2} \;
    \begin{tabular}{|l}
      $\bar{q}_i\geq 0$\\
      $\bar{q}_i\leq l^{(2)}_kp_i+t^{(2)}_{ik},~\forall i\in [m],~\forall k\in [K_2] $\\
      $\sum^{m}_{i=1}t^{(2)}_{ik}\leq b^{(2)}_k, \forall k\in [K_2] $\\
      $\sum^m_{i=1}\bar{q}_i=\bar{h}(1-p^0)$
\end{tabular}
\right\},
\end{align}
We will only show this for $M^{\text{ca}}_{h,l}(\mathbf{p})$, since the case for $N^{\text{cv}}_{\bar{h},l}(\mathbf{p})$ is identical, \textit{mutatis mutandis}. 

Let $(\mathbf{q}, \mathbf{t}^{(1)})\in M^{\text{ca}}_{h,l}(\mathbf{p})$. 
Then, $q_i\leq l^{(1)}_k\cdot p_i+t^{(1)}_{ik}$ for all $i\in [m]$ and $k\in [K_1]$. 
Thus, for any subset $I\subset [m]$, we also have that, for all $j\in [K_1]$, 
\begin{align*}
    \sum_{i\in I}q_i\leq l^{(1)}_k\sum_{i\in I}p_i+\sum_{i\in I} t^{(1)}_{ik}\leq l^{(1)}_k\sum_{i\in I}p_i+\sum^m_{i=1} t^{(1)}_{ik}\leq l^{(1)}_k\sum_{i\in I}p_i+b^{(1)}_k.  
\end{align*}
Hence, $\mathbf{q}\in M^{\text{ca}}_h(\mathbf{p})$. 
Conversely, let $\mathbf{q}\in M^{\text{ca}}_h(\mathbf{p})$. 
Then, $q_i\leq l^{(1)}_k\cdot p_i+t^{(1)}_{ik}$ for all $i\in [m]$ and $k\in [K_1]$ with $t^{(1)}_{ik}\triangleq\max\{q_i-l^{(1)}_k p_i,0\}$. 
Moreover, we have, for all $k\in [K_1]$, 
\begin{align*}
    \sum^m_{i=1}t^{(1)}_{ik}=\sum^m_{i=1}\max\{q_i-l^{(1)}_k p_i,0\}=\sum_{i\in I_+}q_i-\sum_{i\in I_+}l_k^{(1)}p_i\leq b^{(1)}_k,~\text{where}~I_+\triangleq\{i: q_i-l^{(1)}_k p_i\geq 0\}.
\end{align*}
Indeed, the last inequality follows from the fact that, for any $\mathbf{q}\in M^{\text{ca}}_h(\mathbf{p})$, we have that $\sum_{i\in I}\bar{q}_i\leq l^{(1)}_k\sum_{i\in I}q_i+b^{(1)}_k$ if $\sum_{i\in I}p_i\leq h(p^0)$. 
For any index set $I$ such that $\sum_{i\in I}p_i>h(p^0)$, we also have that $l_k^{(1)}\sum_{i\in I}p_i+b^{(1)}_k> l_k^{(1)}h(p^0)+b^{(1)}_k\stackrel{(*)}{\geq} l_{K_1}^{(1)}h(p^0)+b^{(1)}_{K_1}= h(p^0)\geq \sum_{i\in I}q_i$, where $(*)$ follows from our ordering of the slopes and the intercepts $(l^{(1)}_k,b^{(1)}_k)^{K_1}_{k=1}$ as described in \eqref{pw_concave_convex}. 
Hence, the inequality under consideration indeed holds for the index set $I_+$. 
Thus, $(\mathbf{q}, \mathbf{t}^{(1)})\in M^{\text{ca}}_{h,l}(\mathbf{p})$.
The reformulation in \eqref{RC: invS} now follows as in Theorem~\ref{thm:RCpiecewiselin}.
\end{proof}

\begin{proof}[\textbf{Proof of Theorem~\ref{thm: cut_end_invS}}]
We first note that, by Theorem~\ref{thm: inverse-s-shape}, 
 \begin{align*}
     \rho_{u,h,\mathbf{q}}(f(\mathbf{a},\mathbf{X}))=\sup_{\bar{\mathbf{q}}\in M^{\text{ca}}_h(\mathbf{q})}\sum^m_{i=1}-\bar{q}_iu(f(\mathbf{a},\mathbf{x}_i)) -\sup_{\tilde{\mathbf{q}}\in N^{\text{cv}}_{\bar{h}}(\mathbf{q})}\sum^m_{i=1}\tilde{q}_iu(f(\mathbf{a},\mathbf{x}_i)).
 \end{align*}
Since $\mathcal{U}_{t}\subset \mathcal{D}_{\phi}(\mathbf{p},r)$, we have that for each iteration $t$, the cutting-plane procedure yields a lower bound $c^t$ on the true optimal objective value \eqref{P}, which we denote by $\tilde{c}$. 
By Assumption~\ref{assumption:well-defined}, $\tilde{c}$ is finite. 
Hence, we may assume that $|c^s-c^t|\leq \frac{1}{2}\epsilon_{\mathrm{tol}}$, for all $s,t$ sufficiently large. 
Otherwise, the lower bounds will attain or exceed $\tilde{c}$ after finitely many iterations. 
Assume now that $s>t$. 
Let $\mathbf{q}^t$ be the worst-case probability vector added to $\mathcal{U}_t$ and let $\mathbf{a}_t,\mathbf{a}_s$ be the optimal solutions obtained at the $t,s$-th iteration. 
By definition, 
    \begin{align}
        \sup_{\bar{\mathbf{q}}\in M^{\text{ca}}_h(\mathbf{q}^t)}\sum^m_{i=1}-\bar{q}_iu(f(\mathbf{a}^t,\mathbf{x}_i)) -\sup_{\tilde{\mathbf{q}}\in N^{\text{cv}}_{\bar{h}}(\mathbf{q}^t)}\sum^m_{i=1}\tilde{q}_iu(f(\mathbf{a}^t,\mathbf{x}_i))>c^t+\mathrm{\epsilon}_{\mathrm{tol}},
   \end{align}
and  
\begin{align}
        \sup_{\bar{\mathbf{q}}\in M^{\text{ca}}_h(\mathbf{q}^t)}\sum^m_{i=1}-\bar{q}_iu(f(\mathbf{a}^s,\mathbf{x}_i)) -\sup_{\tilde{\mathbf{q}}\in N^{\text{cv}}_{\bar{h}}(\mathbf{q}^t)}\sum^m_{i=1}\tilde{q}_iu(f(\mathbf{a}^s,\mathbf{x}_i))\leq c^s.
    \end{align}
Therefore, 
\begin{align*}
    &\sup_{\bar{\mathbf{q}}\in M^{\text{ca}}_h(\mathbf{q}^t)}\sum^m_{i=1}-\bar{q}_iu(f(\mathbf{a}^t,\mathbf{x}_i))-\sup_{\bar{\mathbf{q}}\in M^{\text{ca}}_h(\mathbf{q}^t)}\sum^m_{i=1}-\bar{q}_iu(f(\mathbf{a}^s,\mathbf{x}_i))\\
    &\qquad\qquad\qquad+\sup_{\tilde{\mathbf{q}}\in N^{\text{cv}}_{\bar{h}}(\mathbf{q}^t)}\sum^m_{i=1}\tilde{q}_iu(f(\mathbf{a}^s,\mathbf{x}_i))-\sup_{\tilde{\mathbf{q}}\in N^{\text{cv}}_{\bar{h}}(\mathbf{q}^t)}\sum^m_{i=1}\tilde{q}_iu(f(\mathbf{a}^t,\mathbf{x}_i))\\
    &> \epsilon_{\mathrm{tol}}-(c^s-c^t)\geq \frac{1}{2}\epsilon_{\mathrm{tol}}.
\end{align*}
We can further upper bound the differences of the suprema as follows:
\begin{align*}
    &\sup_{\bar{\mathbf{q}}\in M^{\text{ca}}_h(\mathbf{q}^t)}\sum^m_{i=1}-\bar{q}_iu(f(\mathbf{a}^t,\mathbf{x}_i))-\sup_{\bar{\mathbf{q}}\in M^{\text{ca}}_h(\mathbf{q}^t)}\sum^m_{i=1}-\bar{q}_iu(f(\mathbf{a}^s,\mathbf{x}_i))\\
    &\leq \sup_{\bar{\mathbf{q}}\in M^{\text{ca}}_h(\mathbf{q}^t)}\left|\sum^m_{i=1}-\bar{q}_i(u(f(\mathbf{a}^s,\mathbf{x}_i))-u(f(\mathbf{a}^t,\mathbf{x}_i)))\right|\\
    &\leq \sup_{\bar{\mathbf{q}}\in M^{\text{ca}}_h(\mathbf{q}^t)}\|\mathbf{\bar{q}}\|_2\|(u(f(\mathbf{a}^{s},\mathbf{x}_i))-u(f(\mathbf{a}^{t},\mathbf{x}_i)))^m_{i=1}\|_2\\
    &\leq \|(u(f(\mathbf{a}^{s},\mathbf{x}_i))-u(f(\mathbf{a}^{t},\mathbf{x}_i)))^m_{i=1}\|_2. 
\end{align*}
Similarly, 
\begin{align*}
    &\sup_{\tilde{\mathbf{q}}\in N^{\text{cv}}_{\bar{h}}(\mathbf{q}^t)}\sum^m_{i=1}\tilde{q}_iu(f(\mathbf{a}^s,\mathbf{x}_i))-\sup_{\tilde{\mathbf{q}}\in N^{\text{cv}}_{\bar{h}}(\mathbf{q}^t)}\sum^m_{i=1}\tilde{q}_iu(f(\mathbf{a}^t,\mathbf{x}_i))\leq \|(u(f(\mathbf{a}^{s},\mathbf{x}_i))-u(f(\mathbf{a}^{t},\mathbf{x}_i)))^m_{i=1}\|_2.
\end{align*}
Therefore, 
\begin{align*}
    \|(u(f(\mathbf{a}^{s},\mathbf{x}_i))-u(f(\mathbf{a}^{t},\mathbf{x}_i)))^m_{i=1}\|_2\geq \frac{1}{4}\epsilon_{\mathrm{tol}}.
\end{align*}
However, we have the assumption that $\sup_{\mathbf{a}\in \mathcal{A}, i\in [m]}|u(f(\mathbf{a},\mathbf{x}_i))|<\infty$. 
Hence, this leads to a similar contradiction as in the proof of Theorem~\ref{thm:cuttingplane_conv}. 
Therefore, the cutting-plane procedure must terminate.
\end{proof}

\begin{proof}[\textbf{Proof of Theorem~\ref{thm: invS_PL_error}}]
Let $M\triangleq\sup_{\mathbf{a}\in \mathcal{A}, i\in [m]}|u(f(\mathbf{a},\mathbf{x}_i))|<\infty$. 
Similar to the proof of Theorem~\ref{thm:convergencePL}, for any two distortion functions $h_1, h_2$ such that $\sup_{p\in [0,1]}|h_1(p)-h_2(p)|\leq \epsilon$, we have that
\begin{align*}
    \rho_{u,h_1,\mathbf{p}}(f(\mathbf{a},\mathbf{X}))&\leq \rho_{u,h_2,\mathbf{p}}(f(\mathbf{a},\mathbf{X}))+\epsilon\cdot \left(\max_{\substack{i\in [m]}}u(f(\mathbf{a},\mathbf{x}_i))-\min_{i\in [m]} u(f(\mathbf{a},\mathbf{x}_i))\right)\\
    &\leq \rho_{u,h_2,\mathbf{p}}(f(\mathbf{a},\mathbf{X}))+2M\epsilon.
\end{align*}
Since $2M\epsilon$ does not depend on both $\mathbf{a}$ and $\mathbf{q}$, we have that, in the nominal case,
\begin{align*}
    \left|\min_{\mathbf{a}\in \mathcal{A}} \rho_{u,h_1,\mathbf{p}}(f(\mathbf{a},\mathbf{X}))- \min_{\mathbf{a}\in \mathcal{A}}\rho_{u,h_2,\mathbf{p}}(f(\mathbf{a},\mathbf{X}))\right|\leq 2M\epsilon,
\end{align*}
and similarly, in the robust case,
\begin{align*}
    \left|\min_{\mathbf{a}\in \mathcal{A}} \sup_{\mathbf{q}\in \mathcal{D}_\phi(\mathbf{p},r)}\rho_{u,h_1,\mathbf{q}}(f(\mathbf{a},\mathbf{X}))- \min_{\mathbf{a}\in \mathcal{A}}\sup_{\mathbf{q}\in \mathcal{D}_\phi(\mathbf{p},r)}\rho_{u,h_2,\mathbf{q}}(f(\mathbf{a},\mathbf{X}))\right|\leq 2M\epsilon,
\end{align*}
which both approach zero as $\epsilon\to 0$. 
Since Algorithm~\ref{algo: cutting_plane_inv_S} yields a final objective value $c^*$ such that 
\begin{align*}
    \left| c^*-\min_{\mathbf{a}\in \mathcal{A}}\sup_{\mathbf{q}\in \mathcal{D}_\phi(\mathbf{p},r)}\rho_{u,h_\epsilon,\mathbf{q}}(f(\mathbf{a},\mathbf{X}))\right|\leq \epsilon_{\mathrm{tol}},
\end{align*}
it follows that
\begin{align*}
    \left| c^*-\min_{\mathbf{a}\in \mathcal{A}}\sup_{\mathbf{q}\in \mathcal{D}_\phi(\mathbf{p},r)}\rho_{u,h,\mathbf{q}}(f(\mathbf{a},\mathbf{X}))\right|\leq \epsilon_{\mathrm{tol}} + \epsilon 2M\to 0,
\end{align*}
as $\epsilon_{\mathrm{tol}},\epsilon\to 0$.
\end{proof}

\begin{proof}[\textbf{Proof of Theorem~\ref{thm: inv_S_constraint_cut}}]
The proof that Algorithm~\ref{algo: cutting_plane_inv_S} terminates after finitely many steps for problem \eqref{P-constraint} is similar to the proof of Theorem~\ref{thm: cut_end_invS}, where we now take $c^s=c^t=c$, with $c$  the constraint value in \eqref{P-constraint}. 
To show convergence, we note that $\mathcal{U}_j\subset \mathcal{D}_{\phi}(\mathbf{p},r)$. 
Hence, $g(\mathbf{a}_{\epsilon_{\mathrm{tol}}})$ is always a lower bound of \eqref{P-constraint}, for which we denote its optimal objective value by $P(0)$. 
We define 
    \begin{align*}
        P(\epsilon_{\mathrm{tol}})\triangleq\min_{\mathbf{a}\in \mathcal{A}}\ \left\{g(\mathbf{a})~\middle |~\sup_{\mathbf{q}\in\mathcal{D}_{\phi}(\mathbf{p},r)}\rho_{u,h,\mathbf{q}}(f(\mathbf{a},\mathbf{X}))\leq c+\epsilon_{\mathrm{tol}} \right\}.
    \end{align*}
Then, by Algorithm~\ref{algo: cutting_plane_inv_S}, the final solution $\mathbf{a}_{\epsilon_{\mathrm{tol}}}$ is feasible for the minimization problem of $P(\epsilon_{\mathrm{tol}})$. 
Hence, $P(\epsilon_{\mathrm{tol}})\leq g(\mathbf{a}_{\epsilon_{\mathrm{tol}}})\leq P(0)$. 
It remains to show that $\lim_{\epsilon_{\mathrm{tol}}\to 0}P(\epsilon_{\mathrm{tol}})=P(0)$, which follows from the proof of Theorem~\ref{thm: inv_S_constraint} below.
\end{proof}

\begin{proof}[\textbf{Proof of Theorem~\ref{thm: inv_S_constraint}}]
     Let
     \begin{align*}
        L_\epsilon\triangleq \min_{\mathbf{a}\in \mathcal{A}}\ \left\{g(\mathbf{a})~\middle |~\sup_{\mathbf{q}\in\mathcal{D}_{\phi}(\mathbf{p},r)}\rho_{u,h_\epsilon,\mathbf{q}}(f(\mathbf{a},\mathbf{X}))\leq c \right\}.
    \end{align*}
    Then, by the proof of Theorem~\ref{thm: invS_PL_error},  for all $\mathbf{a}\in \mathcal{A}$, 
    \begin{align*}
    \left|\sup_{\mathbf{q}\in \mathcal{D}_\phi(\mathbf{p},r)}\rho_{u,h_\epsilon,\mathbf{q}}(f(\mathbf{a},\mathbf{X}))- \sup_{\mathbf{q}\in \mathcal{D}_\phi(\mathbf{p},r)}\rho_{u,h,\mathbf{q}}(f(\mathbf{a},\mathbf{X}))\right|\leq \epsilon M,
\end{align*}
where $M\triangleq 2\sup_{\mathbf{a}\in \mathcal{A}, i\in [m]}|u(f(\mathbf{a},\mathbf{x}_i))|<\infty$, which follows from the continuity of $u(f(\mathbf{a},\mathbf{x}_i))$ in $\mathbf{a}$ and the compactness of $\mathcal{A}$. 
Hence, we conclude that $P(\epsilon)\leq L_\epsilon\leq P(0)$, where $P(0)$ is the optimal objective value of \eqref{P-constraint}, and
\begin{align*}
    P(\epsilon)\triangleq\min_{\mathbf{a}\in \mathcal{A}}\ \left\{g(\mathbf{a})~\middle |~\sup_{\mathbf{q}\in\mathcal{D}_{\phi}(\mathbf{p},r)}\rho_{u,h,\mathbf{q}}(f(\mathbf{a},\mathbf{X}))\leq c+\epsilon M \right\}.
\end{align*}
We will now show that $\lim_{\epsilon\downarrow 0}P(\epsilon)=P(0)$, hence concluding the proof. 
To show this, we use Berge's maximum theorem (\citealp{Berge}), for which we have to show that the set-valued function
\begin{align*}
    \epsilon\mapsto G(\epsilon)\triangleq\left\{\mathbf{a}\in \mathcal{A}: \sup_{\mathbf{q}\in\mathcal{D}_{\phi}(\mathbf{p},r)}\rho_{u,h,\mathbf{q}}(f(\mathbf{a},\mathbf{X}))\leq c+\epsilon M \right\}
\end{align*}
is a compact-valued, continuous correspondence at $\epsilon=0$. 
We first examine compactness, which entails that for each $\epsilon\geq 0$, the set $G(\epsilon)$ must be compact. We note that since $\mathcal{A}$ is compact, we have that $G(\epsilon)$ is bounded for all $\epsilon\geq 0$. Hence, we only need to show that $G(\epsilon)$ is closed. 
This holds if $\mathbf{a}\mapsto \sup_{\mathbf{q}\in\mathcal{D}_{\phi}(\mathbf{p},r)}\rho_{u,h,\mathbf{q}}(f(\mathbf{a},\mathbf{X}))$ is continuous, which can be proven as follows: by Lemma~\ref{lem: continuity_rank_dependent}, we have that $(\mathbf{q},\mathbf{a})\mapsto \rho_{u,h,\mathbf{q}}(f(\mathbf{a},\mathbf{X}))$ is jointly continuous in $(\mathbf{q},\mathbf{a})$. 
Since the set $\mathcal{D}_{\phi}(\mathbf{p},r)$ is compact and independent of $\mathbf{a}$, Berge's maximum theorem implies that $\mathbf{a}\mapsto \sup_{\mathbf{q}\in\mathcal{D}_{\phi}(\mathbf{p},r)}\rho_{u,h,\mathbf{q}}(f(\mathbf{a},\mathbf{X}))$ is continuous.
Hence, $G(\epsilon)$ is compact.

We now show that $G(\epsilon)$ is lower-and upper-hemicontinuous at $\epsilon = 0$. 
For upper-hemicontinuity, we must show that $G(0)$ is non-empty, which holds due to Assumption~\ref{assump: finite_strict_feasibility}, and that, for any $\epsilon_j\to 0$ and any sequence $\mathbf{a}_j\in G(\epsilon_j)$, there exists a convergent subsequence $\mathbf{a}_{j_k}$ such that $\mathbf{a}_{j_k}\to \mathbf{a}_0\in G(0)$. 
Note that due to the compactness of $\mathcal{A}$, there is indeed a subsequence $\mathbf{a}_{j_k}\to \mathbf{a}_0\in \mathcal{A}$. 
It remains to show that we have $\mathbf{a}_0\in G(0)$. 
This follows from continuity: $\sup_{\mathbf{q}\in\mathcal{D}_{\phi}(\mathbf{p},r)}\rho_{u,h,\mathbf{q}}(f(\mathbf{a}_{0},\mathbf{X}))=\lim_{k\to \infty}\sup_{\mathbf{q}\in\mathcal{D}_{\phi}(\mathbf{p},r)}\rho_{u,h,\mathbf{q}}(f(\mathbf{a}_{j_k},\mathbf{X}))\leq \lim_{k\to \infty}c+\epsilon_{j_k}M=c$.

Finally, for lower-hemicontinuity at $\epsilon=0$, we have to show that, for every $\mathbf{a}_0\in G(0)$ and every sequence $\epsilon_j\to 0$, there exist a $J\geq 1$ and a sequence $\mathbf{a}_j$ such that, for all $j\geq J$, we have $\mathbf{a}_j\in G(\epsilon_j)$ and $\mathbf{a}_j\to \mathbf{a}_0$. 
However, since $\epsilon_j>0$ for all $j$, we can take the sequence $\mathbf{a}_j=\mathbf{a}_0\in G(\epsilon_j)$.
Hence, we may apply Berge's maximum theorem to conclude that $\lim_{\epsilon\downarrow 0}P(\epsilon)=P(0)$, which concludes the proof.
\end{proof}

\begin{lemma}\label{lem: continuity_rank_dependent}
    Let $u(f(\mathbf{a},\mathbf{x}_i))$ be continuous in $\mathbf{a}$ for all $i\in [m]$ and let $h$ be a continuous distortion function. 
    Then, the rank-dependent evaluation function
    \begin{align*}
        (\mathbf{q},\mathbf{a})\mapsto \rho_{u,h,\mathbf{q}}(f(\mathbf{a},\mathbf{X})),
    \end{align*}
    is jointly continuous in $(\mathbf{q},\mathbf{a})$. 
\end{lemma}

\begin{proof}
For any $\mathbf{a}\in\mathcal{A}$, we denote the indices $1(\mathbf{a}),\ldots, m(\mathbf{a})$ such that
\begin{align*}
    -u(f(\mathbf{a},\mathbf{x}_{1(\mathbf{a})}))\leq \cdots \leq -u(f(\mathbf{a},\mathbf{x}_{m(\mathbf{a})})),
\end{align*}
and we define
\begin{align*}
    \Delta u(f(\mathbf{a},\mathbf{x}_{1(\mathbf{a})}))\triangleq-u(f(\mathbf{a},\mathbf{x}_{1(\mathbf{a})})),~\Delta u(f(\mathbf{a},\mathbf{x}_{i(\mathbf{a})}))\triangleq u(f(\mathbf{a},\mathbf{x}_{(i-1)(\mathbf{a})}))-u(f(\mathbf{a},\mathbf{x}_{i(\mathbf{a})})),
\end{align*}
for $i=2,\ldots,m$. 
Denote the index set 
\begin{align*}
    \mathcal{I}_+(\mathbf{a}_0)=\{i\in \{2,\ldots ,m\}: \Delta u(f(\mathbf{a}_0,\mathbf{x}_{i(\mathbf{a}_0)}))>0\}.
\end{align*}
We enumerate the elements in $\mathcal{I}_+(\mathbf{a}_0)$ as $i_1<\cdots < i_K$, for $K=|\mathcal{I}_+(\mathbf{a}_0)|\leq m-1$. 
Set $i_0=1$ and $i_{K+1}=m+1$. 
Then, the index set $[m]$ is partitioned into $K+1$ disjoint classes $\Pi_j\triangleq\{i(\mathbf{a}_0): i_{j-1}\leq i<i_j\}$, with $j\in [K+1]$, that have the following two properties: 
(i) If $k,l\in [m]$ are indices of adjacent classes, e.g., $k\in \Pi_j$ and $l\in \Pi_{j+1}$ for some $j$, then $u(f(\mathbf{a}_0,\mathbf{x}_k))-u(f(\mathbf{a}_0,\mathbf{x}_l))=\Delta u(f(\mathbf{a}_0,\mathbf{x}_{i_{j+1}(\mathbf{a}_0)}))>0$. 
(ii) If $k,l$ belong to the same class $\Pi_j$, we have $u(f(\mathbf{a}_0,\mathbf{x}_k))-u(f(\mathbf{a}_0,\mathbf{x}_l))=0$. 
Continuity of $u(f(\mathbf{a},\mathbf{x}))$ in $\mathbf{a}_0$ implies that for all $\mathbf{a}$ sufficiently close to $\mathbf{a}_0$, the ranking indices $1(\mathbf{a}),\ldots, m(\mathbf{a})$ can also be partitioned such that $\{i(\mathbf{a}):i_{j-1}\leq i<i_j\}=\{i(\mathbf{a}_0):i_{j-1}\leq i<i_j\}$, for all $j\in [K+1]$. 
In particular, this means that within each class, the ranking indices of $\mathbf{a}$ constitute a permutation of that of $\mathbf{a}_0$. 

Hence, we have
\begin{align*}
    &\rho_{u,h,\mathbf{q}}(f(\mathbf{a},\mathbf{X}))\\
    \qquad&= \sum^m_{i=1}h\left(\sum^m_{k=i}q_{k(\mathbf{a})}\right)\Delta u(f(\mathbf{a},\mathbf{x}_{i(\mathbf{a})}))=\sum^{K+1}_{j=1}\sum^{i_j-1}_{i=i_{j-1}}h\left(\sum^m_{k=i}q_{k(\mathbf{a})}\right)\Delta u(f(\mathbf{a},\mathbf{x}_{i(\mathbf{a})}))\\
    &=\sum^{K+1}_{j=1}h\left(\sum^m_{k=i_{j-1}}q_{k(\mathbf{a})}\right)\Delta u(f(\mathbf{a},\mathbf{x}_{i_{j-1}(\mathbf{a})}))+\sum^{K+1}_{j=1}\sum^{i_j-1}_{i=i_{j-1}+1}h\left(\sum^m_{k=i}q_{k(\mathbf{a})}\right)\Delta u(f(\mathbf{a},\mathbf{x}_{i(\mathbf{a})}))\\
    &\stackrel{(\mathbf{q},\mathbf{a})\to (\mathbf{q}_0,\mathbf{a}_0)}{=}\sum^{K+1}_{j=1}h\left(\sum^m_{k=i_{j-1}}q_{0,k(\mathbf{a}_0)}\right)\Delta u(f(\mathbf{a}_0,\mathbf{x}_{i_{j-1}(\mathbf{a}_0)}))+0\\
    &=\rho_{u,h,\mathbf{q}_0}(f(\mathbf{a}_0,\mathbf{X})).
\end{align*}
This completes the proof.
\end{proof}

\section{Derivations of the Conjugate Functions and Their Epigraphs}\label{App:RCderivation} 
In this appendix, we provide detailed derivations of the explicit conjugate functions, as well as the epigraphs of $\lambda(-h)^*(\frac{-\nu}{\lambda})\leq z, \lambda>0$, and $\gamma \phi^*(\frac{s}{\gamma})\leq t, \gamma>0$, for the collection of canonical examples in Tables~\ref{tab:h_conjugates}--\ref{tab:phi_conjugate};
we only provide the details in case the derivations are considered to be non-trivial.

\subsection{Distortion Functions}
\begin{itemize}
\item For $h(p)=\begin{cases}
    (1+r)p&p<1/2\\
    (1-r)p+r&p\geq 1/2
\end{cases}$ and $0<r<1$, we examine the conjugate
\begin{align*}
(-h)^*(y)
&=\max\{\sup_{t\in [0,\frac{1}{2})}(y+(1+r))t,\sup_{t\geq \frac{1}{2}}(y+(1-r))t+r\}.
\end{align*}
We have $(-h)^*(y)=\infty$ for $y>-(1-r)$. 
For $-(1+r)\leq y\leq -(1-r)$, we have $(-h)^*(y)=\max\{1/2(y+(1+r)),1/2(y+(1-r))+r\}=1/2(y+(1+r))$. For $y<-(1+r)$, $(-h)^*(y)=\max\{0,1/2(y+(1-r))+r\}=0$. 
Therefore, $(-h)^*(y)=\max\{(y+1+r)/2,0\},~\text{for}~y\leq -(1-r)$.
\item Let $h(p)=(1+r)p-rp^2$, $0<r<1$. 
For $y\leq 0$, we have
\begin{align*}
    (-h)^*(y)=\sup_{t\in [0,1]}\{ty+(1+r)t-rt^2\}=\sup_{t\in [0,1]}\{(y+(1+r))t-rt^2\}.
\end{align*}
Differentiating the objective w.r.t.\ $t$ yields
\begin{align*}
    y+(1+r)-2rt,
\end{align*}
which is non-negative for $t\leq \frac{y+1+r}{2r}$ and negative otherwise. 
If $\frac{y+1+r}{2r}<0$, then the derivative is always negative, hence the maximum is obtained at $t=0$ and thus we have $(-h)^*(y)=0$. 
If $\frac{y+1+r}{2r}>1$, then the derivative is always positive, hence $(-h)^*(y)=y+1$. 
If $\frac{y+1+r}{2r}\in [0,1]$, then the maximum is attained at $t=\frac{y+1+r}{2r}$. 
Hence, for $y\leq 0$, we have
\begin{align*}
    (-h)^*(y)=\begin{cases}
        \frac{1}{4r}\max\{y+1+r,0\}^2 & y+1\leq r\\
        y+1 & y+1>r.
     \end{cases}
\end{align*}

The epigraph of $(-h)^*$ can be represented by
\begin{align*}
    \begin{cases}
        t = t_1+t_2+r\\
        y+1-r= y_1+y_2\\
        t_1\geq \frac{1}{4r}\max\{y_1+2r,0\}^2-r,~t_2\geq y_2\\
        y_1\leq 0, y_2\geq 0.
    \end{cases}
\end{align*}
Indeed, let $(y,t)\in \mathrm{Epi}((-h)^*)$. 
If $y+1-r\leq 0$, then we can choose $y_1=y+1-r$, $t_1=t-r$, $y_2=t_2=0$. If $y+1-r\geq 0$, then choose $y_1=t_1=0$, $y_2=y+1-r$, $t_2=t-r$.

Conversely, let $(y,t,y_1,y_2,t_1,t_2)$ satisfy the above constraints. 
Define 
\begin{align*}
f(z)=\begin{cases}
     \frac{1}{4r}\max\{z+2r,0\}^2-r & z\leq 0\\
        z & z\geq 0.
\end{cases}    
\end{align*}
Then, $(-h)^*(y)=f(y+1-r)+r$.
Hence, $(-h)^*(y)\leq t\Leftrightarrow f(y+1-r)\leq t-r$.
Since $y_1\leq 0$, $y_2\geq 0$, we have that $y_1\leq y_1+y_2\leq y_2$. 
Since $f$ is convex and $f(0)=0$, we have
\begin{align*}
    \frac{f(y_1+y_2)-f(y_1)}{y_2}\leq \frac{f(y_2)-f(0)}{y_2},
\end{align*}
and thus $f(y_1+y_2)\leq f(y_1)+f(y_2)$. 
Therefore, $(-h)^*(y)=f(y+1-r)+r=f(y_1+y_2)\leq f(y_1)+f(y_2)+r\leq t_1+t_2+r=t$.

The epigraph of $\lambda (-h)^*(\frac{-\nu}{\lambda})\leq z $ is then given by
\begin{align*}
    \begin{cases}
        z = z_1+z_2\\
        -\nu+\lambda(1-r)= \xi_1+\xi_2\\
        (z_1+\lambda)\geq \sqrt{\frac{w^2}{r}+(z_1-\lambda)^2},~z_2\geq \xi_2\\
        \xi_1+2r\leq w\\
        \xi_1\leq 0, \xi_2, w, \nu \geq 0.
    \end{cases}
\end{align*}

\item For $h(p)=\begin{cases}
    1-(1-p)^n&0\leq p< 1\\
    1&p\geq 1
\end{cases}$, we will derive a tractable reformulation of the epigraph $(-h)^*$ using duality. 
We have
\begin{align*}
 \mathrm{Epi}(-h)= \{(p,t)\in \mathbb{R}_{\geq 0}\times \mathbb{R}:\exists~ (u_1,u_2)\in \mathbb{R}^2_{\geq 0}:u_1-1\leq t, u_2\leq u_1^{1/n}, 1-p\leq u_2\}.
\end{align*}
Indeed, let $(p,t)\in \mathrm{Epi}(-h)$. 
If $0\leq p<1$, then $(1-p)^n-1\leq t$. 
We can choose $u_2=1-p\geq 0$ and $u_1=u_2^{n}\geq 0$. 
If $p\geq 1$, then $-1\leq t$. 
We can choose $u_1=u_2=0$. 
Conversely, let $(p,t,u_1,u_2)$ satisfy the above constraints. 
If $p\geq 1$, we have $t\geq -1+u_1\geq -1$. 
Thus, $(p,t)\in \mathrm{Epi}(-h)$. 
If $0\leq p<1$, then $(1-p)^n-1\leq u_2^n-1\leq u_1-1\leq t$. 
Hence, we also have $(p,t)\in \mathrm{Epi}(-h)$.

Consider the epigraph of $(-h)^*$. 
We have
\begin{align*}
    \mathrm{Epi}((-h)^*)=\{(y,s): yp-(-h)(p)\leq s, \forall p\geq 0\}&=\{(y,s): yp-t\leq s, \forall (p,t)\in \mathrm{Epi}(-h)\}.
\end{align*}
Therefore, we have that $(y,s)\in \mathrm{Epi}((-h)^*)$ if and only if the optimization problem
\begin{align*}
    \min_{p,u_1,u_2\geq 0, t\in \mathbb{R}}\{-yp+t|u_1-1\leq t, u_2\leq u_1^{1/n}, 1-p\leq u_2\}
\end{align*}
is bounded below by $-s$. 
Since this is a convex problem with a point $p=1,u_2=0,u_1=1,t=0$ that satisfies Slater's condition, we may apply the duality theorem and obtain that this optimization problem is equal to
\begin{align*}
    \max_{\xi_1,\xi_2,\xi_3\geq 0}\inf_{p,u_1,u_2\geq 0, t\in \mathbb{R}}-yp+t+\xi_1(u_1-1-t)+\xi_2(u_2-u_1^{1/n})+\xi_3(1-p-u_2),
\end{align*}
which, after some rewriting is equal to
\begin{align*}
    \max_{\xi_2,\xi_3\geq 0}\{-1+\xi_3+(n^{-\frac{n}{n-1}}-n^{-\frac{1}{n-1}})\xi_2^{\frac{n}{n-1}}~|~y+\xi_3\leq 0, \xi_2\geq \xi_3\}.
\end{align*}
Therefore, we have
\begin{align*}
    &\mathrm{Epi}((-h)^*)\\
    &=\{(y,s):\exists~ \xi_2,\xi_3\geq 0:1-\xi_3+(n^{-\frac{1}{n-1}}-n^{-\frac{n}{n-1}})\xi_2^{\frac{n}{n-1}}\leq s,y+\xi_3\leq 0, \xi_2\geq \xi_3 \}\\
    &=\{(y,s):\exists~ \xi_2,\xi_3,\xi_4\geq 0:1-\xi_3+(n^{-\frac{1}{n-1}}-n^{-\frac{n}{n-1}})\xi_4\leq s,\xi_2\leq \xi_4^{\frac{n-1}{n}}\cdot 1^{1-\frac{n-1}{n}},\\
    &\qquad y+\xi_3\leq 0, \xi_2\geq \xi_3 \}.
\end{align*}
This gives the reformulation of the epigraph of the perspective $\lambda (-h)^*(\frac{-\nu}{\lambda})\leq z$ as
\begin{align*}
  \lambda (-h)^*(\frac{-\nu}{\lambda})\leq z\Leftrightarrow \begin{cases}
      \lambda-\xi_3+(n^{-\frac{1}{n-1}}-n^{-\frac{n}{n-1}})\xi_4\leq z\\
      \xi_2\leq \xi_4^{\frac{n-1}{n}}\cdot \lambda^{1-\frac{n-1}{n}}\\
      -\nu+\xi_3\leq 0\\
      \xi_2\geq \xi_3\\
      \xi_2,\xi_3,\xi_4\geq 0.
  \end{cases} 
\end{align*}
\item For $h(p)=\begin{cases}
    (1-(1-p)^n)^{1/n}&0\leq p< 1\\
    1&p\geq 1
\end{cases}$, we will derive a tractable reformulation of the epigraph $(-h)^*$ using duality. 
We have
\begin{align*}
    \mathrm{Epi}(-h)=\{(p,t)\in \mathbb{R}_{\geq 0}\times \mathbb{R}:\exists~ u_1,u_2\geq 0 :u_1\geq -t,u_2^{1/n}\geq u_1,(1-u_2)^{1/n}\geq 1-p,u_2\leq 1\}.
\end{align*}
Indeed, if $(p,t)\in \mathrm{Epi}(-h)$, then for $0\leq p<1$, we have $-(1-(1-p)^n)^{1/n}\leq t$. 
Choose $u_2=(1-(1-p)^n)$ and $u_1=u_2^{1/n}$ gives the right inclusion. 
If $p\geq 1$, then choose $u_2=1=u_1$, which also gives the right inclusion. 
Conversely, let $(p,t,u_1,u_2)$ satisfy the above constraints. 
If $0\leq p<1$, then, by construction, $(p,t)\in \mathrm{Epi}(-h)$. 
If $p\geq 1$, we have $u_1\leq u_2^{1/n}\leq 1$. 
Hence, $t\geq -u_1\geq -1$, thus $(p,t)\in \mathrm{Epi}(-h)$.

Consider the epigraph of $(-h)^*$. 
Again, we have
\begin{align*}
    \mathrm{Epi}((-h)^*)=\{(y,s): yp-t\leq s, \forall (p,t)\in \mathrm{Epi}(-h)\}.
\end{align*}
Therefore, we have that $(y,s)\in \mathrm{Epi}((-h)^*)$ if and only if the optimization problem
\begin{align*}
    \min_{p,u_1,\geq 0, u_2\in [0,1], t\in \mathbb{R}}\{-yp+t|u_1\geq -t,u_2^{1/n}\geq u_1,(1-u_2)^{1/n}\geq 1-p\}
\end{align*}
is bounded below by $-s$. 
Since this is a convex problem with a point $p=2,u_2=1,u_1=1/2,t=1$ that satisfies Slater's condition, we may apply the duality theorem and obtain that this optimization problem is equal to
\begin{align*}
    &\max_{\xi_1,\xi_2,\xi_3\geq 0}\{\xi_3+\inf_{p,u_1\geq 0, u_2\in [0,1], t\in \mathbb{R}}-(y+\xi_3)p+(1-\xi_1)t+(\xi_2-\xi_1)u_1-\xi_2u_2^{1/n}-\xi_3(1-u_2)^{1/n}\}\\
    &=\max_{\xi_2,\xi_3\geq 0}\{\xi_3+\inf_{u_2\in [0,1]}-\xi_2u_2^{1/n}-\xi_3(1-u_2)^{1/n}|\xi_1=1,y+\xi_3\leq 0,\xi_2\geq \xi_1\}.
\end{align*}
The convex function $-\xi_2u_2^{1/n}-\xi_3(1-u_2)^{1/n}$ has derivative
\[\frac{\mathrm{d}}{\mathrm{d}u_2}=\frac{1}{n}(\xi_3(1-u_2)^{\frac{1-n}{n}}-\xi_2u_2^{\frac{1-n}{n}}),\]
which has a root at 
\begin{align*}
    \frac{\xi_2}{u_2^{\frac{n-1}{n}}}&=\frac{\xi_3}{(1-u_2)^{\frac{n-1}{n}}}\\
    \xi_2(1-u_2)^{\frac{n-1}{n}}&=\xi_3u_2^{\frac{n-1}{n}}\\
    u_2&=\frac{\xi_2^{\frac{n}{n-1}}}{\xi_2^{\frac{n}{n-1}}+\xi_3^{\frac{n}{n-1}}} \in [0,1].
\end{align*}
Since we are examining a convex function, this is where the minimum is attained. 
Hence, we have
\begin{align*}
   \inf_{u_2\in [0,1]}-\xi_2u_2^{1/n}-\xi_3(1-u_2)^{1/n}&=-\xi_2\cdot \left(\frac{\xi_2^{\frac{n}{n-1}}}{\xi_2^{\frac{n}{n-1}}+\xi_3^{\frac{n}{n-1}}}\right)^{1/n}-\xi_3\cdot  \left(\frac{\xi_3^{\frac{n}{n-1}}}{\xi_2^{\frac{n}{n-1}}+\xi_3^{\frac{n}{n-1}}}\right)^{1/n}\\
   &=-\frac{\xi_2^{\frac{n}{n-1}}+\xi_3^{\frac{n}{n-1}}}{(\xi_2^{\frac{n}{n-1}}+\xi_3^{\frac{n}{n-1}})^{1/n}}=-(\xi_2^{\frac{n}{n-1}}+\xi_3^{\frac{n}{n-1}})^{\frac{n-1}{n}}.
\end{align*}
Therefore,
\begin{align*}
    \mathrm{Epi}((-h)^*)&=\{(y,s):\exists~ \xi_2,\xi_3\geq 0:y+\xi_3\leq 0,\xi_2\geq 1, -\xi_3+(\xi_2^{\frac{n}{n-1}}+\xi_3^{\frac{n}{n-1}})^{\frac{n-1}{n}}\leq s\}\\
    &=\{(y,s):\exists~ \xi_2,\xi_3\geq 0:y+\xi_3\leq 0,\xi_2\geq 1, -\xi_3+\|(\xi_2,\xi_3)\|_{\frac{n}{n-1}}\leq s\},
\end{align*}
where $\|.\|_p$ is the p-norm for $p\geq 1$. 
Thus, the reformulation of the epigraph of the perspective $\lambda (-h)^*(\frac{-\nu}{\lambda})\leq z$ is given by
\begin{align*}
    \lambda (-h)^*(\frac{-\nu}{\lambda})\leq z\Leftrightarrow -\nu+\xi_3\leq 0,\ \xi_2\geq \lambda,\, -\xi_3+\|(\xi_2,\xi_3)\|_{\frac{n}{n-1}}\leq z,\ \xi_2,\xi_3\geq 0.
\end{align*}
\end{itemize}

\subsection{Divergence Functions}
\begin{itemize}
    \item For $\phi^*(y)=e^y-1$, we have
\begin{align*}
    \gamma (e^{\frac{s}{\gamma}}-1)&\leq t\Leftrightarrow w-\gamma \leq t,~e^{\frac{s}{\gamma}}\leq \frac{w}{\gamma}\Leftrightarrow w-\gamma \leq t, \gamma \log(\frac{\gamma}{w})+s\leq 0.
\end{align*}
Note that $\gamma \log(\frac{\gamma}{w})$ is the relative entropy.
\item For $\phi^*(y)=2-2\sqrt{1-y}$, $y<1$, we have
\begin{align*}
    \gamma(2-2\sqrt{1-\frac{s}{\gamma}})\leq t, s<\gamma&\Leftrightarrow 2\gamma -2\gamma \sqrt{\frac{\gamma-s}{\gamma}}\leq t,s <\gamma\\
    \Leftrightarrow 2\gamma - 2\gamma \sqrt{\frac{v}{\gamma}}\leq t, v=\gamma -s, s<\gamma
    &\Leftrightarrow 2\gamma -2\sqrt{\gamma v}\leq t, v= \gamma -s, v>0\\
    &\Leftrightarrow 2\gamma -2w \leq t, w^2\leq \gamma v, w\geq 0, v=\gamma -s, v>0\\
    &\Leftrightarrow 2\gamma -2w \leq t, \sqrt{w^2+\frac{1}{4}(\gamma -v)^2}\leq \frac{1}{2}(\gamma+v), w\geq 0, \\
    &\qquad v=\gamma -s, v>0,
\end{align*}
where in the last equivalence we used the equality $xy=\frac{1}{4}((x+y)^2-(x-y)^2)$.
\item For $\phi^*(y)=\frac{y}{1-y}=-1+\frac{1}{1-y}, y<1$, we have
\begin{align*}
    -1+\frac{1}{1-y}\leq t, y<1 &\Leftrightarrow -1+v\leq t, \frac{1}{w}\leq v, w=1-y, w>0\\
    &\Leftrightarrow -1+v\leq t, \sqrt{1+\frac{1}{4}(v-w)^2}\leq \frac{1}{2}(v+w), w=1-y, w>0.
\end{align*}
The epigraph of the perspective function can then be obtained as
\begin{align*}
    \gamma \phi^*(\frac{s}{\gamma})\leq t \Leftrightarrow -\gamma+v\leq t, \sqrt{\gamma^2+\frac{1}{4}(v-w)^2}\leq \frac{1}{2}(v+w), w=\gamma-s, w>0.
\end{align*}
\item For $\phi^*(y)=y+(\theta -1)(\frac{|y|}{\theta})^{\frac{\theta}{\theta-1}}$. 
We show that the epigraph can be represented by a power cone, which is $\mathcal{P}_3^{\alpha,1-\alpha}\triangleq\{x\in \mathbb{R}^3:x_1^\alpha x_2^{1-\alpha}\geq |x_3|, x_1, x_2\geq 0\}$, for $0<\alpha<1$. 
More on tractability of power cones can be found in \citet{Cha09}. 
We have that
\begin{align*}
    y+(\theta -1)(\frac{|y|}{\theta})^{\frac{\theta}{\theta-1}}\leq t&\Leftrightarrow y+(\theta-1)\theta^{\frac{\theta}{1-\theta}}w\leq t,~|y|^{\frac{\theta}{\theta-1}}\leq w\\
    &\Leftrightarrow y+(\theta-1)\theta^{\frac{\theta}{1-\theta}}w\leq t,~|y|\leq w^{\frac{\theta}{\theta-1}}\cdot 1^{1-\frac{\theta}{\theta-1}},
\end{align*}
is conic quadratic representable. 
Hence, so is the epigraph of the perspective
\begin{align*}
    \gamma \phi^*(\frac{s}{\gamma})\leq t\Leftrightarrow s+(\theta-1)\theta^{\frac{\theta}{1-\theta}}w\leq t,~|s|\leq w^{\frac{\theta}{\theta-1}}\cdot \gamma^{1-\frac{\theta}{\theta-1}}.
\end{align*}
\item Similarly, for $\phi^*(y)=\frac{1}{\theta}(1-y(1-\theta))^{\frac{\theta}{\theta -1}}-\frac{1}{\theta}, y<\frac{1}{1-\theta}$, we have that
\begin{align*}
    \frac{1}{\theta}(1-y(1-\theta))^{\frac{\theta}{\theta -1}}-\frac{1}{\theta}\leq t,y<\frac{1}{1-\theta} &\Leftrightarrow |w|\leq (t\theta +1)^{\frac{\theta-1}{\theta}}\cdot 1^{1-\frac{\theta-1}{\theta}}, w = 1-y(1-\theta),\\
    &\qquad y<\frac{1}{1-\theta},
\end{align*}
is conic quadratic. 
Hence, so is the epigraph of the perspective
\begin{align*}
   \gamma \phi^*(\frac{s}{\gamma})\leq t\Leftrightarrow |w|\leq (t\theta +\gamma)^{\frac{\theta-1}{\theta}}\cdot \gamma^{1-\frac{\theta-1}{\theta}}, w = \gamma-s(1-\theta), s<\frac{\gamma}{1-\theta}.
\end{align*}
\end{itemize}

\section{Robust Rank-Dependent Evaluation Using Penalization}\label{app:remark}
This section investigates the comparison between 
\begin{align}\label{rdu_rob}
    \sup_{\mathbf{q}\in \mathcal{D}_\phi(\mathbf{p},r)}\rho_{u,h,\mathbf{q}}(\cdot),
\end{align}
and
\begin{align}\label{rdu-pen}
\tilde{\rho}_{\text{rob}}(\cdot)\triangleq\sup_{\mathbf{q}\in \Delta_m}\rho_{u,h,\mathbf{q}}(\cdot)-\theta I_\phi(\mathbf{q},\mathbf{p}),\quad\theta>0,
\end{align}
and associated optimization problems,
where $\Delta_m$ denotes the set of all $m$-dimensional probability vectors.
\begin{lemma}\label{thm:representation_robust}
Let $h:[0,1]\to [0,1]$ be a concave distortion function. 
Then, the robust risk measure $\tilde{\rho}_{\text{rob}}$ defined in \eqref{rdu-pen} admits the dual representation
\begin{align*}
    \tilde{\rho}_{\text{rob}}(X)=\sup_{\mathbf{\bar{q}}\in \Delta _m}\mathbb{E}_{\mathbf{\bar{q}}}[-u(X)]-c(\mathbf{\bar{q}}),
\end{align*}
with ambiguity index
\begin{align*}
    c(\mathbf{\bar{q}})\triangleq\inf_{\mathbf{q}\in \Delta_m}\theta I_\phi(\mathbf{q},\mathbf{p})+\alpha(\mathbf{\bar{q}},\mathbf{q}),\qquad \alpha(\bar{\mathbf{q}},\mathbf{q})\triangleq\begin{cases}
    0,&~\text{if}~\mathbf{\bar{q}}\in M_h(\mathbf{q});\\
    \infty,&~\text{else};
    \end{cases}
\end{align*}
where $M_h(\mathbf{q})$ is as defined in \eqref{defMhq}.
\end{lemma}
\begin{proof}
By \citet{Denneberg}, Proposition~10.3, we have the dual representation
\begin{align*}
 \rho_{u,h,\mathbf{q}}(X)=\sup_{\mathbf{\bar{q}}\in \Delta_m} \mathbb{E}_{\mathbf{\bar{q}}}[-u(X)]-\alpha(\bar{\mathbf{q}},\mathbf{q}).  
\end{align*}
Hence, we derive
\begin{align*}
  \tilde{\rho}_{\text{rob}}(X)&=\sup_{\mathbf{q}\in \Delta_m}\rho_{u,h,\mathbf{q}}(X)-\theta \sum^m_{i=1}p_i\phi\left(\frac{q_i}{p_i}\right)\\
  &=\sup_{\mathbf{q}\in \Delta _m}\sup_{\mathbf{\bar{q}}\in \Delta_m} \mathbb{E}_{\mathbf{\bar{q}}}[-u(X)]-\alpha(\bar{\mathbf{q}},\mathbf{q})-\theta \sum^m_{i=1}p_i\phi\left(\frac{q_i}{p_i}\right)\\
  &=\sup_{\mathbf{\bar{q}}\in \Delta_m}\sup_{\mathbf{q}\in \Delta_m} \mathbb{E}_{\mathbf{\bar{q}}}[-u(X)]-\alpha(\bar{\mathbf{q}},\mathbf{q})-\theta \sum^m_{i=1}p_i\phi\left(\frac{q_i}{p_i}\right)\\
  &=\sup_{\mathbf{\bar{q}}\in \Delta_m}\mathbb{E}_{\mathbf{\bar{q}}}[-u(X)]-\inf_{\mathbf{q}\in \Delta_m} \left(\alpha(\bar{\mathbf{q}},\mathbf{q})+\theta \sum^m_{i=1}p_i\phi\left(\frac{q_i}{p_i}\right)\right)\\
  &=\sup_{\mathbf{\bar{q}}\in \Delta_m}\mathbb{E}_{\mathbf{\bar{q}}}[-u(X)]-c(\mathbf{\bar{q}}).
\end{align*}
 \end{proof}
We next show that a decision-maker who minimizes \eqref{rdu_rob} obtains the same minimizer as when minimizing \eqref{rdu-pen}, for a specific $\theta$.
\begin{proposition}\label{prop: connectieORvsPen}
Assume the existence of a minimizer $\mathbf{a}^*\in \mathrm{argmin}_{\mathbf{a}\in \mathcal{A}}\sup_{\mathbf{q}\in \mathcal{D}_\phi(\mathbf{p},r)}\rho_{u,h,\mathbf{q}}(f(\mathbf{a},\mathbf{X}))$. 
Then for each $r>0$, there exists a $\theta^*$, such that
\begin{align*}
\mathbf{a}^*\in \mathrm{argmin}_{\mathbf{a}\in \mathcal{A}}\sup_{\mathbf{q}\in \Delta_m}\rho_{u,h,\mathbf{q}}(f(\mathbf{a},\mathbf{X}))-\theta^* I_\phi(\mathbf{q},\mathbf{p}).    
\end{align*}
\end{proposition}
\begin{proof}
As we have shown in Lemma~\ref{thm:representation_robust}, $\min_{\mathbf{a}\in \mathcal{A}}\sup_{\mathbf{q}\in \Delta_m}\rho_{u,h,\mathbf{q}}(f(\mathbf{a},\mathbf{X}))-\theta I_\phi(\mathbf{q},\mathbf{p})$ is equivalent to
\begin{align}\label{robust_penalty}
    \min_{\mathbf{a}\in \mathcal{A}}\left\{\sup_{\substack{\mathbf{q}\geq \mathbf{0}\\\mathbf{q}^T\mathbf{1}=1}}\sup_{\mathbf{\bar{q}}\in M_h(\mathbf{q})}-\sum^m_{i=1}\bar{q}_iu(f(\mathbf{a},\mathbf{x}_i))-\theta \sum^m_{i=1}p_i\phi\left(\frac{q_i}{p_i}\right)\right\},
\end{align}
for any fixed constant $\theta>0$. 
Let $\mathbf{a}^*$ be a solution of $\min_{\mathbf{a}\in \mathcal{A}}\sup_{\mathbf{q}\in \mathcal{D}_\phi(\mathbf{p},r)}\rho_{u,h,\mathbf{q}}(f(\mathbf{a},\mathbf{X}))$. 
Then, strong duality implies the existence of a $\gamma^*$ (depending on $\mathbf{a}^*$), such that
\begin{align}\label{sup_with_gamma}
\begin{split}
    &\sup_{\mathbf{q}\in \mathcal{D}_{\phi}(\mathbf{p},r)}\sup_{\mathbf{\bar{q}}\in M_h(\mathbf{q})}-\sum^m_{i=1}\bar{q}_iu(f(\mathbf{a}^*,\mathbf{x}_i))\\
    &= \gamma^* r +   \sup_{\substack{\mathbf{q}\geq \mathbf{0}\\\mathbf{q}^T\mathbf{1}=1}}\sup_{\mathbf{\bar{q}}\in M_h(\mathbf{q})}-\sum^m_{i=1}\bar{q}_iu(f(\mathbf{a}^*,\mathbf{x}_i))-\gamma^*\sum^m_{i=1}p_i\phi\left(\frac{q_i}{p_i}\right).
    \end{split}
\end{align}
Since $\mathbf{a}^*$ is a minimizer, we also have
\begin{align*}
   &\sup_{\mathbf{q}\in \mathcal{D}_{\phi}(\mathbf{p},r)}\sup_{\mathbf{\bar{q}}\in M_h(\mathbf{q})}-\sum^m_{i=1}\bar{q}_iu(f(\mathbf{a}^*,\mathbf{x}_i))\\
   &\leq \sup_{\mathbf{q}\in \mathcal{D}_{\phi}(\mathbf{p},r)}\sup_{\mathbf{\bar{q}}\in M_h(\mathbf{q})}-\sum^m_{i=1}\bar{q}_iu(f(\mathbf{a},\mathbf{x}_i))\\
   &=\inf_{\gamma \geq 0}\left\{\gamma r +   \sup_{\substack{\mathbf{q}\geq \mathbf{0}\\\mathbf{q}^T\mathbf{1}=1}}\sup_{\mathbf{\bar{q}}\in M_h(\mathbf{q})}-\sum^m_{i=1}\bar{q}_iu(f(\mathbf{a},\mathbf{x}_i))-\gamma\sum^m_{i=1}p_i\phi\left(\frac{q_i}{p_i}\right)\right\}\\
   &\leq \gamma^* r +   \sup_{\substack{\mathbf{q}\geq \mathbf{0}\\\mathbf{q}^T\mathbf{1}=1}}\sup_{\mathbf{\bar{q}}\in M_h(\mathbf{q})}-\sum^m_{i=1}\bar{q}_iu(f(\mathbf{a},\mathbf{x}_i))-\gamma^*\sum^m_{i=1}p_i\phi\left(\frac{q_i}{p_i}\right).
\end{align*}
Bringing $\gamma^*r$ to the left-hand side of the inequality, we obtain from \eqref{sup_with_gamma} that
\begin{align*}
    &\sup_{\substack{\mathbf{q}\geq \mathbf{0}\\\mathbf{q}^T\mathbf{1}=1}}\sup_{\mathbf{\bar{q}}\in M_h(\mathbf{q})}-\sum^m_{i=1}\bar{q}_iu(f(\mathbf{a}^*,\mathbf{x}_i))-\gamma^*\sum^m_{i=1}p_i\phi\left(\frac{q_i}{p_i}\right)\\
    &\leq \sup_{\substack{\mathbf{q}\geq \mathbf{0}\\\mathbf{q}^T\mathbf{1}=1}}\sup_{\mathbf{\bar{q}}\in M_h(\mathbf{q})}-\sum^m_{i=1}\bar{q}_iu(f(\mathbf{a},\mathbf{x}_i))-\gamma^*\sum^m_{i=1}p_i\phi\left(\frac{q_i}{p_i}\right),
\end{align*}
for any $\mathbf{a}\in \mathcal{A}$.
Hence, $\mathbf{a}^*$ is also a solution of \eqref{robust_penalty}, where $\theta^*=\gamma^*$.
\end{proof}

\section{A Visualization of the Various Shapes of the Uncertainty Set \texorpdfstring{$\mathcal{U}_{\phi,h}(\mathbf{p})$}{TEXT}}\label{App:Shape_UC}

\begin{figure}[H]
    \centering

   {%
   
        \begin{minipage}{0.48\textwidth} 
            \centering
            \includegraphics[width=\textwidth]{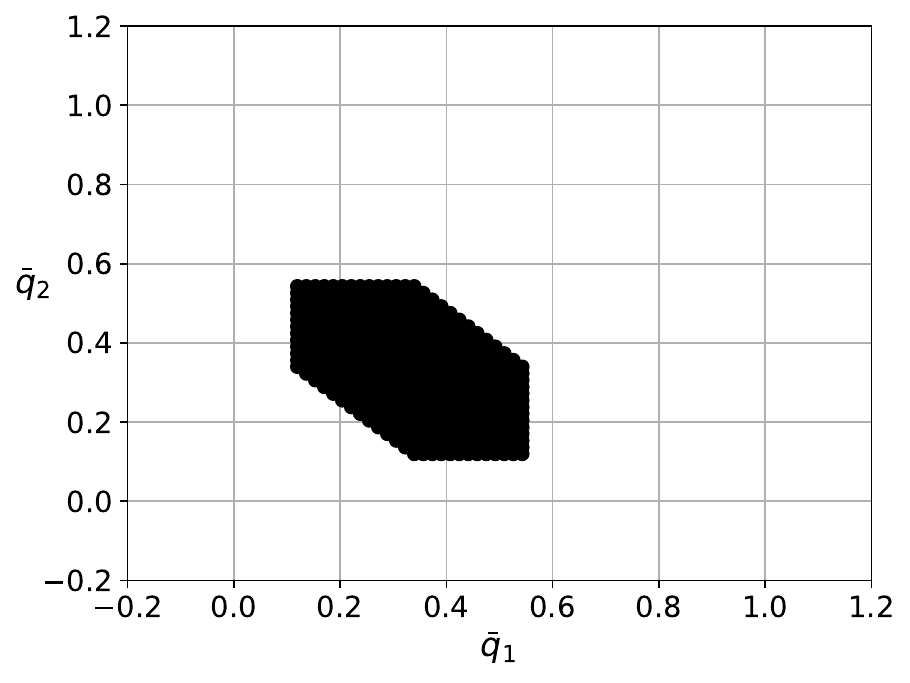} 
            \footnotesize (a). $n = \infty$
        \end{minipage}
    }%
    {%
        \begin{minipage}{0.48\textwidth}
            \centering
            \includegraphics[width=\textwidth]{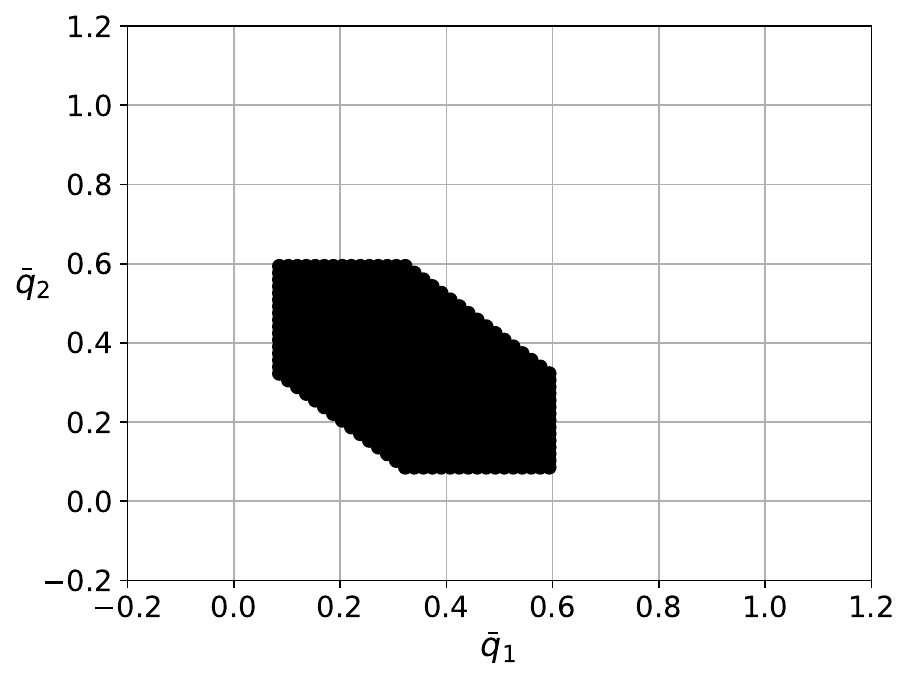} 
            \footnotesize (b). $n = 500$
        \end{minipage}
    }%

    {%
        \begin{minipage}{0.48\textwidth}
            \centering
            \includegraphics[width=\textwidth]{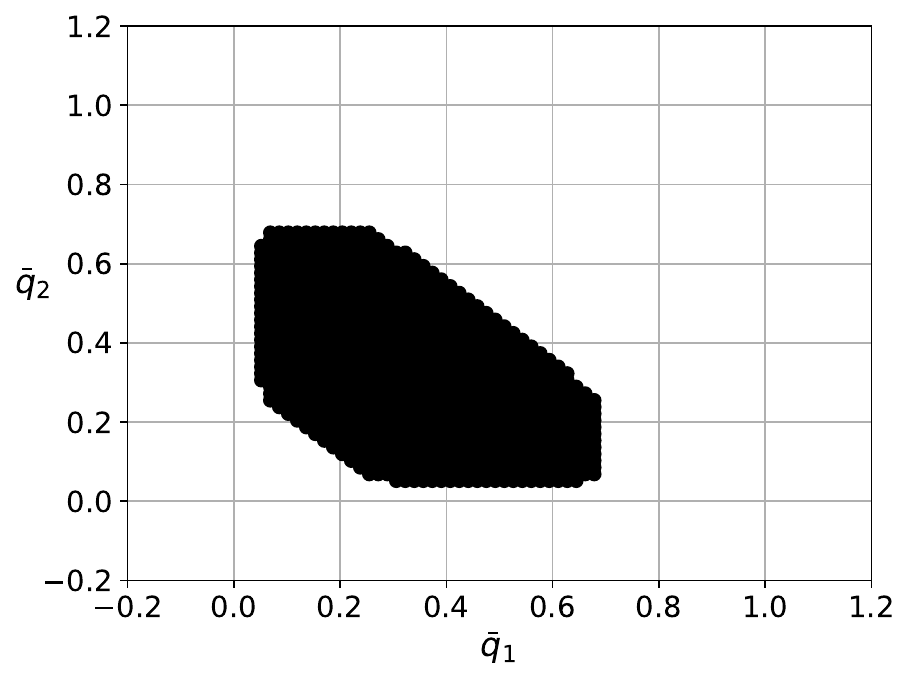} 
            \footnotesize (c). $n = 100$
        \end{minipage}
    }%
    {%
        \begin{minipage}{0.48\textwidth}
            \centering
            \includegraphics[width=\textwidth]{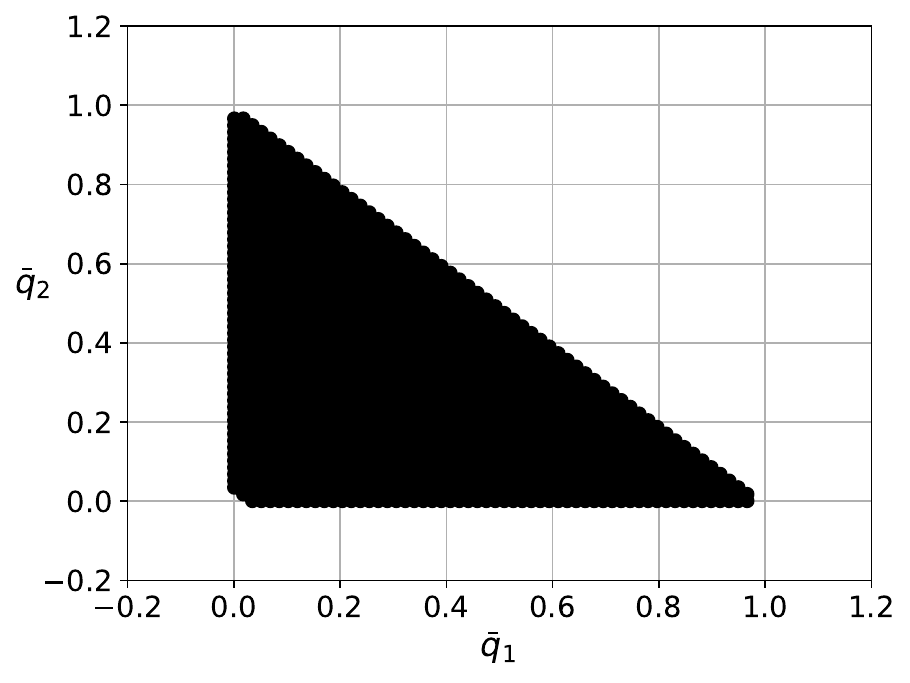} 
            \footnotesize (d). $n = 5$
        \end{minipage}
    }
    \caption{{\small Projections of the uncertainty set $\mathcal{U}_{\phi,h}(\mathbf{p})$ in \eqref{uncertaintyset} on the coordinates $(\bar{q}_1,\bar{q}_2)$, for $\mathbf{p}=(1/3,1/3,1/3)$ and $r=\frac{1}{n}\chi^2_{0.95,2}$, plotted for a range of values of the sample size $n$. 
   We choose the modified chi-squared divergence function $\phi(t)=(t-1)^2$ and the second dual moment distortion function $h(p)=1-(1-p)^2$. 
   As $n$ approaches $0$, we observe that the uncertainty set grows and the projection eventually approaches the entire probability simplex in $\mathbb{R}^2$, the case in which the decision-maker is completely ambiguous w.r.t.\ $\mathbf{p}$.}}\label{fig: formUset}

\end{figure}

\section{The Optimistic Dual Counterpart}\label{app: OPDual}
\citet{PrimalWorstDualBest} and \citet{Gorissen2014} 
have shown that a robust minimization problem with a compact convex uncertainty set can also be reformulated by considering the optimistic dual counterpart, obtained by maximizing the dual of its uncertain problem over all uncertain variables. 
In this section, we derive and reformulate the optimistic dual counterpart of the robust problem
\begin{align}\label{prob1}
    \min_{\mathbf{a}\in \mathcal{A},c\in \mathcal{C}}\left\{\mathbf{a}^T\mathbf{d} + \zeta\cdot c~\middle|~\sup_{(\mathbf{q},\bar{\mathbf{q}})\in \mathcal{U}_{\phi,h}(\mathbf{p})}-\sum^m_{i=1}\bar{q}_iu(f(\mathbf{a},\mathbf{x}_i))\leq c\right\},
\end{align}
where
\begin{align*}
    \mathcal{U}_{\phi,h}(\mathbf{p})=\left\{(\mathbf{q},\bar{\mathbf{q}})\in \mathbb{R}^{2m} \;
    \begin{tabular}{|l}
      $\sum^m_{i=1}q_i=\sum^m_{i=1}\bar{q}_i=1$\\
      $\sum^m_{i=1}p_i\phi\left(\frac{q_i}{p_i}\right)\leq r$\\
      $\sum_{i\in J}\bar{q}_{i}\leq h\left(\sum_{i\in J} q_{i}\right),~\forall J\subset [m]$\\
      $q_i,\bar{q}_i\geq 0,~ \forall i\in [m]$
\end{tabular}
\right\}.
\end{align*}
Here, we have either $(\mathcal{C}=\mathbb{R}, \mathbf{d}=\mathbf{0},\zeta=1)$ or $(\mathcal{C}=\{c_0\}, \mathbf{d}\neq \mathbf{0}, \zeta=0$), where $\mathbf{d}\in \mathbb{R}^{n_a}, c_0\in \mathbb{R}$.
Note that $\mathcal{U}_{\phi,h}(\mathbf{p})$ is a compact set since it is bounded and closed due to Assumption~\ref{assump:lower_semi} that the functions $\phi,-h$ are lower-semicontinuous convex functions.

To derive the optimistic dual counterpart of \eqref{prob1}, we first consider its uncertain problem, which for a given $(\mathbf{q},\mathbf{\bar{q}})\in \mathcal{U}_{\phi,h}(\mathbf{p})$ is defined as
\begin{align}\label{uncertainpb}
\begin{split}
     &\min_{\mathbf{a}\in \mathcal{A},c\in \mathcal{C}}\left\{\mathbf{a}^T\mathbf{d} + \zeta\cdot c~\middle|~-\sum^m_{i=1}\bar{q}_iu(f(\mathbf{a},\mathbf{x}_i))\leq c\right\}\\
     &=\min_{\mathbf{a}\in \mathbb{R}^I, c\in \mathcal{C}}\left\{\mathbf{a}^T\mathbf{d} + \zeta\cdot c~\middle|~ -\sum^m_{i=1}\bar{q}_iu(f(\mathbf{a},\mathbf{x}_i))\leq c, f_i(\mathbf{a})\leq 0, i\in [L]\right\}.
\end{split}
\end{align}
Here, recall that $\mathcal{A}$ is a set represented by convex inequalities, hence we can express it as $\mathcal{A}=\{\mathbf{a}\in \mathbb{R}^I|f_i(\mathbf{a})\leq 0, i=1,\ldots , L\}$, for some convex functions $f_i$'s. 
Assume that \eqref{uncertainpb} satisfies Slater's condition and is bounded from below. 
Then, the dual of \eqref{uncertainpb} is given by
\begin{align}
    \max_{y_0,\ldots,y_L\geq 0}\inf_{\mathbf{a}\in \mathbb{R}^I, c\in \mathcal{C}}\left\{\mathbf{a}^T\mathbf{d} + \zeta\cdot c-y_0\sum^m_{i=1}\bar{q}_iu(f(\mathbf{a},\mathbf{x}_i))-y_0c+\sum^L_{k=1}y_kf_k(\mathbf{a})\right\}.
\end{align}
The optimistic dual counterpart is defined by maximizing the dual over all uncertain variables in the uncertainty set, i.e.,
\begin{align}\label{OD}
\tag{OD}
    \sup_{(\mathbf{q},\mathbf{\bar{q}})\in \mathcal{U}_{\phi,h}(\mathbf{p})} \max_{y_0,\ldots,y_L\geq 0}\inf_{\mathbf{a}\in \mathbb{R}^I, c\in \mathcal{C}}\left\{\mathbf{a}^T\mathbf{d} + \zeta\cdot c-y_0\sum^m_{i=1}\bar{q}_iu(f(\mathbf{a},\mathbf{x}_i))-y_0c+\sum^L_{k=1}y_kf_k(\mathbf{a})\right\}.
\end{align}
The following theorem states a reformulation of \eqref{OD}.
\begin{mytheorem}\label{thm:optdual}
The optimistic dual counterpart \eqref{OD} is equivalent to the following concave problem:
\begin{align}\label{optdual}
\begin{split}
       \begin{split}
    \sup_{\substack{\mathbf{z},\mathbf{\bar{z}}\in \mathbb{R}^m_{\geq 0}\\ y_0,\ldots,y_L\geq 0\\\boldsymbol{\lambda}_1,\ldots ,\boldsymbol{\lambda}_m, \boldsymbol{\eta}_1,\ldots,\boldsymbol{\eta}_L\in \mathbb{R}^I}}&-\sum^m_{i=1}\bar{z}_i(-u\circ f)^*\left(\frac{\mathbf{\boldsymbol{\lambda}}_i}{y_0\bar{q}_i},\mathbf{x}_i\right)-\sum^L_{k=1}y_k(f_k)^*\left(\frac{\boldsymbol{\eta}_k}{y_k}\right)-y_0c_0\mathbbm{1}_{\{\zeta=0\}}\\
        \mathrm{subject}~ \mathrm{to}&~\sum^m_{i=1}\boldsymbol{\lambda}_i=-\sum^L_{k=1}\boldsymbol{\eta}_k-\mathbf{d}\\
        &\sum^m_{i=1}z_i=\sum^m_{i=1}\bar{z}_i=y_0\\
        &\sum_{i\in I}\bar{z}_i-y_0h\left(\frac{\sum_{i\in I}z_i}{y_0}\right)\leq 0, \forall I\in 2^N_-,\\
    &\sum^m_{i=1}y_0p_i\phi\left(\frac{z_i}{y_0p_i}\right)-y_0r\leq 0,
    \end{split}
\end{split}
\end{align}
where $\mathbf{z}=y_0\mathbf{q}$ and $\mathbf{\bar{z}}=y_0\mathbf{\bar{q}}$. 
If Slater's condition holds for problem \eqref{optdual}, then \eqref{prob1} is equal to \eqref{optdual}. 
Moreover, the KKT-vector of \eqref{optdual} corresponding to the dual equality constraints $\sum^m_{i=1}\boldsymbol{\lambda}_i=-\sum^L_{k=1}\boldsymbol{\eta}_k-\mathbf{d}$ gives the optimal solution of \eqref{prob1}.
\end{mytheorem}
\begin{proof}
We have that
\begin{align*}
&\inf_{\mathbf{a}\in \mathbb{R}^I, c\in \mathcal{C}}\left\{\mathbf{a}^T\mathbf{d} + \zeta\cdot c-y_0\sum^m_{i=1}\bar{q}_iu(f(\mathbf{a},\mathbf{x}_i))-y_0c+\sum^L_{k=1}y_kf_k(\mathbf{a})\right\}\\
&=\inf_{\mathbf{a}\in \mathbb{R}^I}\left\{\mathbf{a}^T\mathbf{d}-y_0\sum^m_{i=1}\bar{q}_iu(f(\mathbf{a},\mathbf{x}_i))+\sum^L_{k=1}y_kf_k(\mathbf{a})\right\}+\inf_{c\in \mathcal{C}}(\zeta-y_0)c\\
&=-\sup_{\mathbf{a}\in \mathbb{R}^I}\left\{-\mathbf{a}^T\mathbf{d}+y_0\sum^m_{i=1}\bar{q}_iu(f(\mathbf{a},\mathbf{x}_i))-\sum^L_{k=1}y_kf_k(\mathbf{a})\right\}+\inf_{c\in \mathcal{C}}(\zeta-y_0)c.
\end{align*}
Note that $\mathcal{C}$ is either $\mathbb{R}$ (if $\zeta =1$) or $\{c_0\}$ for some $c_0\in \mathbb{R}$ (if $\zeta = 0$). 
Hence, we have that
\begin{align*}
\inf_{c\in \mathcal{C}}(\zeta-y_0)c=\begin{cases}
-y_0c_0&~\text{if}~\mathcal{C}=c_0\\
0&~\text{if}~\mathcal{C}=\mathbb{R},~y_0=1\\
-\infty&~\text{else}.
\end{cases}    
\end{align*}
We examine the supremum term. 
We have
\begin{align*}
 &\sup_{\mathbf{a}\in \mathbb{R}^I}\left\{-\mathbf{a}^T\mathbf{d}+y_0\sum^m_{i=1}\bar{q}_iu(f(\mathbf{a},\mathbf{x}_i))-\sum^L_{k=1}y_kf_k(\mathbf{a})\right\}\\
 &=\sup_{\substack{\mathbf{a},\mathbf{w}_1,\ldots, \mathbf{w}_m,\\ \mathbf{v}_1,\ldots, \mathbf{v}_L\in \mathbb{R}^I}}\left\{-\mathbf{a}^T\mathbf{d}+y_0\sum^m_{i=1}\bar{q}_iu(f(\mathbf{w}_i,\mathbf{x}_i))-\sum^L_{k=1}y_kf_k(\mathbf{v}_k)~ \middle|~\mathbf{w}_i=\mathbf{a},~\mathbf{v}_k=\mathbf{a}, i\in [m], k\in [L]\right\}\\
 &=\inf_{\substack{\boldsymbol{\lambda}_1,\ldots \boldsymbol{\lambda}_m,\\ \boldsymbol{\eta}_1,\ldots,\boldsymbol{\eta}_L\in \mathbb{R}^I}}\sup_{\substack{\mathbf{a},\mathbf{w}_1,\ldots, \mathbf{w}_m,\\ \mathbf{v}_1,\ldots, \mathbf{v}_L\in \mathbb{R}^I}}\left\{-\mathbf{a}^T\mathbf{d}+\sum^m_{i=1}\mathbf{\boldsymbol{\lambda}}_i^T(\mathbf{w}_i-\mathbf{a})+\sum^L_{k=1}\mathbf{\boldsymbol{\eta}}_k^T(\mathbf{v}_k-\mathbf{a})+y_0\sum^m_{i=1}\bar{q}_iu(f(\mathbf{w}_i,\mathbf{x}_i))\right.\\
 &\qquad \left.-\sum^L_{k=1}y_kf_k(\mathbf{v}_k)\right\}\\
 &=\inf_{\substack{\boldsymbol{\lambda}_1,\ldots \boldsymbol{\lambda}_m,\\ \boldsymbol{\eta}_1,\ldots,\boldsymbol{\eta}_L\in \mathbb{R}^I}}\sup_{\substack{\mathbf{a},\mathbf{w}_1,\ldots, \mathbf{w}_m,\\ \mathbf{v}_1,\ldots, \mathbf{v}_L\in \mathbb{R}^I}}\left\{-\left(\sum^m_{i=1}\boldsymbol{\lambda}_i+\sum^L_{k=1}\boldsymbol{\eta}_k+\mathbf{d}\right)^T\mathbf{a}+\sum^m_{i=1}\boldsymbol{\lambda}_i^T\mathbf{w}_i-y_0\bar{q}_i(-u)(f(\mathbf{w}_i,\mathbf{x}_i))\right.\\
 &\qquad\left.+\sum^L_{k=1}\boldsymbol{\eta}_k^T\mathbf{v}_k-y_kf_k(\mathbf{v}_k)\right\}\\
 &=\inf_{\substack{\boldsymbol{\lambda}_1,\ldots \boldsymbol{\lambda}_m,\\ \boldsymbol{\eta}_1,\ldots,\boldsymbol{\eta}_L\in \mathbb{R}^I}}\left\{\sum^m_{i=1}y_0\bar{q}_i(-u\circ f)^*\left(\frac{\mathbf{\boldsymbol{\lambda}}_i}{y_0\bar{q}_i},\mathbf{x}_i\right)+\sum^L_{k=1}y_k(f_k)^*\left(\frac{\boldsymbol{\eta}_k}{y_k}\right)~\middle |~\sum^m_{i=1}\boldsymbol{\lambda}_i=-\sum^L_{k=1}\boldsymbol{\eta}_k-\mathbf{d}\right\},
\end{align*}
where $(-u\circ f)$ denotes the composition of the two functions. 
A change of variables $\mathbf{z}=y_0\mathbf{q}$ and $\mathbf{\bar{z}}=y_0\mathbf{\bar{q}}$ and a multiplication of the constraints in $\mathcal{U}_{\phi,h}(\mathbf{p})$ by $y_0$ yields the desired statement.
\end{proof}

\section{Optimization of Rank-Dependent Models in the Constraint}\label{app:con}
We describe in detail the adaptation of Algorithm~\ref{cutting_plane_algo} to \eqref{P-constraint} in Algorithm~\ref{Algo: cutting_plane_constraint}.
\begin{algorithm}[htp]
\caption{\textbf{Cutting-Plane Method with RDU Constraint}}\label{Algo: cutting_plane_constraint}
\begin{algorithmic}[1]
\State Start with $\mathcal{U}_1=\{(\mathbf{p},\mathbf{p})\}$. 
Fix a tolerance parameter $\epsilon_{\mathrm{tol}}>0$.
\State\label{step2_constr}At the $j$-th iteration, solve the following problem with the uncertainty set $\mathcal{U}_j$:
\begin{align}\label{cut_pb_j_constr}
    \begin{split}
         &\min_{\mathbf{a}\in \mathcal{A}}~\left\{g(\mathbf{a})~\middle |~ \sup_{(\mathbf{q},\bar{\mathbf{q}}) \in\mathcal{U}_{j}}\sum^m_{i=1}-\bar{q}_iu(f(\mathbf{a},\mathbf{x}_i))\leq c\right\}.
    \end{split}
\end{align}
\State\label{step3_constr}Let $\mathbf{a}_j$ be the optimal solution of \eqref{cutting_plane_pb_j}. 
Determine the ranking: 
\begin{align*}
    -u(f(\mathbf{a}_j,\mathbf{x}_{(1)}))\leq \ldots \leq -u(f(\mathbf{a}_j,\mathbf{x}_{(m)})).
\end{align*}
Then, solve the optimization problem \eqref{ranked_opt_robustcheck}, which gives an optimal objective value $v_j$ and a solution $(\mathbf{q}^*_j,\bar{\mathbf{q}}^*_j)$.
\State\label{step4_constr}If $v_j- c\leq \epsilon_{\mathrm{tol}}$, then the solution is accepted and the process is terminated.
\State\label{step5_constr}If not, set $\mathcal{U}_{j+1} = \mathcal{U}_j \cup \{(\mathbf{q}^*_j,\mathbf{\bar{q}}^*_j)\}$ and repeat steps~\ref{step2_constr}--\ref{step5_constr}.
\end{algorithmic}
\end{algorithm}
\section{SOS2-Constraints Formulation}\label{app:SOS2}
Let $u:[l_0,u_0]\to \mathbb{R}$ be a non-decreasing, piecewise-linear utility function defined on some interval $[l_0,u_0]$ that contains the image set $\{\mathbf{a}^T\mathbf{x}_i~|~\mathbf{a}\in \mathcal{A}, i\in [m]\}$. Let $\{t_j, u(t_j)\}^K_{j=1}$ be the support points of $u$, where $l_0=t_1<\cdots <t_K=u_0$. Then, the constraints in \eqref{RC: invS} can be formulated using the following bilinear and SOS2 constraints:
\begin{align}\label{RC: invS_SOS2}
    \begin{cases}
        \beta\cdot h(p^0)+\sum^{K_1}_{k=1}\nu_kb^{(1)}_k+\sum^m_{i=1}\sum^{K_1}_{k=1}\lambda_{ik}l^{(1)}_kp_i-\sum^m_{i=1}\bar{q}_iz_i\leq c\\
    -z_i-\beta-\sum^{K_1}_{k=1}\lambda_{ik}\leq 0,~\forall i\in [m]\\
        \lambda_{ik}\leq \nu_k,~\forall i\in [m],~\forall k \in [K_1]\\
        \bar{q}_i\leq l^{(2)}_kp_i+t_{ik},~\forall i\in [m],~\forall k\in [K_2]\\
        \sum^{m}_{i=1}t_{ik}\leq b^{(2)}_k,~\forall k\in [K_2] \\ 
        \sum^m_{i=1}\bar{q}_i=\bar{h}(1-p^0)\\
        \mathbf{a}^T\mathbf{x}_i=\sum^K_{j=1}\tilde{\lambda}_{ij}t_j,~\forall i\in [m]\\
        z_i=\sum^K_{j=1}\tilde{\lambda}_{ij}u(t_j),~\forall i\in [m]\\
        \sum^K_{j=1}\tilde{\lambda}_{ij}=1,~\forall i\in [m]\\
        \tilde{\lambda}_{ij}\geq 0, ~\mathrm{SOS}2,~\forall i\in [m],~\forall j\in [K].
    \end{cases}
\end{align}

\section{The Hit-and-Run Algorithm}\label{App:HitandRun}
In this appendix, we explain how the Hit-and-Run algorithm is applied to generate a uniform sample on the $\phi$-divergence set. Let $S$ be an open subset in $\mathbb{R}^d$. 
The general procedure of the Hit-and-Run algorithm is fairly simple. 
Start with an interior point $x_0\in S$. 
Choose uniformly a random direction $u\in \partial \mathcal{D}$ where $\partial \mathcal{D}\triangleq\{x\in \mathbb{R}^d: \|x\|_2=1\}$. 
Draw uniformly a scalar $\lambda\in \{\lambda\in \mathbb{R}: x_0+\lambda u \in S\}$ and then update $x_0$ with $x_0+\lambda u$. 
It is shown by \citet{HitRun} that for a convex set $S$, this sampling process will converge to a uniform sample on $S$.

To apply the Hit-and-Run algorithm to the $\phi$-divergence uncertainty set, we first re-parametrize the set. 
Recall the definition of the $\phi$-divergence set \eqref{defDivset}.
Using the equality $\mathbf{q}^T\mathbf{1}=1$, it is clear that we can parametrize the $\phi$-divergence set with the following set:
\begin{align}\label{phi_repar}
\begin{split}
    &\tilde{\mathcal{D}}_\phi(\mathbf{q}|\mathbf{p},r)\\
&\qquad \triangleq\{\mathbf{q}\in \mathbb{R}^{m-1}|\mathbf{q}\geq \mathbf{0},1-\sum^{m-1}_{i=1}q_i\geq 0, \sum^{m-1}_{i=1}p_i\phi\left(\frac{q_i}{p_i}\right)+p_m\phi\left(\frac{1-\sum^{m-1}_{i=1}q_i}{p_m}\right)\leq r\}.
\end{split}
\end{align}
Since this is a set with convex inequalities, and we assume Slater's condition, its interior is an open convex set. 
Hence, we can apply the Hit-and-Run algorithm to obtain a uniform sample of the interior of $\tilde{\mathcal{D}}_\phi(\mathbf{q}|\mathbf{p},r)$ and thus also the interior of $\mathcal{D}_\phi(\mathbf{p},r)$ due to their one-to-one correspondence. 

The precise procedure is the following.
\begin{itemize}
    \item We start with an interior point $\mathbf{q}_k\in \mathrm{int}\tilde{\mathcal{D}}_\phi(\mathbf{q}|\mathbf{p},r)$, where we set $\mathbf{q}_0=(p_1,\ldots, p_{m-1})$. 
    Then, we draw a uniform element $\mathbf{u}\in \partial \mathcal{D}$ of dimension $d=m-1$. 
    This can be done by drawing $d$ standard normal i.i.d.~samples $\mathbf{X}=(X_1,\ldots, X_{d})$ and setting $\mathbf{u}=\frac{\mathbf{X}}{\|\mathbf{X}\|_2}$, where $\|.\|_2$ is the Euclidean 2-norm. 
    Then, $\mathbf{u}$ follows a uniform distribution on $\partial \mathcal{D}$. 
    \item Given $\mathbf{u}$ and $\mathbf{q}_k$, 
    we determine the set $\Lambda\triangleq\{\lambda\in \mathbb{R}:\mathbf{q}_k+\lambda\mathbf{u}\in \tilde{\mathcal{D}}_\phi(\mathbf{q}|\mathbf{p},r)\}$. 
    Note that since $\tilde{\mathcal{D}}_\phi(\mathbf{q}|\mathbf{p},r)$ is convex, bounded and closed, there exist $\lambda_{\min},\lambda_{\max}$ such that $[\lambda_{\min},\lambda_{\max}]=\Lambda$.
    We may find both values using a bisection search. 
    Indeed, since $\mathbf{q}_k\in \mathrm{int}\tilde{\mathcal{D}}_\phi(\mathbf{q}|\mathbf{p},r)$, there exists a sufficiently small $\lambda_0>0$ such that $\lambda_0\in \Lambda$. 
    Take a sufficiently large $\lambda_1>\lambda_0$ such that $\lambda_1\notin \Lambda$, which also exists since $\Lambda$ is bounded. 
    Then, perform the bisection search on $[\lambda_0,\lambda_1]$ to find $\lambda_{\max}$. 
    Do the same for $\lambda_{\min}$. 
    \item We draw a random $\lambda\in [\lambda_{\min},\lambda_{\max}]$ and update $\mathbf{q}_{k+1}=\mathbf{q}_k+\lambda \mathbf{u}$. 
    We then repeat the process.
\end{itemize}

\end{document}